\documentclass[a4paper,reqno]{amsart}

\ifdefined\kanjiskip
  \PassOptionsToPackage{dvipdfmx}{graphicx}
  \PassOptionsToPackage{dvipdfmx}{xcolor}
  \PassOptionsToPackage{dvipdfmx}{hyperref}
\fi

\usepackage{amsmath}
\usepackage{amssymb}
\usepackage{amscd}
\usepackage{amsthm}
\usepackage{type1cm}
\usepackage{tcolorbox}
\usepackage{mathrsfs}
\usepackage{adjustbox}

\usepackage{tikz}
\usepackage{tikz-3dplot} 
\usetikzlibrary{math} 
\usetikzlibrary{arrows.meta} 

\usepackage[colorlinks,citecolor=darkgreen,linkcolor=red]{hyperref}
\definecolor{darkgreen}{rgb}{0.0, 0.6, 0.0}

\usepackage[all]{xy}
\input xy
\xyoption{all}
\tcbuselibrary{breakable, skins, theorems}

\usepackage{tikz-cd}

\def\A{\mathcal{A}}
\def\C{\mathcal{C}}
\def\D{\mathcal{D}}
\def\E{\mathcal{E}}
\def\H{\mathcal{H}}
\def\I{\mathcal{I}}
\def\J{\mathcal{J}}

\def\O{\mathcal{O}}
\def\P{\mathcal{P}}
\def\S{\mathcal{S}}
\def\T{\mathcal{T}}

\def\X{\mathcal{X}}
\def\Y{\mathcal{Y}}
\def\Z{\mathcal{Z}}

\DeclareMathOperator{\md}{\mathsf{mod}}
\renewcommand{\mod}{\md}
\DeclareMathOperator{\proj}{\mathsf{proj}}
\DeclareMathOperator{\fl}{\mathsf{fl}}

\DeclareMathOperator{\add}{\mathsf{add}}

\DeclareMathOperator{\Coh}{\mathsf{Coh}}
\DeclareMathOperator{\per}{\mathsf{per}}
\DeclareMathOperator{\thick}{\mathsf{thick}}
\DeclareMathOperator{\pvd}{\mathsf{pvd}}
\DeclareMathOperator{\Loc}{\mathsf{Loc}}

\DeclareMathOperator{\simp}{sim}

\DeclareMathOperator{\Hom}{Hom}
\DeclareMathOperator{\End}{End}
\DeclareMathOperator{\Ext}{Ext}
\DeclareMathOperator{\op}{op}

\DeclareMathOperator{\Ker}{Ker}
\DeclareMathOperator{\im}{Im}
\renewcommand{\Im}{\im}

\DeclareMathOperator{\Spec}{Spec}

\DeclareMathOperator{\Pic}{Pic}
\DeclareMathOperator{\rad}{rad}

\DeclareMathOperator{\RHom}{\mathbb{R}Hom}

\DeclareMathOperator{\ch}{char}
\DeclareMathOperator{\Conv}{Conv}
\DeclareMathOperator{\Int}{Int}

\def\gl{\mathop{\rm gl.dim}\nolimits}

\def\pd{\mathop{\rm proj.dim}\nolimits}

\theoremstyle{definition}
\newtheorem{Thm}{Theorem}[section]
\newtheorem{Lem}[Thm]{Lemma}
\newtheorem{Prop}[Thm]{Proposition}
\newtheorem{Cor}[Thm]{Corollary}
\newtheorem{Def}[Thm]{Definition}
\newtheorem{Ex}[Thm]{Example}
\newtheorem{Rem}[Thm]{Remark}
\newtheorem{Ques}[Thm]{Question}

\newcommand{\FRAC}[2]{\leavevmode\kern.1em\raise.5ex\hbox{\the\scriptfont0 #1}\kern-.1em/\kern-.15em\lower.25ex\hbox{\the\scriptfont0 #2}}

\title[Higher RI algebras and toric Fano stacks of Picard number one or two]{Higher representation infinite algebras and toric Fano stacks of Picard number one or two}
\author{Ryu Tomonaga}
\address{Graduate School of Mathematical Sciences, The University of Tokyo, 3-8-1 Komaba, Meguro-ku, Tokyo, 153-8914, Japan}
\email{ryu-tomonaga@g.ecc.u-tokyo.ac.jp}

\begin{document}
\begin{abstract}
Tilting bundles translate geometry into non-commutative algebra via derived equivalences. Among them, $d$-tilting bundles on $d$-dimensional smooth proper stacks, namely tilting bundles whose endomorphism algebras have global dimension at most $d$, are especially significant: they provide natural examples of $d$-representation infinite algebras and a bridge to the derived McKay correspondence. In this paper, we prove the existence of and classify basic $d$-tilting bundles consisting of line bundles on $d$-dimensional smooth toric Fano stacks of Picard number one or two.

The classification is motivated by dimer models. An internal perfect matching of a dimer model gives a positive grading on its dimer algebra, whose degree-zero part yields a $2$-representation infinite algebra. The higher representation infinite algebras of type $\widetilde A$ introduced by Herschend--Iyama--Oppermann can be viewed as higher-dimensional analogues of this construction in the simplex case. In the next case, the same principle naturally leads to a new class of higher representation infinite algebras, which we call algebras of type $\widetilde A\widetilde A$.

Our main structural result is that upper sets provide a common framework for tilting bundles, toric non-commutative crepant resolutions (NCCRs), and cuts of higher-dimensional dimer-type quivers. In the Picard-number-one case, $d$-tilting bundles consisting of line bundles are in bijection with non-trivial upper sets in the Picard group, and their endomorphism algebras are precisely $d$-representation infinite algebras of type $\widetilde A$. In the Picard-number-two case, the upper-set construction becomes two-step: the first upper set determines the ambient toric NCCR, and the second selects an internal cut of its quiver. This gives a complete classification of such $d$-tilting bundles and realizes their endomorphism algebras precisely as $d$-representation infinite algebras of type $\widetilde A\widetilde A$.

Thus smooth toric Fano stacks of Picard number one and two serve as geometric models of algebras of type $\widetilde A$ and $\widetilde A\widetilde A$, respectively. Using these models, we show that both classes of algebras are closed under $d$-APR tilts via a combinatorial description of the corresponding $d$-APR tilting modules.
\end{abstract}

\maketitle
\tableofcontents

\section*{Introduction}

\subsection{Background from tilting theory for toric stacks and higher Auslander--Reiten theory}

Tilting theory is an indispensable tool for establishing derived equivalences and serves as a bridge among many areas of mathematics including representation theory, algebraic geometry and mathematical physics: a tilting bundle $\E$ on a variety $X$ yields a derived equivalence
\[
\D(X)\ \simeq\ \D(\End_X(\E)).
\]
For projective varieties, Beilinson first constructed tilting bundles on projective spaces $\mathbb{P}^d$ \cite{Bei}. Since then, many tilting bundles have been constructed on various spaces \cite{BH,HIMO,HP14,Kap,Tom25a}, including stacky varieties.

We recall some classical results on tilting theory for stacky curves. By Beilinson's result, the projective line $\mathbb{P}^1$ is derived equivalent to the path algebra $k[\xymatrix@C=15pt{\circ \ar@2[r] & \circ}]$ of the Kronecker quiver. As a generalization, Geigle and Lenzing \cite{GL87} showed that every weighted projective line (=root stack of $\mathbb{P}^1$ at rational points) admits a tilting bundle. In the Fano case, the endomorphism algebra of a suitable tilting bundle is the path algebra of an extended Dynkin quiver. Conversely, all such path algebras can be obtained in this way. In this sense, Fano weighted projective lines can be viewed as geometric models of path algebras of extended Dynkin type. We remark here that over algebraically closed fields of characteristic zero, Fano weighted projective lines are precisely smooth Fano stacky curves \cite{VZB22}.

\[
\adjustbox{center}{\small
\scalebox{1.7}[1.7]{
$\begin{xy}
(20,0)*+[F:<10pt>]{\begin{array}{c}\mbox{Smooth toric Fano}\\ \mbox{stacky curves}\end{array}}="0",
(80,0)*+[F:<10pt>]{\begin{array}{c}\mbox{Path algebras}\\ \mbox{of type $\widetilde A$}\end{array}}="1",
(20,20)*+[F:<10pt>]{\begin{array}{c}\mbox{Smooth Fano}\\ \mbox{stacky curves}\end{array}}="2",
(80,20)*+[F:<10pt>]{\begin{array}{c}\mbox{Path algebras}\\ \mbox{of extended Dynkin type}\end{array}}="3",

\ar@{}|-{\mathrel{\scalebox{3.0}[2]{$\sim$}}}_-{\begin{array}{c}
    \text{\tiny derived}\\[-6pt]
    \text{\tiny equivalent}
  \end{array}}"0";"1",
\ar@{}|-{\mathrel{\scalebox{3.0}[2]{$\sim$}}}_-{\begin{array}{c}
    \text{\tiny derived}\\[-6pt]
    \text{\tiny equivalent}
  \end{array}}"2";"3",
\ar@{}|-{\mathrel{\rotatebox[origin=c]{90}{$\subset$}}}"0";"2",
\ar@{}|-{\mathrel{\rotatebox[origin=c]{90}{$\subset$}}}"1";"3",
\end{xy}
$}}
\]

Among Fano weighted projective lines, those derived equivalent to the path algebras of type $\widetilde A$ are precisely smooth toric Fano stacky curves. This paper establishes higher-dimensional generalizations of this derived equivalence in the setting of toric stacks.

For smooth projective toric stacks, one of the basic guiding problems is the following.

\begin{Ques}\label{probtoric}
Let $\X$ be a smooth projective toric stack. Does $\X$ have a tilting bundle consisting of line bundles?
\end{Ques}

King conjectured an affirmative answer for smooth projective toric varieties \cite{Kin97}, but Hille and Perling found a counterexample \cite{HP06}. Borisov and Hua later proposed that Question \ref{probtoric} is true for smooth toric weak Fano stacks, and proved it for smooth toric Fano stacks of Picard number at most two and for smooth toric Fano stacks of dimension two \cite{BH}. Ishii and Ueda proved the weak Fano surface case by using dimer models \cite{IU09}. On the other hand, Efimov constructed infinitely many smooth toric Fano varieties of Picard number three which give counterexamples to Question \ref{probtoric} \cite{Efi}. We also recall that Kawamata proved that arbitrary smooth toric stacks admit full exceptional collections, not necessarily consisting of line bundles \cite{Kaw06}.

Higher Auslander--Reiten theory provides a second motivation.
Iyama \cite{Iya07a,Iya07b} generalized classical Auslander--Reiten
theory to study higher structures in module and derived categories
\cite{HIO,Iya11,IO11}. This theory connects to non-commutative crepant
resolutions \cite{VdB04a}, Calabi--Yau dg algebras \cite{Gin06,Kel11},
and additive categorification \cite{Ami09,BMRRT}. Herschend, Iyama and
Oppermann \cite{HIO} introduced $d$-representation infinite algebras
for $d\ge1$ as a higher-dimensional analogue of non-Dynkin path algebras.

Dimer models provide one of the basic sources of higher representation infinite algebras. More precisely, under suitable consistency assumptions, one can associate a $2$-representation infinite algebra to the following data (\cite{AIR15,MM,Nak22}).
\begin{itemize}
\item A dimer model.
\item An internal perfect matching of this dimer model.
\end{itemize}
In this construction, the perfect matching plays the role of a cut: it gives a positive grading on the dimer algebra, and the corresponding degree-zero part is the desired finite dimensional algebra.

Herschend--Iyama--Oppermann \cite{HIO} introduced a higher-dimensional analogue of this picture in a special but fundamental situation. They considered periodic quivers equipped with bounding cuts, and the resulting algebras are called $d$-representation infinite algebras of type $\widetilde A$; for $d=1$, they recover the path algebras of quivers of type $\widetilde A$. In the present paper, we develop the next case of this philosophy. Namely, starting from a higher-dimensional dimer model and an internal cut, we introduce a new class of finite dimensional algebras, which we call $d$-representation infinite algebras of type $\widetilde A\widetilde A$.

Another source of higher representation infinite algebras is projective geometry. If a $d$-dimensional smooth proper (stacky) variety has a $d$-tilting bundle (that is, a tilting bundle whose endomorphism algebra has global dimension at most $d$), then its endomorphism algebra is $d$-representation infinite. For example, Beilinson's tilting bundle $\bigoplus_{i=0}^d\O_{\mathbb{P}^d}(i)\in\Coh\mathbb{P}^d$ is a $d$-tilting bundle. More generally, $d$-tilting bundles on smooth proper varieties and stacks have been studied systematically in \cite{BHI,Han24a,HI,HIMO,Nak22,Tom25c,Tom25a}, and they admit useful geometric interpretations \cite{BuH,Tom25a}. This leads to the following refinement of Question \ref{probtoric}.

\begin{Ques}\label{probtoricdtilt}
Let $\X$ be a $d$-dimensional smooth projective toric stack. Does $\X$ have a $d$-tilting bundle consisting of line bundles?
\end{Ques}

The purpose of this paper is to answer Question \ref{probtoricdtilt} affirmatively for smooth toric Fano stacks of Picard number at most two, and to classify all such $d$-tilting bundles.  The classification goes beyond existence: it identifies the endomorphism algebras with explicit classes of higher representation infinite algebras and describes their $d$-APR tilting mutations combinatorially. More precisely, the endomorphism algebras of $d$-tilting bundles consisting of line bundles on $d$-dimensional smooth toric Fano stacks of Picard number one (respectively, two) are exactly $d$-representation infinite algebras of type $\widetilde A$ (respectively, type $\widetilde A\widetilde A$).

We recall that there are infinitely many smooth toric Fano varieties of Picard number {\it three} without tilting bundles consisting of line bundles \cite{Efi}. Thus Picard number {\it two} is the natural borderline where a general positive theorem can still be expected, and it is precisely in this range that the present paper gives a complete classification.

\subsection{Picard number one}

Let $\X$ be a $d$-dimensional smooth toric Fano stack with Picard number one. If we put $G:=\Pic\X$, then by Gale duality, we have $d+1$ elements $\vec{x}_0,\cdots,\vec{x}_d\in G$. These elements define a partial order on $G$ as follows:
\[\vec{g}_1\geq\vec{g}_2\Leftrightarrow\vec{g}_1-\vec{g}_2\in\sum_{i=0}^d\mathbb{Z}_{\geq0}\vec{x}_i\subseteq G.\]
Let $\vec{p}=\sum_{i=0}^d\vec{x}_i\in G$. The first main result gives a complete classification of tilting bundles on $\X$ consisting of line bundles in terms of upper sets.

\begin{Thm}[Theorem \ref{classfitiltrk1}]\label{introclassfitiltrk1}
Let $\X$ be a $d$-dimensional smooth toric Fano stack of Picard number one. The assignment
\[
  I\longmapsto
  \bigoplus_{\vec g\in I\cap(I^c+\vec p)}\O_{\X}(\vec g)
\]
gives a bijection from the set of non-trivial upper sets in $G$ to the set of tilting bundles on $\X$ consisting of line bundles.  Moreover, all these tilting bundles are $d$-tilting bundles.
\end{Thm}

In this rank-one situation, the Stanley-Reisner locus is the origin, so the derived category of $\X$ can be controlled by the $G$-graded polynomial ring associated with the Gale dual data.

The rank-one classification has a representation-theoretic consequence.  Herschend, Iyama, and Oppermann introduced $d$-representation infinite algebras of type $\widetilde A$ associated with a cofinite subgroup $B$ of a fixed $d$-dimensional lattice and a bounding cut $C$ of a periodic quiver \cite{HIO}. Conceptually, this class of algebras is defined from the following data.
\begin{itemize}
\item The $d$-dimensional dimer model corresponding to a $d$-dimensional lattice simplex.
\item An internal cut of the dual quiver of this dimer model.
\end{itemize}
We prove that the same algebras appear as endomorphism algebras of the tilting bundles above. An important step is to construct a cut $C(I)$ of the periodic quiver from a non-trivial upper set $I\subseteq G$.

\begin{Thm}[Theorem \ref{enddtiltrk1}]
Let $\X$ be a $d$-dimensional smooth toric Fano stack with Picard number one.
\begin{enumerate}
\item For each non-trivial upper set $I\subseteq G$, we have an isomorphism
\[\End_\X\bigg(\bigoplus_{\vec{g}\in I\cap(I^c+\vec{p})}\O_\X(\vec{g})\bigg)\cong A(B,C(I)).\]
In particular, the endomorphism algebra of every tilting bundle on $\X$ consisting of line bundles is a $d$-representation infinite algebra of type $\widetilde A$.
\item Conversely, every $d$-representation infinite algebra of type $\widetilde{A}$ can be obtained in this way.
\end{enumerate}
\end{Thm}

\vspace{10pt}

\[
\adjustbox{center}{\small
\scalebox{1.7}[1.7]{
$\begin{xy}
(20,0)*+[F:<10pt>]{\begin{array}{c}\mbox{Smooth toric Fano stacks}\\ \mbox{of Picard number one}\end{array}}="0",
(80,0)*+[F:<10pt>]{\begin{array}{c}\mbox{Higher representation infinite}\\ \mbox{algebras of type $\widetilde A$}\end{array}}="1",
\ar@{}|-{\mathrel{\scalebox{3.0}[2]{$\sim$}}}_-{\begin{array}{c}
    \text{\tiny derived}\\[-6pt]
    \text{\tiny equivalent}
  \end{array}}"0";"1",
\end{xy}
$}}
\]

\vspace{10pt}

Thus smooth toric Fano stacks of Picard number one can be regarded as geometric models of higher representation infinite algebras of type $\widetilde A$. This viewpoint gives a transparent description of $d$-APR tilting for these algebras.  A cut of the periodic quiver has a type, and cuts of the same type correspond to upper sets related by mutations.

\begin{Thm}[Theorem \ref{dAPRAtilde}, \ref{dAPRconnAtilde}]
Let $A=A(B,C)$ be a $d$-representation infinite algebra of type $\widetilde A$.
\begin{enumerate}
\item The endomorphism algebra of a $d$-APR tilting module of $A$ is again a $d$-representation infinite algebra of type $\widetilde A$, of the form $A(B,C')$ where $C'$ has the same type as $C$.
\item If $C'$ is any cut of $Q$ with the same type as $C$, then $A(B,C)$ and $A(B,C')$ can be connected by a finite sequence of iterated $d$-APR tilts. In particular, they are derived equivalent.
\end{enumerate}
\end{Thm}

This recovers and gives a geometric proof of the expected connectivity statement for algebras of type $\widetilde A$; compare also \cite{DG}.

The upper-set construction appearing in Theorem \ref{introclassfitiltrk1} is not merely a feature of commutative toric geometry. It is a manifestation of a general Beilinson-window phenomenon for graded Calabi--Yau dg algebras. To make this point precise, we include Appendix \ref{appenBei}, where we prove a Beilinson-type theorem for $G_{\ge 0}$-graded dg rings with $G$ a finitely generated abelian group of rank one with a suitable partial order. In this framework, a non-trivial upper set $I\subseteq G$ determines a window $J(I)=I\cap (I^c+p)$, and this window gives a set of generators for the corresponding graded cluster category.

We also prove a $G$-graded version of the Minamoto--Mori correspondence for connective Calabi--Yau dg algebras. In particular, under the finiteness and Calabi--Yau assumptions of Theorem \ref{MMcorr}, when $\Gamma$ is an ordinary algebra, the object $\bigoplus_{g\in J(I)}\mathcal O_\Gamma(-g)$ is a $d$-tilting object, and its endomorphism algebra is $d$-representation infinite. Thus Appendix \ref{appenBei} gives a non-commutative conceptual explanation of the rank-one classification theorem: Theorem \ref{introclassfitiltrk1} is the commutative toric instance of this general window construction.

\subsection{Picard number two}

We next consider smooth toric Fano stacks of Picard number two. We classify $d$-tilting bundles consisting of line bundles on them. Then we introduce higher representation infinite algebras of type $\widetilde{A}\widetilde{A}$. Conceptually, as in the type $\widetilde A$ case, this class of algebras is defined from the following data.
\begin{itemize}
\item The $d$-dimensional dimer model corresponding to a $d$-dimensional simplicial lattice polytope with $d+2$ vertices.
\item An internal cut of the dual quiver of this dimer model.
\end{itemize}
We show that these algebras are precisely the endomorphism algebras of the $d$-tilting bundles we classified.

Let $\X$ be a $d$-dimensional smooth toric Fano stack of Picard number two, and put $G:=\Pic\X$.  By Gale duality, there are $d+2$ elements
\[
  \vec x_0,\cdots,\vec x_l,\vec x'_0,\cdots, \vec x'_{l'}\in G,
\]
where $l+l'=d$. Put $\vec p:=\sum_{i=0}^l\vec x_i+\sum_{i'=0}^{l'}\vec x'_{i'}\in G$ and $H:=G/\mathbb Z\vec p$. Let $q\colon G\to H$ and $\pi\colon H\to H/H_{\rm tors}\cong\mathbb Z$ be the quotient maps. The simpliciality of the corresponding polytope implies that after reordering the $\vec x_i$, the images in $H/H_{\rm tors}\cong\mathbb Z$ split into positive and negative parts:
\[
\pi(q(\vec x_i))>0\quad(0\leq i\leq l),
\qquad
\pi(q(\vec x'_{i'}))<0\quad(0\leq i'\leq l'),
\]
and $l,l'\geq1$.  The group $H$ carries the partial order
\[
  h_1\geq h_2
  \quad\Longleftrightarrow\quad
  h_1-h_2\in
  \sum_{i=0}^{l}\mathbb Z_{\geq0}q(\vec x_i)
  +
  \sum_{i'=0}^{l'}\mathbb Z_{\geq0}q(-\vec x'_{i'}).
\]
Put $s:=\sum_{i=0}^{l}q(\vec x_i)=\sum_{i'=0}^{l'}q(-\vec x'_{i'})\in H$. For an upper set $I\subseteq H$, set $J(I):=I\cap(I^c+s)$.  The Picard-number-two classification is obtained by applying the rank-one upper-set construction twice: first to $H$, and then to the partially ordered set $q^{-1}(J(I))\subseteq G$.

\begin{Thm}[Theorem \ref{classfitiltrk2}]\label{introclassfitiltrk2}
Let $\X$ be a $d$-dimensional smooth toric Fano stack of Picard number two.  There is a bijection between the following two sets.
\begin{enumerate}
\item Pairs $(I,I')$ such that $I$ is a non-trivial upper set in $H$ and $I'$ is a non-trivial upper set in the partially ordered set $q^{-1}(J(I))$.
\item $d$-tilting bundles on $\X$ consisting of line bundles.
\end{enumerate}
The corresponding $d$-tilting bundle is
\[
  \bigoplus_{\vec g\in I'\cap(I'^c+\vec p)}\O_{\X}(\vec g),
\]
where $I'^c$ denotes the complement of $I'$ in $q^{-1}(J(I))$.
\end{Thm}

Theorem \ref{introclassfitiltrk2} strengthens the known existence result of Borisov and Hua for Picard number two \cite{BH}: our theorem gives a classification and, at the same time, proves that the resulting tilting bundles are $d$-tilting.

The corresponding endomorphism algebras form a second explicit family of higher representation infinite algebras.  The first upper set $I\subseteq H$ determines a cut $C:=C(I)$ of a quiver associated with a cofinite subgroup $B\subseteq L\oplus L'$.  The algebra $\Gamma(B,C)$ is a non-commutative crepant resolution of a Gorenstein toric singularity with divisor class group of rank one \cite{Tom25d} whose quiver $Q(C)$ can be viewed as the dual quiver of a higher-dimensional generalization of dimer models. A cut $C'$ of the quiver $Q(C)$ gives a $\mathbb{N}$-grading on $\Gamma(B,C)$, and we define
\[
  A(B,C,C'):=\Gamma(B,C)_0.
\]
Here, a cut $C'$ should be thought of as a higher-dimensional analogue of a perfect matching. Its type is defined to be a lattice point in the corresponding polytope. This type lies in the interior of the polytope exactly when the algebra $A(B,C,C')$ is finite dimensional (Proposition \ref{doubleAtildeacy}). If this is the case, then one can show that $A(B,C,C')$ is $d$-representation tame.

\begin{Thm}[Theorem \ref{AAdtame}]
Let $C'$ be a cut of $Q(C)$ whose type lies in the interior of $P$. 
\begin{enumerate}
\item The finite dimensional algebra $A(B,C,C')$ is a $d$-representation tame algebra.
\item We have the following isomorphism as a graded algebra.
\[\Pi_{d+1}(A(B,C,C'))\cong\Gamma(B,C)\]
\end{enumerate}
\end{Thm}

This motivates the terminology ``$d$-representation infinite algebra of type $\widetilde A\widetilde A$''. An important step is to construct a cut $C'(I')$ of $Q(C(I))$ whose type lies in the interior of the polytope from the second upper set $I'\subseteq q^{-1}(J(I))$. Conversely, every internal cut $C'$ can be obtained in this way (cut-upper set correspondence, Theorem \ref{cutupcorr2}).

\begin{Thm}[Theorem \ref{enddtiltrk2}]
Let $\X$ be a $d$-dimensional smooth toric Fano stack of Picard number two.
\begin{enumerate}
\item For every pair $(I,I')$ as in Theorem \ref{introclassfitiltrk2}, we have an isomorphism
\[
  \End_{\X}\bigg(
    \bigoplus_{\vec g\in I'\cap(I'^c+\vec p)}\O_{\X}(\vec g)
  \bigg)
  \cong
  A(B,C(I),C'(I')).
\]
In particular, the endomorphism algebra of every $d$-tilting bundle on $\X$ consisting of line bundles is a $d$-representation infinite algebra of type $\widetilde A\widetilde A$.
\item Conversely, every $d$-representation infinite algebra of type $\widetilde A\widetilde A$ can be obtained in this way.
\end{enumerate}
\end{Thm}

\vspace{10pt}

\[
\adjustbox{center}{\small
\scalebox{1.7}[1.7]{
$\begin{xy}
(20,0)*+[F:<10pt>]{\begin{array}{c}\mbox{Smooth toric Fano stacks}\\ \mbox{of Picard number two}\end{array}}="0",
(80,0)*+[F:<10pt>]{\begin{array}{c}\mbox{Higher representation infinite}\\ \mbox{algebras of type $\widetilde A\widetilde A$}\end{array}}="1",
\ar@{}|-{\mathrel{\scalebox{3.0}[2]{$\sim$}}}_-{\begin{array}{c}
    \text{\tiny derived}\\[-6pt]
    \text{\tiny equivalent}
  \end{array}}"0";"1",
\end{xy}
$}}
\]

\vspace{10pt}

Thus smooth toric Fano stacks of Picard number two can be regarded as geometric models of higher representation infinite algebras of type $\widetilde A\widetilde A$. Basic examples of algebras of type $\widetilde A\widetilde A$ are given by tensor products of algebras of type $\widetilde A$ (Proposition \ref{tensorAandA}). We emphasize that the class of algebras of type
$\widetilde A\widetilde A$ is not merely a tensor-product construction: the second cut $C'$ may be chosen
independently inside the quiver $Q(C)$, and this produces genuinely new algebras. In Examples \ref{ExHirz}, \ref{Exstackysurface}, and \ref{Ex3fold}, we compute explicit non-product examples, including examples coming from Hirzebruch surfaces $\mathbb{P}^1\times\mathbb{P}^1$ and $\Sigma_1$, and a
three-dimensional example on the $\mathbb P^1$-bundle
$\mathbb P_{\mathbb P^2}(\mathcal O_{\mathbb P^2}\oplus
\mathcal O_{\mathbb P^2}(-2))$.

The geometric model also controls $d$-APR tilting mutations.

\begin{Thm}[Theorems \ref{dAPRdoubleAtilde}, \ref{dAPRconndoubleAtilde}]
Let $A=A(B,C,C')$ be a $d$-representation infinite algebra of type $\widetilde A\widetilde A$.
\begin{enumerate}
\item The endomorphism algebra of a $d$-APR tilting module of $A$ is again of type $\widetilde A\widetilde A$, with the same $B$, the same first cut $C$, and a second cut with the same lattice-point type.
\item Any two algebras $A(B,C,C'_1)$ and $A(B,C,C'_2)$ with the same lattice-point type are connected by a finite sequence of iterated $d$-APR tilts.  In particular, they are derived equivalent.
\end{enumerate}
\end{Thm}

\subsection*{Conventions}
Throughout this paper, $k$ denotes an arbitrary field. All algebras and categories are defined over $k$. Tilting bundles consisting of line bundles are taken to be basic, so their summands are pairwise non-isomorphic. For an abelian group $G$ and a $G$-graded ring $A$, let $\mod^GA$ and $\proj^GA$ denote the categories of finitely generated $G$-graded right $A$-modules and finitely generated $G$-graded projective right $A$-modules respectively. For a $G$-graded dg ring $\Gamma$, we write $\D^G(\Gamma)$ and $\per^G\Gamma$ for the unbounded derived category of $G$-graded right dg $\Gamma$ modules and the perfect derived category of $G$-graded right dg $\Gamma$ modules respectively.

\section*{Acknowledgements}
The author is grateful to Osamu Iyama for fruitful discussions. This work was supported by the WINGS-FMSP program at the Graduate School of Mathematical Sciences, the University of Tokyo, and JSPS KAKENHI Grant Number JP25KJ0818.

\section{Preliminaries}

\subsection{Higher representation infinite algebras}

First, we recall the definition of higher representation infinite algebras introduced by \cite{HIO}.

\begin{Def}\cite[2.7]{HIO}
Let $A$ be a finite dimensional algebra and $d\geq1$ an integer.
For $A$ of finite global dimension, write
$\nu=-\otimes_A^{\mathbb L}DA$ for the Serre functor of $\per A$ and
$\nu_d=\nu[-d]$. Put $\Pi_{d+1}(A)=T_A\Ext^d_A(DA,A)$.
\begin{enumerate}
\item $A$ is called {\it $d$-representation infinite} if $\gl A\leq d$ and 
\[\nu_d^{-n}A\in\mod A\subseteq\per A\]
holds for all $n\geq0$.
\item A $d$-representation infinite algebra $A$ is called {\it $d$-representation tame} if its $(d+1)$-preprojective algebra $\Pi_{d+1}(A)$ is a module-finite algebra over some commutative noetherian ring.
\end{enumerate}
\end{Def}

This is a generalization of non-Dynkin path algebras to higher global dimension from the viewpoint of higher Auslander--Reiten theory. As in the case of non-Dynkin path algebras, we have a $d$-preprojective component $\P:=\add\{\nu_d^{-n}A\mid n\geq0\}\subseteq\mod A$ and a $d$-preinjective component $\I:=\add\{\nu_d^n(DA)\mid n\geq0\}\subseteq\mod A$. For other beautiful properties of higher representation infinite algebras, see \cite{HIO}. In order to show a systematic way to give examples of higher representation infinite algebras, we introduce the following terminology.

\begin{Def}
Let $\T$ be a triangulated category. Take an object $X\in\T$.
\begin{enumerate}
\item $X$ is called {\it pretilting} if $\T(X,X[\neq0])=0$ holds.
\item $X$ is called {\it tilting} if it is pretilting and $\thick X=\T$ holds.
\item For $d\geq1$, $X$ is called {\it d-tilting} if it is tilting and $\gl\End_{\T}(X)\leq d$ holds.
\end{enumerate}
\end{Def}

The following proposition provides a systematic way to construct examples of higher representation infinite algebras by studying tilting objects for certain abelian categories \cite{BuH}. For a proof, see \cite{Tom25a}.

\begin{Prop}\cite{BuH}\label{SerreRI}
Let $\A$ be a Hom-finite abelian category and $T\in\A$ a $d$-tilting
object of $\D^b(\A)$. If $\A$ has an auto-equivalence
$F\colon\A\xrightarrow[\simeq]{}\A$ such that
$F[d]\colon\D^b(\A)\xrightarrow[\simeq]{}\D^b(\A)$ is a Serre functor,
then $\End_\A(T)$ is $d$-representation infinite.
\end{Prop}

For example, if $\A=\Coh X$ for a $d$-dimensional smooth projective variety $X$, then we can apply this proposition. A typical example of a $d$-tilting bundle is Beilinson's tilting bundle $\bigoplus_{i=0}^d\O_{\mathbb{P}^d}(i)\in\Coh\mathbb{P}^d$. Thus $\End_{\mathbb{P}^d}(\bigoplus_{i=0}^d\O_{\mathbb{P}^d}(i))$ is a $d$-representation infinite algebra by Proposition \ref{SerreRI} (see also \cite[2.15]{HIO}), which turns out to be of type $\widetilde A$ (Theorem \ref{classfitiltrk1}). For further connections between higher representation infinite algebras and projective geometry, see \cite{BuH,Tom25a}.

Next, we recall a family of higher representation infinite algebras, which are called higher representation infinite algebras of type $\widetilde A$, introduced by \cite{HIO}. Let $e_i\in\mathbb{Z}^{d+1}$ be the $i$-th unit vector for $0\leq i\leq d$. Put $\alpha_i:=e_i-e_{i-1}$ for $1\leq i\leq d$ and $\alpha_0:=e_0-e_d$. Let $L:=\{v=(v_i)_{i=0}^d\in\mathbb{Z}^{d+1}\mid\sum_{i=0}^dv_i=0\}=\sum_{i=0}^d\mathbb{Z}\alpha_i\subseteq\mathbb{Z}^{d+1}$ be a $d$-dimensional lattice and $B\subseteq L$ a cofinite subgroup. Put $m:=\sharp(L/B)$. As in \cite{DG}, let $\hat{Q}:=(L,\bigsqcup_{i=0}^d\{x\to x+\alpha_i\mid x\in L\})$ be an infinite quiver. We say that an arrow $x\to x+\alpha_i$ in $\hat{Q}$ has {\it type} $i$. A cycle of length $d+1$ in $\hat{Q}$ consisting of arrows of $d+1$ distinct types is called an {\it elementary cycle}. A subset $\hat{C}\subseteq\hat{Q}_1$ is called a {\it cut} if every elementary cycle has exactly one arrow in $\hat{C}$. A cut $\hat{C}\subseteq\hat{Q}_1$ is said to be $B$-{\it periodic} if $\hat{C}$ is invariant under $B$-translation.

Similarly, let $Q:=(L/B,\bigsqcup_{i=0}^d\{x+B\to x+\alpha_i+B\mid x\in L\})$ be a finite quiver which may have multiple arrows. We define a cut of $Q$ similarly. For a cut $C\subseteq Q_1$, we call $\gamma(C):=(\sharp\{a\in C\mid\text{The type of $a$ is }i.\})_{i=0}^d\in\mathbb{Z}^{d+1}_{\geq0}$ the {\it type} of $C$. For a cut $C\subseteq Q_1$ of type $\gamma=(\gamma_i)_{i=0}^d$, we have $\sum_{i=0}^d\gamma_i=m$. Observe that cuts of $Q$ correspond bijectively to $B$-periodic cuts of $\hat{Q}$. In what follows, we identify $B$-periodic cuts of $\hat{Q}$ with cuts of $Q$ freely. For a cut $C\subseteq Q_1$, we define a quiver $Q_C:=(Q_0,Q_1\setminus C)$. A cut $C\subseteq Q_1$ is called {\it bounding} if the quiver $Q_C$ is acyclic.

\begin{Def}\cite[5.6(2)]{HIO}
Let $C\subseteq Q_1$ be a bounding cut. Consider the relation $I_C$ in path algebra $kQ_C$ which is generated by 
\[(x+B\to x+\alpha_i+B\to x+\alpha_i+\alpha_j+B)=(x+B\to x+\alpha_j+B\to x+\alpha_i+\alpha_j+B)\]
for $x\in L, 0\leq i,j\leq d$ such that the four arrows exist in $Q_C$. We call $A(B,C):=kQ_C/I_C$ a $d$-{\it representation infinite algebra of type} $\widetilde A$.
\end{Def}

In \cite{HIO}, it is proved that this $A(B,C)$ is indeed $d$-representation infinite when $k$ is an algebraically closed field of characteristic zero. Below, we give another proof of this fact which is valid for arbitrary field $k$ (Theorem \ref{enddtiltrk1}).

\subsection{Lattice polytopes}\label{polytope}

Let $N\cong\mathbb{Z}^d$ be a free abelian group of rank $d$ and $P$ a convex lattice polytope in $N_\mathbb{R}:=N\otimes_\mathbb{Z}\mathbb{R}$. Let $\{v_i\}_{i=1}^n$ denote the set of the vertices of $P$. Put $\widetilde{N}:=N\oplus\mathbb{Z}$ and $\widetilde{v}_i:=(v_i,1)\in\widetilde{N}$. This $\{\widetilde{v}_i\}_{i=1}^n$ defines a group homomorphism $\phi\colon\mathbb{Z}^n\to\widetilde{N}$ with finite cokernel. Define an abelian group $H$ by the following exact sequence, where $F:=(\mathbb{Z}^n)^*\cong\bigoplus_{i=1}^n\mathbb{Z}e_i$ denotes a free abelian group of rank $n$.
\[0\to\widetilde{N}^*\xrightarrow{\phi^*}F\xrightarrow{\deg_H}H\to0\]
For $1\leq i\leq n$, we write $\overline{\vec{x}}_i\in H$ for the image of $e_i\in F$.

If we put $M:=\Hom_{\mathbb Z}(N,\mathbb Z)$, then we can view $\phi^*\colon M\oplus\mathbb{Z}\to F$. Remark that we have
\[\mathbf{p}:=\phi^*(0,1)=\sum_{i=1}^ne_i\in F.\]
Write $\phi^*(m,n)=A_P(m)+n\mathbf p$, where $A_P\colon M\to F$ is given by
\[
  A_P(m)=
  \sum_{i=1}^n\langle m,v_i\rangle e_i.
\]
Then we have the short exact sequence
\[
  0\longrightarrow A_P(M)+\mathbb Z\mathbf p
  \longrightarrow F
  \xrightarrow{\deg_H}
  H
  \longrightarrow 0.
\]
We also have
\[A_P(M)\cap \mathbb Z\mathbf p=0.\]

\begin{Def}
Let $G$ be a finitely generated abelian group of rank $n-d$ and $\vec{x}_1,\cdots,\vec{x}_n\in G$ its elements. Let $q\colon G\to H$ be a group homomorphism. A tuple
\[
  (G,(\vec{x}_i)_{i=1}^n,q)
\]
is called a \emph{lift} of $(H,(\overline{\vec{x}}_i)_{i=1}^n)$ if the following conditions are satisfied.
\begin{enumerate}
\item $G=\sum_{i=1}^n\mathbb{Z}\vec{x}_i$
\item $q(\vec{x}_i)=\overline{\vec{x}}_i$
\item If we put $\vec p:=\sum_{i=1}^n\vec{x}_i\in G$, then we have
\[\Ker q=\mathbb Z\vec p.\]
\end{enumerate}
\end{Def}

We define subsets $F_{>0},F_{\ge0}\subseteq F$ as
\[F_{>0}:=\left\{\sum_{i=1}^na_ie_i\mid a_i>0\text{ for all }i\right\}\quad\text{and}\quad F_{\ge0}:=\left\{\sum_{i=1}^na_ie_i\mid a_i\ge0\text{ for all }i\right\}.\]
We see that isomorphism classes of lifts of $(H,(\overline{\vec{x}}_i)_{i=1}^n)$ are in bijection with lattice points $u\in N$. Moreover, we can characterize which lifts correspond to (interior) lattice points of the polytope $P$.

\begin{Prop}\label{liftlatticepoint}
For a lift $(G,(\vec{x}_i)_i,q)$ of $(H,(\overline{\vec{x}}_i)_i)$, consider the group homomorphism $F\to G$ sending $e_i$ to $\vec x_i$ and put
\[K:=\Ker(F\to G).\]
\begin{enumerate}
\item There is a natural bijection between the following two sets.
\begin{enumerate}
\item Lattice points $u\in N$.
\item Isomorphism classes of lifts $(G,(\vec{x}_i)_i,q)$ of $(H,(\overline{\vec{x}}_i)_i)$.
\end{enumerate}
Here, an isomorphism of lifts is required to commute with $q$ and to
send each $\vec{x}_i$ to the corresponding element.
\item The bijection in (1) restricts to the bijection between the followings.
\begin{enumerate}
\item Lattice points $u\in P\cap N$.
\item Isomorphism classes of lifts $(G,(\vec{x}_i)_i,q)$ of $(H,(\overline{\vec{x}}_i)_i)$ with
\[K\cap F_{>0}=\emptyset.\]
\end{enumerate}
\item The bijection in (2) restricts to the bijection between the followings.
\begin{enumerate}
\item interior lattice points $u\in\Int(P)\cap N$.
\item Isomorphism classes of lifts $(G,(\vec{x}_i)_i,q)$ of $(H,(\overline{\vec{x}}_i)_i)$ with
\[K\cap F_{\ge0}=0.\]
\end{enumerate}
\end{enumerate}
\end{Prop}

\begin{proof}
(1) First take $u\in N$.  Define
\[
  A_u\colon M\longrightarrow F,
  \qquad
  A_u(m):=A_P(m)-\langle m,u\rangle\mathbf p.
\]
Put
\[
  G_u:=F/A_u(M).
\]
Let $\vec{x}_{i,u}$ be the image of $e_i$ in $G_u$, and put
\[
  \vec p_u:=\sum_{i=1}^n\vec{x}_{i,u}\in G_u.
\]
Since $A_u(M)+\mathbb Z\mathbf p=A_P(M)+\mathbb Z\mathbf p$, the quotient
map $F\to H$ factors through a surjective homomorphism
\[
  q_u\colon G_u\to H.
\]
Moreover, we have
\[
  \Ker q_u=\mathbb Z\vec p_u.
\]
Thus $(G_u,(\vec{x}_{i,u})_i,q_u)$ is a lift.

Conversely, let $(G,(\vec{x}_i)_i,q)$ be a lift and let
$K:=\Ker(F\to G)$.  Since $\Ker q=\mathbb Z\vec p$ and
$H\cong F/(A_P(M)+\mathbb Z\mathbf p)$, we have
\[
  K+\mathbb Z\mathbf p=A_P(M)+\mathbb Z\mathbf p.
\]
Since $\vec p$ has infinite order, $K\cap\mathbb Z\mathbf p=0$ holds. Hence for
any $m\in M$, there exists a unique integer $c(m)$ such that
\[
  A_P(m)-c(m)\mathbf p\in K.
\]
The uniqueness immediately implies that $c\colon M\to\mathbb Z$ is a group
homomorphism.  Hence there is a unique $u\in N$ satisfying
\[
  c(m)=\langle m,u\rangle
\]
for all $m\in M$.  Thus
\[
  K=\{A_P(m)-\langle m,u\rangle\mathbf p\mid m\in M\}=A_u(M)
\]
holds and the given lift is isomorphic to the lift constructed from $u$.

(2) If $u\notin P$, then by the separating hyperplane theorem, there exists
$m\in M$ such that
\[
  \langle m,u\rangle<\min_i\langle m,v_i\rangle
\]
holds. This means $A_P(m)-\langle m,u\rangle\mathbf p\in K\cap F_{>0}$. 

Conversely, assume that the lift $(G_u,(\vec{x}_{i,u})_i,q_u)$ satisfies $K\cap F_{>0}\ne\emptyset$. Take $A_u(m)\in F_{>0}$.  Then
\[
  \langle m,v_i-u\rangle>0
\]
holds for every vertex $v_i$ of $P$.  This means
\[
  \langle m,u\rangle<\min_i\langle m,v_i\rangle.
\]
Therefore we have $u\notin P$.

(3) The same argument as in (2) applies.
\end{proof}

\begin{Rem}
The condition $K\cap F_{\ge0}=0$ in (3) is equivalent to the following conditions.
\begin{enumerate}
\item $\vec x_i\ne0$ for $1\le i\le n$.
\item If we put $G_{\geq0}:=\sum_{i=1}^n\mathbb{Z}_{\geq0}\vec{x}_i\subseteq G$, then $G_{\ge0}\cap(-G_{\ge0})=0$.
\end{enumerate}
\end{Rem}

\subsection{Smooth toric Fano stacks}

In this subsection, we recall the definition and basic properties of smooth toric Fano stacks from \cite{BCS}. We use the notations in Subsection \ref{polytope}. We assume that $P$ is a positive-dimensional simplicial lattice polytope with the origin in its interior. Define an abelian group $G$ by the following exact sequence.
\[0\to M\xrightarrow{A_P}F\to G\to0\]
This $G$ is denoted by $G_0$ in Proposition \ref{liftlatticepoint}. For $1\leq i\leq n$, we write $\vec{x}_i\in G$ for the image of $e_i\in F$. Then the polynomial ring $S:=k[x_1,\cdots, x_n]$ can be viewed as a $G$-graded $k$-algebra by $\deg x_i:=\vec{x}_i$. This grading induces an action of a group scheme $\Spec k[G]$ on $\mathbb{A}_k^n$. We define a Stanley-Reisner locus $SR(P)\subseteq\mathbb{A}_k^n$ as a closed subscheme defined by the reduced monomial ideal $(\prod_{v_i\notin Q}x_i\mid Q\subsetneq P\text{ is a proper face})\subseteq S$. Now we can associate to the polytope $P$ a smooth toric Fano stack $\X(P)$ as the quotient stack
\[\X(P):=[(\mathbb{A}_k^n\backslash SR(P))/\Spec k[G]].\]
Remark that $\X(P)$ becomes a Deligne-Mumford stack if $\ch k=0$ \cite[3.2]{BCS}. For Fano-ness, see \cite[3.11,3.12]{BH}.

\begin{Rem}
For a proper face $Q\subsetneq P$, we define a cone $\sigma_Q:=\sum_{v_i\in Q}\mathbb{R}_{\geq0}v_i\subseteq N_\mathbb{R}$. Then one gets a complete fan $\Sigma:=\{\sigma_Q\mid Q\subsetneq P\text{ is a proper face}\}$ in $N_\mathbb{R}$ and a data of a complete stacky fan $\boldsymbol{\Sigma}=(\Sigma, \{v_i\}_{i=1}^n)$. In characteristic zero this is the construction of \cite{BCS},
and over any field it is the smooth toric Artin stack of
\cite[Section~5]{Sat13}. Over a maximal cone the inverted Cox
coordinates have weights forming a basis of $G\otimes_\mathbb Z\mathbb Q$.
Setting these coordinates to $1$ presents the chart as
$[\mathbb A^d/K]$ for a finite diagonalizable group $K$. Each chart
is tame and proper over its affine toric quotient. Since $\Sigma$ is
complete, $\X(P)$ is proper and tame over $k$.
\end{Rem}

Put $\X:=\X(P)$. We have a categorical equivalence
\[\Coh\X\simeq\Coh^{\Spec k[G]}\mathbb{A}_k^n\backslash SR(P)\simeq\mod^G\!S/\mod^G_{SR(P)}S,\]
where $\mod^G_{SR(P)}S\subseteq\mod^G\!S$ is a full subcategory consisting of modules supported by $SR(P)$. We put $\widetilde{(-)}:=(\mod^G\!S\to\mod^G\!S/\mod^G_{SR(P)}S\xrightarrow[\simeq]{}\Coh\X)$. For $\vec g\in G$, the grading shift
$(\vec g)\colon\mod^G\!S\xrightarrow[\simeq]{}\mod^G\!S$ induces
the auto-equivalence $(\vec g)\colon\Coh\X\xrightarrow[\simeq]{}\Coh\X$.
The group homomorphism
\[G\to\Pic\X;\vec{g}\mapsto\O_\X(\vec{g})\]
is an isomorphism.

For $c=(c_i)_{i=1}^n\in\mathbb{Z}^n$, put
\[\Delta_c:=\{I\subsetneq\{1,\cdots,n\}\mid\Conv\{v_i\}_{i\in I}\text{ is a face of }P\text{ and }c_i\geq0\text{ holds for all }i\in I\}\]
and define a subspace $X_c\subseteq P$ as
\[X_c:=\bigcup_{I\in\Delta_c}\Conv\{v_i\}_{i\in I}.\]
Observe that our definition of $\Delta_c$ (and so $X_c$) differ from those of \cite[2.6]{Tom25d} when $c\in\mathbb{Z}_{\geq0}^n$. In this terminology, the cohomology groups of the line bundles can be computed as follows.

\begin{Prop}\label{calcoh}\cite[4.1]{BH}
Assume $\X$ is Fano. For $\vec{g}\in G$, we have
\[H^r(\X,\O_\X(\vec{g}))\cong\bigoplus_{\substack{c\in\mathbb{Z}^n \\ \sum_ic_i\vec{x}_i=\vec{g}}}\widetilde H_{d-r-1}(X_c;k),\]
where $\widetilde H_{d-r-1}(X_c;k)$ denotes the $(d-r-1)$-th reduced singular homology of $X_c$ with coefficients in $k$. Remark that we set $\widetilde H_{-1}(X;k)\left\{
\begin{array}{ll}
=0 & (X\neq\emptyset)\\
\cong k & (X=\emptyset)
\end{array}
\right.$.
\end{Prop}
\begin{proof}
On the Cox cover, the intersections of the invariant affine opens
give a \v{C}ech complex whose degree-$\vec g$ part computes the cohomology
of $\O_\X(\vec g)$. Taking homogeneous components is exact over any
field since $\Spec k[G]$ is diagonalizable. The calculation of this
complex in \cite[4.1]{BH} therefore applies over $k$.
\end{proof}

Finally, we see some properties of our group $G$. By construction, and since $P$ contains the origin as an interior point, $G$ and $\vec{x}_i\in G$ satisfy the following conditions (see Proposition \ref{liftlatticepoint}).
\begin{enumerate}
\item[(G1)] $\vec{x}_i\neq0$ for all $1\leq i\leq n$.
\item[(G2)] $G=\sum_{i=1}^n\mathbb{Z}\vec{x}_i$
\item[(G3)] If we put $G_{\geq0}:=\sum_{i=1}^n\mathbb{Z}_{\geq0}\vec{x}_i\subseteq G$, then we have $G_{\geq0}\cap(-G_{\geq0})=0$.
\item[(G4)] Put $\vec{p}:=\sum_{i=1}^n\vec{x}_i\in G$. For each $1\leq i_0\leq n$, there exists $m_i>0$ for $i\neq i_0$ such that $\sum_{i\neq i_0} m_i\vec{x}_i\in\mathbb{Z}\vec{p}$ holds.
\end{enumerate}
Conversely, if a finitely generated abelian group $G$ of rank $n-d$ and $\vec{x}_1,\cdots,\vec{x}_n\in G$ satisfying (G1), (G2), (G3) and (G4) are given, then we obtain $n$ lattice points $v_1,\cdots, v_n\in N$ whose convex hull $\Conv\{v_i\}_{i=1}^n$ contains the origin as an interior point, and each $v_i$ is a vertex of $\Conv\{v_i\}_{i=1}^n$.

Remark that if $G$ is of rank one, then the condition (G4) follows from (G1), (G2) and (G3) since $G/\mathbb{Z}\vec{p}$ is a torsion group.

We remark here that if a finitely generated abelian group $G$ and $\vec{x}_1,\cdots,\vec{x}_n\in G$ satisfy (G1), (G2), (G3) and (G4), then we can define a partial order on $G$ as
\[\vec{g}\geq\vec{h}:\Leftrightarrow\vec{g}-\vec{h}\in G_{\geq0}.\]
If we view $S:=k[x_1,\cdots,x_n]$ as a $G$-graded $k$-algebra by $\deg x_i=\vec{x}_i$, then for $\vec{g}\in G$, $S_{\vec{g}}\neq0$ holds if and only if $\vec{g}\geq0$ holds.

Coherent duality for the smooth proper tame stack $\X$ gives a
dualizing complex $L[d]$ with $L$ a line bundle
\cite[3.1,3.2]{HaPr26}. Write $L=\O_\X(\vec a)$. The complement of
the Cox open set has codimension at least two, so
$H^0(\X,\O_\X(\vec g))=S_{\vec g}$. By Proposition \ref{calcoh},
$H^d(\X,\O_\X(\vec g))\ne0$ exactly when $-\vec p-\vec g\ge0$,
since $X_c=\emptyset$ exactly when all $c_i<0$.
Duality gives $H^d(\X,L)\cong D H^0(\X,\O_\X)$ and
$H^0(\X,L(\vec p))\cong D H^d(\X,\O_\X(-\vec p))$.
Both are nonzero, so $-\vec p-\vec a\ge0$ and $\vec a+\vec p\ge0$.
By (G3), the dualizing sheaf is therefore
$\omega_\X\cong\O_\X(-\vec p)$. Graded projective resolutions over $S$
give finite locally free resolutions on $\X$, so tensoring with
$\omega_\X[d]$ gives the Serre functor $(-\vec p)[d]$ on $\D^b(\Coh\X)$.

\subsection{Combinatorics of upper sets}

In this subsection, let $X$ be a partially ordered set.

\begin{Def}
We call a subset $I\subseteq X$ an {\it upper set} if for all $x\in I$ and $y\in X$ with $x\leq y$, $y\in I$ holds. An upper set $I\subseteq X$ is called {\it non-trivial} if $I\neq\emptyset,X$. We put $\I_X:=\{I\subseteq X\colon\text{non-trivial upper set}\}$.
\end{Def}

Assume $\mathbb{Z}$ acts on the set $X$ satisfying the following conditions. Here, we write $x+np:=n\cdot x$ for $n\in\mathbb{Z}$.
\begin{enumerate}
\item[(A1)] $x<x+p$ holds for all $x\in X$.
\item[(A2)] $x\leq y$ implies $x+np\leq y+np$ for all $x,y\in X$ and $n\in\mathbb{Z}$.
\item[(A3)] For any $x,y\in X$, there exists $n\in\mathbb{Z}$ such that $x+np\geq y$ holds.
\end{enumerate}

As in \cite{Tom25d}, we put
\[\widetilde{\J}_X:=\{J\subseteq X\mid\text{For any }x,y\in J,\text{ we have }x\ngeq y+p.\}\text{ and}\]
\[\J_X:=\{J\in\widetilde{\J}_X\colon\text{maximal with respect to inclusion}\}\subseteq\widetilde{\J}_X.\]
Then we have the following bijection between $\I_X$ and $\J_X$.

\begin{Thm}\cite[1.3]{Tom25d}\label{upJX}
Consider the following sets.
\[J(-)\colon\I_X\rightleftarrows\J_X :I(-)\]
Then $J(I):=I\cap(I^c+p)$ and $I(J):=\{x\in X\mid\text{There exists }y\in J\text{ with }x\geq y.\}$ give inverse maps to each other.
\end{Thm}

Next, we introduce the notion of mutation for upper sets.

\begin{Def}\cite[1.6]{Tom25d}
Let $I\in\I_X$ and take a minimal element $m\in I$. Then we define the {\it mutation} $\mu_{m}^-(I)$ of $I$ at $m$ as
\[\mu_m^-(I):=I\setminus\{m\}.\]
\end{Def}

Finally, we focus on the following explicit setting. Let $G$ be a finitely generated abelian group of rank one. Assume we are given elements $\vec{x}_0,\cdots,\vec{x}_d\in G$ satisfying (G1), (G2) and (G3). Put $\vec{p}:=\sum_{i=0}^d\vec{x}_i\in G$. Then $\mathbb{Z}$ acts on $G$ by $n\cdot\vec{g}:=\vec{g}+n\vec{p}$. This action satisfies the conditions (A1),(A2) and (A3). In this setting, we can describe $\J_G$ in the following way.

\begin{Prop}\label{GJX}
For a subset $J\subseteq G$, the following conditions are equivalent.
\begin{enumerate}
\item $J\in\J_G$
\item $J\subseteq G$ is a complete set of representatives for $G/\mathbb{Z}\vec{p}$ and for every $\vec{g}\in J$ and $0\leq i\leq d$, we have $\vec{g}+\vec{x}_i\in J\sqcup(J+\vec{p})$.
\end{enumerate}
\end{Prop}
\begin{proof}
(1)$\Rightarrow$(2) $J\subseteq G$ is a complete set of representatives by \cite[1.5]{Tom25d}. Take $\vec{g}\in J$ and $0\leq i\leq d$. Then there exists unique $n\in\mathbb{Z}$ with $\vec{g}+\vec{x}_i\in J+n\vec{p}$. If $n>1$, then $\vec{g}\geq\vec{g}-(\vec{p}-\vec{x}_i)\geq(\vec{g}+\vec{x}_i-n\vec{p})+\vec{p}$ holds, but this contradicts to $\vec{g},\vec{g}+\vec{x}_i-n\vec{p}\in J$. If $n<0$, then $\vec{g}+\vec{x}_i-n\vec{p}\geq\vec{g}+\vec{p}$ holds, but this is a contradiction for the same reason. Thus we obtain $n=0$ or $n=1$.

(2)$\Rightarrow$(1) Suppose that there exist $\vec{g},\vec{h}\in J$ with $\vec{h}\geq\vec{g}+\vec{p}$. Then by the definition of the partial order on $G$, there exist $a_0,\cdots, a_d\in\mathbb{Z}_{\geq0}$ such that $\vec{h}=\vec{g}+\vec{p}+\sum_{i=0}^da_i\vec{x}_i$. By our assumption, there exists $m\geq0$ such that $\vec{g}+\sum_{i=0}^da_i\vec{x}_i\in J+m\vec{p}$ holds. This means $\vec{g}+\vec{p}+\sum_{i=0}^da_i\vec{x}_i\in J+(m+1)\vec{p}$ holds, but this contradicts to $\vec{h}\in J$. Thus $J\in\widetilde{\J}_G$ holds. By \cite[1.5]{Tom25d}, we obtain $J\in\J_G$.
\end{proof}

\section{Combinatorics for higher representation infinite algebras of type \texorpdfstring{$\widetilde A$}{\~A}}

In this section, we discuss basic properties of cuts used to define higher representation infinite algebras of type $\widetilde{A}$. We introduce a viewpoint of upper sets and give a theorem which we call {\it cut-upper set correspondence} (Theorem \ref{cutupcorr}).

Remark that almost all the arguments in this section are parallel to \cite[Section 6]{Tom25d}, although neither set of results formally implies the other.

\subsection{Combinatorics of cuts}

Let $L:=\{v=(v_i)_{i=0}^d\in\mathbb{Z}^{d+1}\mid\sum_{i=0}^dv_i=0\}=\sum_{i=0}^d\mathbb{Z}\alpha_i\subseteq\mathbb{Z}^{d+1}$ be a $d$-dimensional lattice and $B\subseteq L$ a cofinite subgroup. Put $m:=\sharp(L/B)$.

First, we introduce a new object which we call cut detectors. This is an analogue of height functions.

\begin{Def}
Let $\gamma\in\mathbb Z_{\geq0}^{d+1}$ with $\sum_{i=0}^d\gamma_i=m$.
A map $f\colon L/B\to\mathbb Z$ is a \emph{cut detector} of type $\gamma$
if it satisfies the following conditions.
\begin{enumerate}
\item $f(0)=0$
\item For every $x\in L$ and $0\le i\le d$, we have $f(x+\alpha_i+B)\in\{f(x+B)+\gamma_i,f(x+B)+\gamma_i-m\}$.
\end{enumerate}
\end{Def}

We now prove that cut detectors correspond bijectively to cuts of $Q$.

\begin{Thm}\label{cutbij}
For $\gamma\in\mathbb Z_{\geq0}^{d+1}$ with $\sum_{i=0}^d\gamma_i=m$, the following sets are in bijection.
\begin{enumerate}
\item The set of cut detectors $f\colon L/B\to\mathbb{Z}$ of type $\gamma$.
\item The set of cuts of $Q$ of type $\gamma$.
\end{enumerate}
\end{Thm}

First, we show that a cut $C$ of $Q$ induces cut detectors of the same type.

\begin{Def}
Let $\gamma$ be the type of $C$. For $a\in\hat{Q}_1$ of type $i$, we define
\[f_C(a):=
  \begin{cases}
    \gamma_i & a\notin C,\\
    \gamma_i-m & a\in C.
  \end{cases}
\]
For a path $p=a_n\cdots a_1$ in $\hat{Q}$, we define
\[f_C(p):=\sum_{i=1}^nf_C(a_i).\]
\end{Def}

\begin{Rem}
For a path $p$ in $\hat{Q}$ of length $0$, we set $f_C(p)=0$.
\end{Rem}

The following can be shown in the same way as \cite[2.5,2.6]{DG}.

\begin{Lem}
For paths $p,q$ in $\hat{Q}$ with same sources and targets, we have $f_C(p)=f_C(q)$.
\end{Lem}

Thanks to this lemma, for $x\in L$, we can define
\[f_C(x):=f_C(p_x),\]
where $p_x$ is any path from $0$ to $x$.

\begin{Prop}
Our $f_C\colon L\to\mathbb{Z}$ induces a cut detector $f_C\colon L/B\to\mathbb{Z}$ of type $\gamma$.
\end{Prop}
\begin{proof}
It is enough to show that $f_C\colon L\to\mathbb{Z}$ is invariant under the action of $B$ on $L$. Take $x\in L$ and $y\in B$. Let $p_x$ be a path in $\hat{Q}$ from $0$ to $x$. Since $C$ is $B$-periodic, for the path $p_x+y$ from $y$ to $x+y$, we have $f_C(p_x)=f_C(p_x+y)$. Thus we obtain
\[f_C(x+y)=f_C(y)+f_C(p_x+y)=f_C(y)+f_C(x).\]
Therefore it is enough to show $f_C(y)=0$.

Our proof below is essentially the same as \cite[2.11]{DG}. Let $o_i$ be the order of $\alpha_i+B\in L/B$. First, we show $f_C(o_i\alpha_i)=0$. Consider the path $0\to\alpha_i\to\cdots\to o_i\alpha_i$, where each arrow is of type $i$ and put $\theta'_i:=\sharp\{1\leq j\leq o_i\mid((j-1)\alpha_i\to j\alpha_i)\in C\}$. Then we have $f_C(o_i\alpha_i)=o_i\gamma_i-\theta'_im$. Here, for any $x\in L$, we have
\[f_C(o_i\alpha_i)=f_C(x+o_i\alpha_i)-f_C(x)=f_C(x\to x+\alpha_i\to\cdots\to x+o_i\alpha_i).\]
This implies $\theta'_i=\sharp\{1\leq j\leq o_i\mid(x+(j-1)\alpha_i\to x+j\alpha_i)\in C\}$ holds. Take $x_1,\cdots,x_{\frac{m}{o_i}}\in L$ so that $\{x_l+B\}_l\subseteq L/B$ gives a complete set of representatives for $(L/B)/\mathbb{Z}(\alpha_i+B)$. Then each arrow of type $i$ in $Q$ appears exactly once in cycles
\[x_l\to x_l+\alpha_i\to\cdots\to x_l+o_i\alpha_i\ (1\leq l\leq\frac{m}{o_i}).\]
This means $\frac{m}{o_i}\theta'_i=\gamma_i$. Therefore we have
\[f_C(o_i\alpha_i)=o_i\gamma_i-\theta'_im=0.\]

Finally, consider arbitrary $y\in B$. Since $f_C(my)=mf_C(y)$, it is enough to show $f_C(my)=0$. If we write $y=\sum_{i=0}^dy_i\alpha_i$, then we have
\[f_C(my)=\sum_{i=0}^dy_i\frac{m}{o_i}f_C(o_i\alpha_i)=0.\qedhere\]
\end{proof}

Using this, we can recover \cite[2.15,2.16]{DG} easily.

\begin{Cor}
Let $C$ be a cut of $Q$ and $\gamma$ its type.
\begin{enumerate}
\item\cite[2.15]{DG} Take $(m_i)_{i=0}^d\in\mathbb{Z}^{d+1}$. If $\sum_{i=0}^dm_i\alpha_i\in B$ holds, then we have $\sum_{i=0}^dm_i\gamma_i\in m\mathbb{Z}$.
\item\cite[2.16]{DG} The cut $C$ is bounding if and only if $\gamma\in\mathbb{Z}^{d+1}_{>0}$ holds.
\end{enumerate}
\end{Cor}
\begin{proof}
(1) By the definition of $f_C$, we have $f_C(\sum_{i=0}^dm_i\alpha_i)-\sum_{i=0}^dm_i\gamma_i\in m\mathbb{Z}$. Since $f_C(\sum_{i=0}^dm_i\alpha_i)=0$, we get the conclusion.

(2) The necessity is obvious. We prove the sufficiency. Take $x,y\in L/B$. Observe that if there exists a path in $Q_C$ from $x$ to $y$, then we have $f_C(x)<f_C(y)$ by the definition of $f_C$. This proves the conclusion.
\end{proof}

Now we prove Theorem \ref{cutbij}.
\begin{proof}[Proof of Theorem \ref{cutbij}]
Let $f\colon L/B\to\mathbb{Z}$ be a cut detector of type $\gamma$. We define a subset $C_f\subseteq Q_1$: for an arrow $a\colon x+B\to x+\alpha_i+B$ in $Q$,
\[a\in C_f\Leftrightarrow f(x+\alpha_i+B)=f(x+B)+\gamma_i-m.\]
Summing the differences of $f$ around an elementary cycle $\omega$ gives
$0=\sum_i\gamma_i-m\#(C_f\cap\omega)$, so $C_f$ is a cut.
We show that its type is $\gamma$. Let $o_i$ be the order of $\alpha_i+B\in L/B$. Then we have
\[0=f(o_i\alpha_i+B)=\sum_{j=1}^{o_i}(f(j\alpha_i+B)-f((j-1)\alpha_i+B))=o_i\gamma_i-\theta'_im,\]
where $\theta'_i=\sharp\{1\leq j\leq o_i\mid((j-1)\alpha_i\to j\alpha_i)\in C_f\}$. Here, for any $x\in L$, we have
\[0=f(x+o_i\alpha_i+B)-f(x+B)=\sum_{j=1}^{o_i}(f(x+j\alpha_i+B)-f(x+(j-1)\alpha_i+B)).\]
This implies $\theta'_i=\sharp\{1\leq j\leq o_i\mid(x+(j-1)\alpha_i\to x+j\alpha_i)\in C_f\}$. Thus by taking a complete set of representatives for $(L/B)/\mathbb{Z}(\alpha_i+B)$, we can calculate that the number of the arrows in $C_f$ of type $i$ is
\[\frac{m}{o_i}\theta'_i=\gamma_i.\]

By construction, it is easy to check that $f_{C_f}=f$ and $C_{f_C}=C$ hold. This completes the proof.
\end{proof}

In \cite{DG}, the following theorem is proved by constructing an explicit cut which is periodic with respect to another cofinite subgroup of $L$.

\begin{Thm}\label{chartype}\cite[3.5]{DG}
For $\gamma=(\gamma_i)_{i=0}^d\in\mathbb{Z}_{\geq0}^{d+1}$, $\gamma$ is a type of a $B$-periodic cut if and only if both of the following conditions are satisfied.
\begin{enumerate}
\item $\sum_{i=0}^d\gamma_i=m$
\item For any $(m_i)_{i=0}^d\in\mathbb{Z}^{d+1}$ with $\sum_{i=0}^dm_i\alpha_i\in B$, we have $\sum_{i=0}^dm_i\gamma_i\in m\mathbb{Z}$.
\end{enumerate}
\end{Thm}

The necessity of these conditions has already been proved. In the next subsection, we give a new proof of the sufficiency by introducing cut-upper set correspondence.

\subsection{Cut-upper set correspondence}

Let $\gamma=(\gamma_i)_{i=0}^d\in\mathbb{Z}_{\geq0}^{d+1}$ be an integer vector satisfying the conditions (1) and (2) in Theorem \ref{chartype}. We define a group homomorphism $\Phi\colon\mathbb{Z}^{d+1}\to\mathbb{Z}\oplus L/B$ by
\[\Phi(e_i):=(\gamma_i,\alpha_i+B)\]
and put $G=G(B,\gamma):=\Im\Phi$ and $\vec{x}_i:=\Phi(e_i)\in G$. Observe that $\vec{p}:=\sum_{i=0}^d\vec{x}_i=(m,0)$ holds. Thus the composition $G\hookrightarrow\mathbb{Z}\oplus L/B\twoheadrightarrow L/B$ induces a group homomorphism $\phi\colon G/\mathbb{Z}\vec{p}\to L/B$.

\begin{Lem}
The group homomorphism $\phi\colon G/\mathbb{Z}\vec{p}\to L/B$ is an isomorphism.
\end{Lem}
\begin{proof}
The surjectivity follows from $\phi(\vec{x}_i+\mathbb{Z}\vec{p})=\alpha_i+B$. Take $\vec{g}=\Phi(v)\in G$ with $\phi(\vec{g}+\mathbb{Z}\vec{p})=0$. If we put $v=(m_i)_{i=0}^d$, then we have $\sum_{i=0}^dm_i\alpha_i\in B$. Thus by our assumption, there exists $n\in\mathbb{Z}$ with $\sum_{i=0}^dm_i\gamma_i=mn$. This implies $\vec{g}=n\vec{p}$.
\end{proof}

Let $\J$ be the set of complete sets of representatives $J\subseteq G$
for $G/\mathbb Z\vec p$ satisfying
\[
 \vec g+\vec x_i\in J\sqcup(J+\vec p)
 \qquad(\vec g\in J,\ 0\le i\le d).
\]
Let $\pi:=(G\hookrightarrow\mathbb{Z}\oplus L/B\twoheadrightarrow\mathbb{Z})$ denote the composition of natural group homomorphisms. The following proposition is key to prove Theorem \ref{chartype}.

\begin{Prop}\label{cutJcorr}
We have a surjective map
\[C(-)\colon\J\to\{\text{Cuts of $Q$ of type }\gamma\}.\]
For $J,J'\in\J$, $C(J)=C(J')$ holds if and only if $J=J'+n\vec{p}$ holds for some $n\in\mathbb{Z}$.
\end{Prop}
\begin{proof}
For $J\in\J$, let $C(J)\subseteq Q_1$ be a subset consisting of arrows which do not appear in the Cayley quiver of $J$. More precisely, we can describe $C(J)$ in terms of cut detectors as follows (see Theorem \ref{cutbij}). There exists a unique $n\in\mathbb{Z}$ with $n\vec{p}\in J$. Define a map $f_J\colon L/B\to\mathbb{Z}$ in the following way. For $x\in L$, take $\vec{g}\in J$ with $\phi(\vec{g}+\mathbb{Z}\vec{p})=x+B$. Then put $f_J(x+B):=\pi(\vec{g}-n\vec{p})=\pi(\vec{g})-nm$. Then we can check that $f_J$ is a cut detector of type $\gamma$ and put $C(J):=C_{f_J}$. By our definition, for $J,J'\in\J$, $f_J=f_{J'}$ holds if and only if $J=J'+n\vec{p}$ holds for some $n\in\mathbb{Z}$.

We prove the surjectivity of $C(-)$. We use Theorem \ref{cutbij}. Take a cut detector $f\colon L/B\to\mathbb{Z}$. Put $J:=\{\vec{g}\in G\mid\pi(\vec{g})=f(\phi(\vec{g}+\mathbb{Z}\vec{p}))\}\subseteq G$. Then we have $J\in\J$ and $f_J=f$.
\end{proof}

Now we can prove Theorem \ref{chartype}.

\begin{proof}[Proof of Theorem \ref{chartype}]
By Proposition \ref{cutJcorr}, it is enough to show $\J\neq\emptyset$. For example, if we put $J:=\{\vec{g}\in G\mid0\leq\pi(\vec{g})<m\}\subseteq G$, then  we have $J\in\J$.
\end{proof}

Finally, to state cut-upper set correspondence, we focus on the case of $\gamma\in\mathbb{Z}^{d+1}_{>0}$. In this case, our $G$ and $\vec{x}_i\in G$ satisfy the conditions (G1), (G2) and (G3).

\begin{Thm}\label{cutupcorr}(Cut-upper set correspondence)
Assume $\gamma\in\mathbb{Z}^{d+1}_{>0}$. Then we have a surjective map
\[C(-)\colon\I_G\to\{\text{Cuts of $Q$ of type }\gamma\}.\]
For $I,I'\in\I_G$, $C(I)=C(I')$ holds if and only if $I=I'+n\vec{p}$ holds for some $n\in\mathbb{Z}$.
\end{Thm}
\begin{proof}
By Proposition \ref{GJX}, we have $\J=\J_G$. Thus the assertion follows from Theorem \ref{upJX} and Proposition \ref{cutJcorr}.
\end{proof}

\subsection{Starting from $G$}\label{startG}

Let $G$ be a finitely generated abelian group of rank one. Assume we are given elements $\vec{x}_0,\cdots,\vec{x}_d\in G$ satisfying (G1), (G2), and (G3). In this subsection, we construct a cofinite subgroup $B\subseteq L$ and a type of a cut from $G$.

Put $\vec{p}:=\sum_{i=0}^d\vec{x}_i\in G$. Then we can define a surjective group homomorphism $L\to G/\mathbb{Z}\vec{p}$ sending $\alpha_i$ to $\vec{x}_i+\mathbb{Z}\vec{p}$. Let $B\subseteq L$ be the kernel of this homomorphism. Put $m:=\#(G/\mathbb{Z}\vec{p})=\#(L/B)$. Let $\pi':G\to G/G_{\rm tors}\cong\mathbb{Z}$ and put $m':=\pi'(\vec{p})$. Then since a surjective group homomorphism $G/\mathbb{Z}\vec{p}\to\mathbb{Z}/m'\mathbb{Z}$ is induced, we have $m\in m'\mathbb{Z}$. Define $\pi:=\frac{m}{m'}\pi'\colon G\to\mathbb{Z}$. If we put $\gamma_i:=\pi(\vec{x}_i)$, then we can check that our $\gamma=(\gamma_i)_{i=0}^d\in\mathbb{Z}^{d+1}_{>0}$ satisfies the conditions in Theorem \ref{chartype}.

Define a group homomorphism $\Phi\colon\mathbb{Z}^{d+1}\to\mathbb{Z}\oplus L/B$ and $\Phi'\colon\mathbb{Z}^{d+1}\to G$ as
\[\Phi(e_i):=(\gamma_i,\alpha_i+B),\Phi'(e_i)=\vec{x}_i.\]
For $v=(m_i)_{i=0}^d\in\mathbb{Z}^{d+1}$, $\Phi(v)=0$ if and only if $\sum_{i=0}^dm_i\gamma_i=0$ and $\sum_{i=0}^dm_i\alpha_i\in B$ holds. $\sum_{i=0}^dm_i\gamma_i=0$ is equivalent to $\Phi'(v)\in G_{\rm tors}$. $\sum_{i=0}^dm_i\alpha_i\in B$ is equivalent to $\Phi'(v)\in\mathbb{Z}\vec{p}$ holds. Since $G_{\rm tors}\cap\mathbb{Z}\vec{p}=0$, we obtain $\Ker\Phi=\Ker\Phi'$. Thus we have an isomorphism $G\cong\Im\Phi=G(B,\gamma)$.

\section{Tilting theory for smooth toric Fano stacks of Picard number one}

In this section, we give a classification of tilting bundles consisting of line bundles on smooth toric Fano stacks of Picard number one. Moreover, we prove that they are $d$-tilting and their endomorphism algebras are $d$-representation infinite algebras of type $\widetilde A$. Furthermore, we can prove that all $d$-representation infinite algebras of type $\widetilde A$ arise in this way. Using this derived equivalence, we give a new combinatorial description of $d$-APR tilting mutations, $d$-preprojective components and $d$-preinjective components of $d$-representation infinite algebras of type $\widetilde A$.

\subsection{Classification of tilting bundles consisting of line bundles}

A lattice $d$-simplex $P\subseteq N_{\mathbb R}$ with ordered vertices
$v_0,\ldots,v_d$ and the origin in its interior determines a rank-one
group $G$ with generators $\vec x_0,\ldots,\vec x_d$ satisfying
(G1)--(G3). Conversely, such a group with its ordered generators
determines the simplex up to lattice isomorphism by the construction
in Subsection~\ref{polytope}. Put $\vec p=\sum_{i=0}^d\vec x_i$.

Combining with the arguments in Subsection \ref{startG}, we obtain the following bijections.

\begin{Thm}\label{3corr}
We have bijections between the following three sets.
\begin{enumerate}
\item $\{(B,\gamma)\mid B\subseteq L\colon\text{cofinite subgroup},\gamma\in\mathbb{Z}_{>0}^{d+1}\text{ satisfies the conditions in Theorem \ref{chartype}}\}$
\item $\{(G,(\vec{x}_i)_{i=0}^d)\mid G\colon\text{finitely generated abelian group of rank one}, \\
\vec{x}_0,\cdots\vec{x}_d\in G\text{ satisfy (G1), (G2), and (G3)}\}/\cong$

Here, we write $(G,(\vec{x}_i)_i)\cong(G',(\vec{x}_i')_i)$ if there exists a group isomorphism $G\cong G'$ sending each $\vec{x}_i$ to $\vec{x}_i'$.
\item Lattice $d$-simplices in $N_{\mathbb R}$ with ordered vertices
and the origin in their interior, modulo $\operatorname{GL}(N)$.
\end{enumerate}
\end{Thm}

Let $\X:=\X(P)$ be the smooth toric Fano stack corresponding to the polytope $P$. Then we have $\Pic\X\cong G$. First, we compute the cohomology groups of line bundles on $\X$ in terms of the $G$-graded $k$-algebra $S=k[x_0,\cdots,x_d]$.

\begin{Lem}\label{calcohPic1}
For $\vec g\in G$ and $i\ge0$, we have
\[
H^i(\X,\O_{\X}(\vec g))\cong
\begin{cases}
  S_{\vec g} & i=0, \\
  DS_{-\vec p-\vec g} & i=d, \\
  0 & i\ne0,d.
  \end{cases}
\]
\end{Lem}
\begin{proof}
We use Proposition \ref{calcoh}. Take $c=(c_i)_{i=0}^d\in\mathbb Z^{d+1}$. Recall that the topological space $X_c$ is defined as a union of proper faces of $P$. If $c\in\mathbb Z_{\ge0}^{d+1}$ holds, then we have $X_c=\partial P$, which is homeomorphic to the $d-1$ dimensional sphere $S^{d-1}$. If not, then $X_c$ itself is a simplex. In addition, $X_c=\emptyset$ holds exactly when $c\in\mathbb Z_{<0}^{d+1}$. This together with Proposition \ref{calcoh} prove the assertion.
\end{proof}

\begin{Rem}
Lemma \ref{calcohPic1} also follows from the description of
$\Coh\X$ as $\mod^G\!S/\fl^G\!S$ and Serre duality on $\D^b(\Coh\X)$.
\end{Rem}

As an immediate corollary, among vector bundles consisting of line bundles, we can determine pretilting one.

\begin{Cor}\label{pretiltPic1}
For a finite subset $J\subseteq G$, the following conditions are equivalent.
\begin{enumerate}
\item $\E(J):=\bigoplus_{\vec{g}\in J}\O_\X(\vec{g})\in\D^b(\Coh\X)$ is a pretilting bundle.
\item $J\in\widetilde{\J}_G$.
\end{enumerate}
\end{Cor}
\begin{proof}
By Lemma \ref{calcohPic1}, $\E(J)$ is a pretilting bundle if and only if there is no $\vec g,\vec h\in J$ with $\vec g-\vec h\le-\vec p$.
\end{proof}

With these preparations, we give a classification of tilting bundles consisting of line bundles on smooth toric Fano stacks of Picard number one.

\begin{Thm}\label{classfitiltrk1}
Let $P\subseteq N_\mathbb{R}$ be a lattice $d$-simplex containing the origin as an interior point and $\X:=\X(P)$. Then for a subset $J\subseteq G\cong\Pic\X$, the following conditions are equivalent.
\begin{enumerate}
\item $\E(J):=\bigoplus_{\vec{g}\in J}\O_\X(\vec{g})\in\D^b(\Coh\X)$ is a tilting bundle.
\item $J\in\J_G$
\end{enumerate}
These bundles are $d$-tilting. By Theorem \ref{upJX}, they correspond
bijectively to non-trivial upper sets in $G$.
\end{Thm}
\begin{proof}
(2)$\Rightarrow$(1) By Corollary \ref{pretiltPic1}, we have only to check $\D^b(\Coh\X)=\thick\E(J)$. Since $\Coh\X$ is obtained as a Serre quotient of $\mod^G\!S$, every object in $\Coh\X$ can be lifted to some $M\in\mod^G\!S$. Since $\pd^G\!M\le\gl^G\!S=d+1<\infty$ holds, $M$ can be resolved by a bounded complex of $\proj^G\!S$. Since $S(\vec g)\in\mod^G\!S$ is sent to $\O_{\X}(\vec g)\in\Coh\X$, this implies $\D^b(\Coh\X)=\thick\{\O_{\X}(\vec g)\mid\vec g\in G\}$. Thus it is enough to show $\O_{\X}(\vec g)\in\thick\E(J)$ for every $\vec g\in G$. Observe that we have a Koszul complex
\[0\to\O_{\X}(-\vec p)\to\bigoplus_{i=0}^d\O_{\X}(-\vec p+\vec x_i)\to\cdots\to\bigoplus_{i=0}^d\O_{\X}(-\vec x_i)\to\O_{\X}\to0.\]
This exact sequence shows that we have 
\[\O_{\X}\in\thick\{\O_{\X}(-\vec h)\mid0<\vec h\le\vec p\}\quad\text{and}\quad\O_{\X}(-\vec p)\in\thick\{\O_{\X}(-\vec h)\mid0\le\vec h<\vec p\}.\]
Take $I\in\I_G$ with $J=J(I)$ and a minimal element $\vec g\in I$. Since any $\vec h\in G$ with $\vec g\le\vec h<\vec g+\vec p$ satisfies $\vec h\in J$, we have
\[\O_{\X}(\vec g+\vec p)\in\thick\{\O_{\X}(\vec h)\mid\vec g\le\vec h<\vec g+\vec p\}\subseteq\thick\E(J).\]
This with the dual argument shows
\[\thick\E(J(\mu_{\vec g}^-(I)))=\thick(\E(J)).\]
Therefore by \cite[1.10]{Tom25d}, we obtain the assertion.

(1)$\Rightarrow$(2) By Corollary \ref{pretiltPic1} and \cite[1.4]{Tom25d},
there is $J'\in\J_G$ containing $J$. Both $\E(J)$ and $\E(J')$ are
basic tilting bundles, so each has as many summands as the rank of
$K_0(\D^b(\Coh\X))$. Thus $J=J'$.

For $J\in\J_G$, put $A=\End_\X(\E(J))$. The order on $G$ makes $A$
a triangular matrix algebra with diagonal entries $k$, so $A$ has finite
global dimension. Under the tilting equivalence, tensoring with
$\O_\X(\vec p)$ corresponds to $\nu_d^{-1}$ by Serre duality.
Lemma \ref{calcohPic1} gives $\Ext_\X^{>0}(\E(J),\E(J)(\vec p))=0$,
since $\vec g\ngeq\vec h+2\vec p$ for $\vec g,\vec h\in J$.
Hence $\nu_d^{-1}A=\RHom_A(DA,A)[d]$ is concentrated in degree zero.
It follows that $\gl A=\pd_A DA\le d$.
\end{proof}

\begin{Rem}
The implication (2)$\Rightarrow$(1) in Theorem \ref{classfitiltrk1} is a commutative toric instance of the Beilinson-type theorem proved in Appendix \ref{appenBei}. Indeed, Theorem \ref{GBei} shows that, for a $G_{\ge0}$-graded dg ring with generalized Gorenstein parameter $p$, the window $J(I)=I\cap(I^c+p)$ gives a generator of the corresponding graded cluster category. Theorem \ref{MMcorr} further upgrades this statement, under Calabi--Yau assumptions, to a $G$-graded Minamoto--Mori correspondence. Thus the upper-set classification in Theorem \ref{classfitiltrk1} should be regarded as the toric rank-one realization of a more general non-commutative Beilinson phenomenon.
\end{Rem}

Next, we see that the endomorphism algebras of tilting bundles obtained in Theorem \ref{classfitiltrk1} are $d$-representation infinite algebras of type $\widetilde A$.

\begin{Thm}\label{enddtiltrk1}
Let $P\subseteq N_\mathbb{R}$ be a lattice $d$-simplex containing the origin as an interior point and $\X:=\X(P)$. Let $B\subseteq L$ be a cofinite subgroup corresponding to $P$. Then for $J\in\J_G$, we have
\[\End_\X(\E(J))\cong A(B,C(J)).\]
In particular, the following statements hold.
\begin{enumerate}
\item The endomorphism algebra of a tilting bundle consisting of line bundles on a $d$-dimensional smooth toric Fano stack of Picard number one is a $d$-representation infinite algebra of type $\widetilde A$.
\item Conversely, every $d$-representation infinite algebra of type $\widetilde A$ can be obtained in this way.
\end{enumerate}
\end{Thm}
\begin{proof}
By Lemma \ref{calcohPic1}, morphisms between the summands indexed by
$\vec g,\vec h\in J$ have a basis of monomials of degree $\vec h-\vec g$,
and composition is multiplication. For any factorization of such a monomial into variables, the intermediate
degrees lie between $\vec g$ and $\vec h$ and hence in $J$. Thus each
monomial is represented by a path of the cut quiver. Any two paths
representing the same monomial differ by adjacent interchanges of
variables, all within $J$. The commutativity relations therefore give
$\End_\X(\E(J))\cong A(B,C(J))$. The last assertions follow from
Theorems \ref{3corr} and \ref{classfitiltrk1}, the cut--upper set
correspondence, and Proposition \ref{SerreRI}.
\end{proof}

In particular, the construction gives a proof over any field that
$A(B,C)$ is $d$-representation infinite.

As an immediate corollary, we obtain the following. We strengthen this corollary in Theorem \ref{dAPRconnAtilde}.

\begin{Cor}
Let $B\subseteq L$ be a cofinite subgroup and $\gamma\in\mathbb{Z}^{d+1}_{>0}$ a vector satisfying the conditions in Theorem \ref{chartype}. Take two cuts $C_1,C_2\subseteq Q_1$ of common type $\gamma$. Then the two algebras $A(B,C_1)$ and $A(B,C_2)$ are derived equivalent.
\end{Cor}

Using the derived equivalence $\D^b(\Coh\X)\simeq\per A$ obtained by Theorem \ref{classfitiltrk1}, we can give a description of the $d$-preprojective component and the $d$-preinjective component $\P,\I\subseteq\mod A$. Remark that we have the following commutative diagram obtained by the uniqueness of the Serre functor.
\[\xymatrix{
\D^b(\Coh\X) \ar[r]_\simeq \ar[d]_{(\vec{p})} & \per A \ar[d]^{\nu_d^{-1}}\\
\D^b(\Coh\X) \ar[r]_\simeq & \per A
}\]

\begin{Prop}
Take $I\in\I_G$. Put $A:=\End_\X(\E(J(I)))$ in the notation of Theorem \ref{classfitiltrk1}. Then the derived equivalence $\D^b(\Coh\X)\simeq\per A$ restricts to equivalences
\[\add\{\O_\X(\vec{g})\mid\vec{g}\in I\}\simeq\P\text{ and }\add\{\O_\X(\vec{g})\mid\vec{g}\in I^c\}\simeq\I[-d].\]
In particular, we obtain an equivalence
\[\add\{\O_\X(\vec{g})\mid\vec{g}\in G\}\simeq\I[-d]\vee\P.\]
\end{Prop}
\begin{proof}
The assertion follows from the above commutative diagram.
\end{proof}

Next, we investigate the $d$-APR tilts \cite{IO11} of $d$-representation infinite algebras of type $\widetilde A$ through their geometric models. First, we give a proof of the following folklore statement: the endomorphism algebra of a $d$-APR tilting module of a $d$-representation infinite algebra of type $\widetilde A$ is again a $d$-representation infinite algebra of type $\widetilde A$ with the same $B$ and $\gamma$.

\begin{Thm}\label{dAPRAtilde}
Let $B\subseteq L$ be a cofinite subgroup and $\gamma\in\mathbb{Z}^{d+1}_{>0}$ a vector satisfying the conditions in Theorem \ref{chartype}. Take $I\in\I_G$ and put $A:=A(B,C(I))$. Take a minimal element $\vec{m}\in I$ and let $T:=\nu_d^{-1}(e_{\vec{m}}A)\oplus\bigoplus_{\vec{g}\in J(I)\setminus\{\vec{m}\}}e_{\vec{g}}A\in\mod A$ be the $d$-APR tilting module with respect to $e_{\vec{m}}A$. Then we have
\[\End_A(T)\cong A(B,C(\mu^-_{\vec{m}}(I))).\]
\end{Thm}
\begin{proof}
By minimality of $\vec m$, the projective $e_{\vec m}A$ is simple.
Under the tilting equivalence, $T$ corresponds to the bundle obtained
from $\E(J(I))$ by replacing $\O_\X(\vec m)$ with
$\O_\X(\vec m+\vec p)$. Its summand degrees form $J(\mu^-_{\vec m}(I))$,
so Theorem \ref{enddtiltrk1} identifies its endomorphism algebra.
\end{proof}

We emphasize that Theorem \ref{dAPRAtilde} is difficult to prove without using these geometric models. Thanks to Theorem \ref{dAPRAtilde}, we can prove that all $d$-representation infinite algebras of type $\widetilde A$ having same $B$ and $\gamma$ are connected by a finite sequence of iterated $d$-APR tilts. Although a related statement appears in \cite[5.2, 5.12]{DG}, our proof below is independent and uses combinatorics of upper sets.

\begin{Thm}\label{dAPRconnAtilde}
Let $B\subseteq L$ be a cofinite subgroup and $\gamma\in\mathbb{Z}^{d+1}_{>0}$ a vector satisfying the conditions in Theorem \ref{chartype}. Take two cuts $C_1,C_2\subseteq Q_1$ of common type $\gamma\in\mathbb{Z}_{>0}^{d+1}$. Then the two algebras $A(B,C_1)$ and $A(B,C_2)$ are connected by a finite sequence of iterated $d$-APR tilts.
\end{Thm}
\begin{proof}
Choose $I_1,I_2\in\I_{G(B,\gamma)}$ inducing $C_1,C_2$ by
Theorem \ref{cutupcorr}. Each upper set is a tail in each of the finitely
many $\mathbb Z\vec p$-orbits. Translating $I_2$ by a sufficiently large
multiple of $\vec p$ puts it inside $I_1$ and preserves $C_2$.
By \cite[1.9]{Tom25d}, one can reach the translated $I_2$ from $I_1$
by successively deleting minimal elements. Apply Theorem \ref{dAPRAtilde}
at each step.
\end{proof}

\subsection{Examples}

We see several examples. First, as the simplest case, we see that we can obtain a classification of tilting bundles consisting of line bundles on the projective space $\mathbb{P}^d$. 

\begin{Ex}\label{ExBei}
Put $G:=\mathbb{Z}$ and $\vec{x}_0=\cdots=\vec{x}_d=1\in G$. Then the resulting toric stack $\X$ is isomorphic to the projective space $\mathbb{P}^d$. If we equip $G$ with our partial order, then the quiver of $G$ becomes the following.
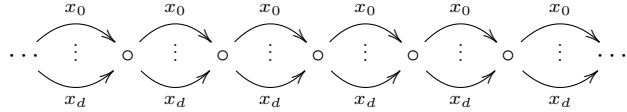
\[\xymatrix{
\cdots \ar@/^1pc/[r]^{x_0}_{\scalebox{0.7}{\vdots}}\ar@/_1pc/[r]_{x_d} & \circ \ar@/^1pc/[r]^{x_0}_{\scalebox{0.7}{\vdots}}\ar@/_1pc/[r]_{x_d} & \circ \ar@/^1pc/[r]^{x_0}_{\scalebox{0.7}{\vdots}}\ar@/_1pc/[r]_{x_d} & \circ \ar@/^1pc/[r]^{x_0}_{\scalebox{0.7}{\vdots}}\ar@/_1pc/[r]_{x_d} & \circ \ar@/^1pc/[r]^{x_0}_{\scalebox{0.7}{\vdots}}\ar@/_1pc/[r]_{x_d} & \circ \ar@/^1pc/[r]^{x_0}_{\scalebox{0.7}{\vdots}}\ar@/_1pc/[r]_{x_d} & \cdots
}\]
Then there is the following just one kind of non-trivial upper set in $G$ up to translations.
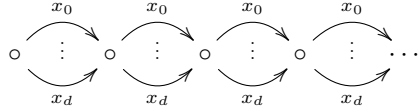
\[\xymatrix{
\circ \ar@/^1pc/[r]^{x_0}_{\scalebox{0.7}{\vdots}}\ar@/_1pc/[r]_{x_d} & \circ \ar@/^1pc/[r]^{x_0}_{\scalebox{0.7}{\vdots}}\ar@/_1pc/[r]_{x_d} & \circ \ar@/^1pc/[r]^{x_0}_{\scalebox{0.7}{\vdots}}\ar@/_1pc/[r]_{x_d} & \circ \ar@/^1pc/[r]^{x_0}_{\scalebox{0.7}{\vdots}}\ar@/_1pc/[r]_{x_d} & \cdots
}\]
Remark that we have $\vec{p}=d+1$. Therefore there is the following just one kind of tilting bundle up to translations where there are $d+1$ vertices.
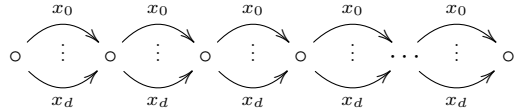
\[\xymatrix{
\circ \ar@/^1pc/[r]^{x_0}_{\scalebox{0.7}{\vdots}}\ar@/_1pc/[r]_{x_d} & \circ \ar@/^1pc/[r]^{x_0}_{\scalebox{0.7}{\vdots}}\ar@/_1pc/[r]_{x_d} & \circ \ar@/^1pc/[r]^{x_0}_{\scalebox{0.7}{\vdots}}\ar@/_1pc/[r]_{x_d} & \circ \ar@/^1pc/[r]^{x_0}_{\scalebox{0.7}{\vdots}}\ar@/_1pc/[r]_{x_d} & \cdots \ar@/^1pc/[r]^{x_0}_{\scalebox{0.7}{\vdots}}\ar@/_1pc/[r]_{x_d} & \circ
}\]
\end{Ex}

Next, we see that even when $d=1$, Theorem \ref{classfitiltrk1} gives us a new description of APR tilting mutations.
\begin{Ex}($d=1$)
(1) Put $G:=\mathbb{Z}$ and $\vec{x}=2, \vec{y}=3\in G$. Then the resulting toric stack $\X$ is isomorphic to the weighted projective stack $\mathbb{P}(2,3)$. If we equip $G$ with our partial order, then the quiver of $G$ becomes the following.
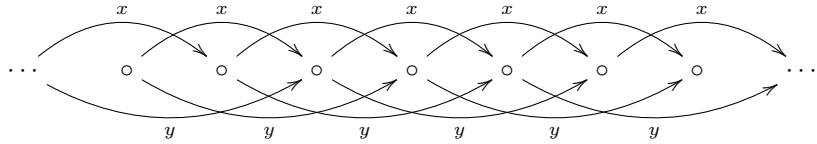
\[\xymatrix{
\cdots \ar@/^18pt/[rr]^x \ar@/^-18pt/[rrr]_y & \circ \ar@/^18pt/[rr]^x \ar@/^-18pt/[rrr]_y & \circ \ar@/^18pt/[rr]^x \ar@/^-18pt/[rrr]_y & \circ \ar@/^18pt/[rr]^x \ar@/^-18pt/[rrr]_y & \circ \ar@/^18pt/[rr]^x \ar@/^-18pt/[rrr]_y & \circ \ar@/^18pt/[rr]^x \ar@/^-18pt/[rrr]_y & \circ \ar@/^18pt/[rr]^x & \circ & \cdots
}\]
Then there are the following two kinds of non-trivial upper sets in $G$ up to translations.
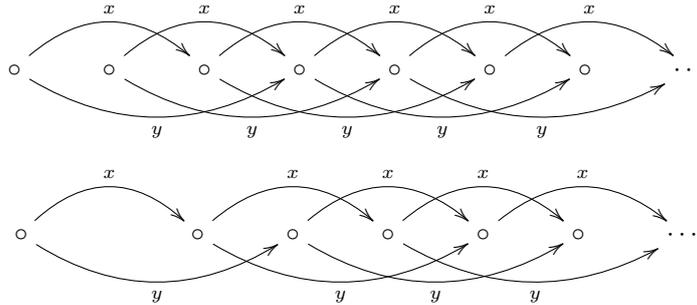
\[\xymatrix{
\circ \ar@/^18pt/[rr]^x \ar@/^-18pt/[rrr]_y & \circ \ar@/^18pt/[rr]^x \ar@/^-18pt/[rrr]_y & \circ \ar@/^18pt/[rr]^x \ar@/^-18pt/[rrr]_y & \circ \ar@/^18pt/[rr]^x \ar@/^-18pt/[rrr]_y & \circ \ar@/^18pt/[rr]^x \ar@/^-18pt/[rrr]_y & \circ \ar@/^18pt/[rr]^x & \circ & \cdots
}\]
\[\xymatrix{
\circ \ar@/^18pt/[rr]^x \ar@/^-18pt/[rrr]_y & & \circ \ar@/^18pt/[rr]^x \ar@/^-18pt/[rrr]_y & \circ \ar@/^18pt/[rr]^x \ar@/^-18pt/[rrr]_y & \circ \ar@/^18pt/[rr]^x \ar@/^-18pt/[rrr]_y & \circ \ar@/^18pt/[rr]^x & \circ & \cdots
}\]
Remark that we have $\vec{p}=5$. Therefore there are the following two kinds of tilting bundles up to translations. Observe that by mutations of non-trivial upper sets in $G$, they are mutated to each other, which correspond to APR tilting mutations.
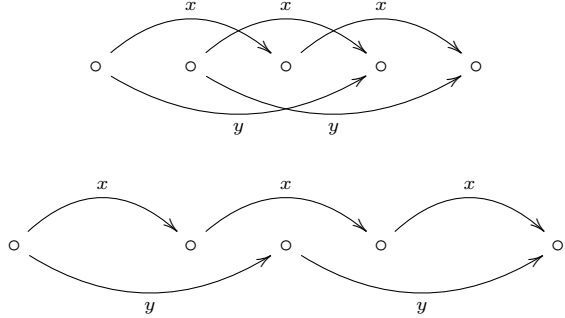
\[\xymatrix{
\circ \ar@/^18pt/[rr]^x \ar@/^-18pt/[rrr]_y & \circ \ar@/^18pt/[rr]^x \ar@/^-18pt/[rrr]_y & \circ \ar@/^18pt/[rr]^x & \circ & \circ
}\]
\[\xymatrix{
\circ \ar@/^18pt/[rr]^x \ar@/^-18pt/[rrr]_y & & \circ \ar@/^18pt/[rr]^x & \circ \ar@/^-18pt/[rrr]_y & \circ \ar@/^18pt/[rr]^x & & \circ
}\]

(2) Put $G:=\mathbb{Z}\oplus(\mathbb{Z}/2\mathbb{Z})$ and $\vec{x}=(1,0), \vec{y}=(1,1)\in G$. If we equip $G$ with our partial order, then the quiver of $G$ becomes the following.
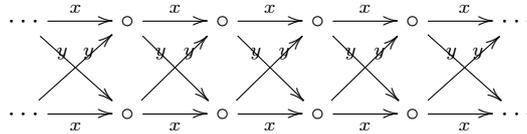
\[\xymatrix{
\cdots \ar[r]^x \ar[dr]^y & \circ \ar[r]^x \ar[dr]^y & \circ \ar[r]^x \ar[dr]^y & \circ \ar[r]^x \ar[dr]^y & \circ \ar[r]^x \ar[dr]^y & \cdots \\
\cdots \ar[r]_x \ar[ur]^y & \circ \ar[r]_x \ar[ur]^y & \circ \ar[r]_x \ar[ur]^y & \circ \ar[r]_x \ar[ur]^y & \circ \ar[r]_x \ar[ur]^y & \cdots
}\]
Then there are the following two kinds of non-trivial upper sets in $G$ up to translations.
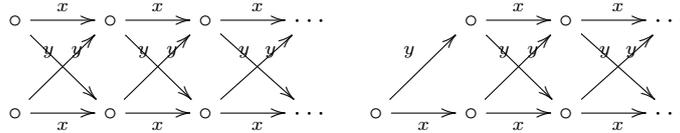
\[\begin{array}{c c}
\xymatrix{
\circ \ar[r]^x \ar[dr]^y & \circ \ar[r]^x \ar[dr]^y & \circ \ar[r]^x \ar[dr]^y & \cdots \\
\circ \ar[r]_x \ar[ur]^y & \circ \ar[r]_x \ar[ur]^y & \circ \ar[r]_x \ar[ur]^y & \cdots
}&\xymatrix{
 & \circ \ar[r]^x \ar[dr]^y & \circ \ar[r]^x \ar[dr]^y & \cdots\\
\circ \ar[r]_x \ar[ur]^y & \circ \ar[r]_x \ar[ur]^y & \circ \ar[r]_x \ar[ur]^y & \cdots
}\end{array}\]
Remark that we have $\vec{p}=(2,1)$. Therefore there are the following two kinds of tilting bundles up to translations. Observe that by mutations of non-trivial upper sets in $G$, they are mutated to each other, which correspond to APR tilting mutations.
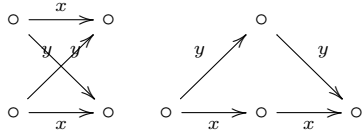
\[\begin{array}{c c}
\xymatrix{
\circ \ar[r]^x \ar[dr]^y & \circ \\
\circ \ar[r]_x \ar[ur]^y & \circ
}&\xymatrix{
 & \circ \ar[dr]^y \\
\circ \ar[r]_x \ar[ur]^y & \circ \ar[r]_x & \circ
}\end{array}\]
\end{Ex}

Finally, we see a $2$-dimensional example which cannot be obtained as a weighted projective space in the sense of \cite{HIMO}.

\begin{Ex}($d=2$)
Put $G:=\mathbb{Z}\oplus(\mathbb{Z}/2\mathbb{Z}), \vec{x}=\vec{y}=(1,0), \vec{z}=(1,1)\in G$. If we equip $G$ with our partial order, then the quiver of $G$ becomes the following.
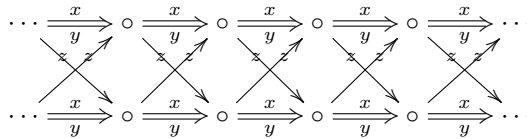
\[\xymatrix{
\cdots \ar@2[r]^x_y \ar[dr]^z & \circ \ar@2[r]^x_y \ar[dr]^z & \circ \ar@2[r]^x_y \ar[dr]^z & \circ \ar@2[r]^x_y \ar[dr]^z & \circ \ar@2[r]^x_y \ar[dr]^z & \cdots \\
\cdots \ar@2[r]^x_y \ar[ur]^z & \circ \ar@2[r]^x_y \ar[ur]^z & \circ \ar@2[r]^x_y \ar[ur]^z & \circ \ar@2[r]^x_y \ar[ur]^z & \circ \ar@2[r]^x_y \ar[ur]^z & \cdots
}\]
Then there are the following two kinds of non-trivial upper sets in $G$ up to translations.
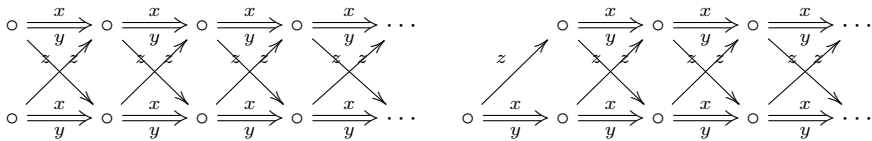
\[\begin{array}{c c}
\xymatrix{
\circ \ar@2[r]^x_y \ar[dr]^z & \circ \ar@2[r]^x_y \ar[dr]^z & \circ \ar@2[r]^x_y \ar[dr]^z & \circ \ar@2[r]^x_y \ar[dr]^z & \cdots \\
\circ \ar@2[r]^x_y \ar[ur]^z & \circ \ar@2[r]^x_y \ar[ur]^z & \circ \ar@2[r]^x_y \ar[ur]^z & \circ \ar@2[r]^x_y \ar[ur]^z & \cdots
}&\xymatrix{
 & \circ \ar@2[r]^x_y \ar[dr]^z & \circ \ar@2[r]^x_y \ar[dr]^z & \circ \ar@2[r]^x_y \ar[dr]^z & \cdots \\
\circ \ar@2[r]^x_y \ar[ur]^z & \circ \ar@2[r]^x_y \ar[ur]^z & \circ \ar@2[r]^x_y \ar[ur]^z & \circ \ar@2[r]^x_y \ar[ur]^z & \cdots
}\end{array}\]
Remark that we have $\vec{p}=(3,1)$. Therefore there are the following two kinds of tilting bundles up to translations. Observe that by mutations of non-trivial upper sets in $G$, they are mutated to each other, which correspond to $2$-APR tilting mutations.
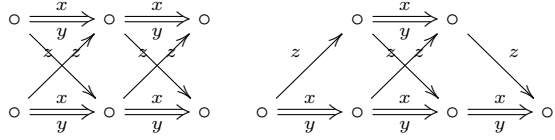
\[\begin{array}{c c}
\xymatrix{
\circ \ar@2[r]^x_y \ar[dr]^z & \circ \ar@2[r]^x_y \ar[dr]^z & \circ \\
\circ \ar@2[r]^x_y \ar[ur]^z & \circ \ar@2[r]^x_y \ar[ur]^z & \circ
}&\xymatrix{
 & \circ \ar@2[r]^x_y \ar[dr]^z & \circ \ar[dr]^z \\
\circ \ar@2[r]^x_y \ar[ur]^z & \circ \ar@2[r]^x_y \ar[ur]^z & \circ \ar@2[r]^x_y & \circ
}\end{array}\]
\end{Ex}

\section{Higher representation infinite algebras of type \texorpdfstring{$\widetilde A\widetilde A$}{\~A\~A}}\label{doubleAtilde}

For $l,l'\ge1$ and $d=l+l'$, we construct $d$-representation infinite
algebras of type $\widetilde A\widetilde A$ as degree-zero parts of toric
NCCRs. The elementary-cycle condition on a cut of $Q(C)$ makes the path
relations homogeneous. We classify these cuts by complete sets of
representatives in lifts of $H$ and, for interior types, by upper sets.

\subsection{Definition of $Q(C)$}

We use the quiver $Q(C)$ of \cite{Tom25d} for toric NCCRs of
Gorenstein singularities with divisor class group of rank one.

Let $e_0,\ldots,e_l$ be the standard basis of $\mathbb Z^{l+1}$, and put
\[
 \alpha_i:=e_i-e_{i-1}\quad(1\le i\le l),\qquad
 \alpha_0:=e_0-e_l,\qquad
 L:=\left\{(v_i)_{i=0}^l\in\mathbb Z^{l+1}\ \middle|\
                         \sum_{i=0}^l v_i=0\right\}.
\]
Then $L=\sum_{i=0}^l\mathbb Z\alpha_i$ has rank $l$, and the only
relation among these generators is $\sum_{i=0}^l\alpha_i=0$.
Define $L'$ and $\alpha'_0,\ldots,\alpha'_{l'}$ analogously.
The primitive quiver $\widehat Q$ has vertex set $L\oplus L'$ and arrows
\[
 x\longrightarrow x+\alpha_i\quad(0\le i\le l),\qquad
 x\longrightarrow x+\alpha'_{i'}\quad(0\le i'\le l')
\]
for every $x\in L\oplus L'$.
These arrows have types $\alpha_i$ and $\alpha'_{i'}$, respectively.

Let $H$ be a finitely generated abelian group of rank one, and fix
$\pi:H\to H/H_{\rm tors}\cong\mathbb Z$.
Choose elements
\[
 \overline{\vec x}_0,\ldots,\overline{\vec x}_l,
 \overline{\vec x}'_0,\ldots,\overline{\vec x}'_{l'}\in H
\]
satisfying the following conditions.
\begin{enumerate}
\item[(H1)] $H=\sum_{i=0}^l\mathbb{Z}\overline{\vec{x}}_i+\sum_{i'=0}^{l'}\mathbb{Z}\overline{\vec{x}}'_{i'}$
\item[(H2)] $\sum_{i=0}^l\overline{\vec{x}}_i+\sum_{i'=0}^{l'}\overline{\vec{x}}'_{i'}=0$
\item[(H3)] $\pi(\overline{\vec{x}}_i)>0,\pi(\overline{\vec{x}}'_{i'})<0\ (0\le i\le l,0\le i'\le l')$
\end{enumerate}
Up to isomorphism, these data correspond to $d$-dimensional simplicial
lattice polytopes $P\subseteq N_{\mathbb R}$ with the two ordered families
of vertices, up to affine lattice isomorphism.

Order $H$ by the monoid generated by the elements $\overline{\vec x}_i$
and $-\overline{\vec x}'_{i'}$. For a non-trivial upper set $I\subseteq H$, put
\[
 \vec s:=\sum_i\overline{\vec x}_i=-\sum_{i'}\overline{\vec x}'_{i'},
 \qquad J:=I\setminus(I+\vec s).
\]
Translation by $\vec s$ satisfies (A1)--(A3), since $\vec s>0$ and $2\vec s$
is the sum of the generators of the order. Thus $J$ is a complete set of
representatives for $H/\mathbb Z\vec s$. The surjection
\[
 \Psi:L\oplus L'\longrightarrow H/\mathbb Z\vec s,\qquad
 \alpha_i\longmapsto\overline{\vec x}_i+\mathbb Z\vec s,\qquad
 \alpha'_{i'}\longmapsto\overline{\vec x}'_{i'}+\mathbb Z\vec s
\]
has cofinite kernel $B$. Let $h_x\in J$ represent $\Psi(x)$.
Put the arrow $x\to x+\alpha_i$ in $\widehat C$ if
$h_x+\overline{\vec x}_i\notin J$, and put $x\to x+\alpha'_{i'}$ in
$\widehat C$ if $h_x+\overline{\vec x}'_{i'}\notin J$.
This gives a $B$-periodic cut of $\widehat Q$ \cite[6.1,6.12]{Tom25d}.
A cycle using exactly the types of one family, each once, contains one
arrow of $\widehat C$. A cycle using all primitive types once contains
equally many arrows of $\widehat C$ from the two families.

Form $\widehat Q(\widehat C)$ by deleting the arrows of $\widehat C$ and
adding a diagonal arrow $x\to x+\alpha_i+\alpha'_{i'}$ for each pair of
types whose corresponding primitive arrows out of $x$ both lie in
$\widehat C$. Put $Q:=\widehat Q/B$, $C:=\widehat C/B$, and
$Q(C):=\widehat Q(\widehat C)/B$. We identify its vertex set
$(L\oplus L')/B$ with $J$ via $\Psi$.

For the path labels, put
\[
 F:=\bigoplus_{i=0}^l\mathbb Ze_i\oplus
       \bigoplus_{i'=0}^{l'}\mathbb Ze'_{i'},\qquad
 \mathbf p_+:=\sum_{i=0}^l e_i,\qquad
 \mathbf p_-:=\sum_{i'=0}^{l'}e'_{i'},\qquad
 \mathbf p:=\mathbf p_++\mathbf p_-.
\]
As in Subsection \ref{polytope}, let $\deg_H:F\to H$ send $e_i$ to
$\overline{\vec x}_i$ and $e'_{i'}$ to $\overline{\vec x}'_{i'}$.
We also use the displacement homomorphism
\[
 \operatorname{disp}:F\longrightarrow L\oplus L',\qquad
 \operatorname{disp}(e_i)=\alpha_i,\quad \operatorname{disp}(e'_{i'})=\alpha'_{i'}.
\]
Thus $\deg_H(\mathbf p)=0$ and
$\Ker\operatorname{disp}=\mathbb Z\mathbf p_+\oplus\mathbb Z\mathbf p_-$.
Give a surviving primitive arrow of type $\alpha_i$ or $\alpha'_{i'}$
the label $e_i$ or $e'_{i'}$, respectively, and give a diagonal arrow
of type $\alpha_i+\alpha'_{i'}$ the label $e_i+e'_{i'}$.
The label $\lambda(\gamma)\in F_{\ge0}$ of a path is the sum of its
arrow labels.  We use the same notation for paths in the covering quiver
$\widehat Q(\widehat C)$.  For such a path $\gamma:x\to y$, we have
\[
 \operatorname{disp}(\lambda(\gamma))=y-x.
\]
A cycle $\omega$ in either quiver is called an \emph{elementary cycle}
if $\lambda(\omega)=\mathbf p$.

By the monomial realization in \cite[6.14]{Tom25d} and the
path construction in its proof, for
$h,h'\in J$ and $a\in F_{\ge0}$,
\begin{equation}\label{QClabelcriterion}
 \text{there is a path $\gamma:h\to h'$ with $\lambda(\gamma)=a$}
 \quad\Longleftrightarrow\quad \deg_H(a)=h'-h .
\end{equation}
Every element of $H$ has a nonnegative representative in $F$, since
adding a sufficiently large multiple of $\mathbf p$ leaves its
$H$-degree unchanged. Thus $Q(C)$ is strongly connected.
If a path has label $a\le\mathbf p$, then \eqref{QClabelcriterion}
gives a return path with label $\mathbf p-a$, since $\deg_H(\mathbf p)=0$.
In the covering quiver the lift of the return path closes the cycle,
since $\operatorname{disp}(\mathbf p)=0$.

\subsection{Cuts and homogeneous relations}

For $C'\subseteq Q(C)_1$, let $\varepsilon_{C'}$ be its indicator function
on arrows, extended additively to paths. Put
\[
 \delta_{C'}(\gamma):=\lambda(\gamma)-\varepsilon_{C'}(\gamma)\mathbf p.
\]
We use these functions also on the covering quiver by projection.

\begin{Def}\label{defcut}
A subset $C'\subseteq Q(C)_1$ is a \emph{cut} if
$\varepsilon_{C'}(\omega)=1$ for every elementary cycle $\omega$.
\end{Def}

For a cut $C'$, the function $\delta_{C'}$ vanishes on elementary
cycles. Two paths with the same endpoints and the same label
$a\le\mathbf p$ therefore have equal $\delta_{C'}$-values, since a
common return path completes both to elementary cycles.

\begin{Prop}\label{QChomogeneous}
Let $C'$ be a cut. Paths in $Q(C)$ with the same source and the same
label have the same $\varepsilon_{C'}$-value.
\end{Prop}
\begin{proof}
Write $\delta=\delta_{C'}$. For a deleted primitive arrow $a$, let
$a^\vee:t(a)\to s(a)$ be the complementary path in a cycle using each
type of its family once. It lies in $\widehat Q(\widehat C)$ and has
label $\mathbf p_+-e_i$ or $\mathbf p_--e'_{i'}$, according to the
type of $a$. The comparison with a common return path shows that
$\delta(a^\vee)$ is independent of the choice. Extend $\delta$ to
primitive arrows by
\[
 \overline\delta(a)=
 \begin{cases}
   \delta(a),&a\notin\widehat C,\\
   -\delta(a^\vee),&a\in\widehat C.
 \end{cases}
\]
The additive extension to paths in $\widehat Q$ vanishes on each cycle
using exactly the types of one family, each once.

We check that interchanging two consecutive types preserves
$\overline\delta$. For distinct types of the same family, append a
common path of the remaining types in that family and compare the
resulting cycles. For equal types there is nothing to check.
For a mixed square, write
\[
\begin{tikzcd}[column sep=large,row sep=large]
 x \arrow[r,"a"] \arrow[d,"b"'] & x+\alpha_i \arrow[d,"c"]\\
 x+\alpha'_{i'} \arrow[r,"d"'] & x+\alpha_i+\alpha'_{i'} .
\end{tikzcd}
\]
Appending the remaining primitive types to the two paths and using
the balance condition for $\widehat C$ gives
$\chi(a)+\chi(b)=\chi(c)+\chi(d)$, where $\chi$ is the indicator of
$\widehat C$. Up to exchanging the two families, this gives four cases.
\begin{enumerate}
\item If all arrows survive, $ca$ and $db$ have common label
$e_i+e'_{i'}\le\mathbf p$, so their $\delta$-values agree.
\item If all arrows are deleted, the diagonal arrow $q$ exists. The
cycles $a^\vee c^\vee q$ and $b^\vee d^\vee q$ are elementary, so
\[
 \overline\delta(ca)=\delta(q)=\overline\delta(db).
\]
\item If $a,c$ are deleted and $b,d$ survive, the cycle $a^\vee c^\vee db$
is elementary, giving $\overline\delta(ca)=\delta(db)$.
\item If $a,d$ are deleted and $b,c$ survive, $ba^\vee$ and $d^\vee c$
have the same endpoints and label $\mathbf p_+-e_i+e'_{i'}\le\mathbf p$.
Equality of their $\delta$-values gives $\overline\delta(ca)=\overline\delta(db)$.
\end{enumerate}
Every diagonal arrow occurs in the second case, since deletion of $a,b$
forces deletion of $c,d$. Thus expanding a diagonal into two primitive
arrows preserves its value. Lift the given paths to a common source
and expand their diagonals. The expanded paths have the same multiset
of primitive types, so adjacent interchanges relate them and preserve
$\overline\delta$. The original paths therefore have equal $\delta$-values,
and equality of labels gives equality of $\varepsilon_{C'}$.
\end{proof}

\begin{Prop}\label{pathindep}
Let $C'$ be a cut. If $\gamma,\gamma'$ are paths in $\widehat Q(\widehat C)$
with the same source and target, then
$\delta_{C'}(\gamma)=\delta_{C'}(\gamma')$.
\end{Prop}
\begin{proof}
By displacement and \eqref{QClabelcriterion}, their labels differ by
an element of
$\Ker\operatorname{disp}\cap\Ker\deg_H=\mathbb Z\mathbf p$, since
$\deg_H(\mathbf p_+)=-\deg_H(\mathbf p_-)=\vec s$ has infinite order.
Write $\lambda(\gamma)=\lambda(\gamma')+n\mathbf p$ with $n\ge0$ after
exchanging the paths if necessary. At their common projected source,
\eqref{QClabelcriterion} gives an elementary cycle $\omega$. Prepending
$\omega^n$ to the projection of $\gamma'$ gives a path with label
$\lambda(\gamma)$. Proposition \ref{QChomogeneous} now gives
$\varepsilon_{C'}(\gamma)=\varepsilon_{C'}(\gamma')+n$, as required.
\end{proof}

\subsection{Cuts and complete representatives}

For $u\in P\cap N$, use the lift $q_u:G_u=F/A_u(M)\to H$ of
Proposition \ref{liftlatticepoint}. Write $\vec\lambda(\gamma)$ for
the image of $\lambda(\gamma)$ in $G_u$.

\begin{Def}
A section $h\mapsto g_h$ of $q_u$ over $J$ is \emph{$Q(C)$-complete}
if for every arrow $\alpha:h\to h'$ we have
\begin{equation}\label{QCrepresentatives}
 g_h+\vec\lambda(\alpha)=g_{h'}+n_\alpha\vec p_u,
 \qquad n_\alpha\in\{0,1\}.
\end{equation}
Let $\J_{u,J}$ be the collection of images $J'=\{g_h:h\in J\}$ of
such sections, and put $C'(J'):=\{\alpha:n_\alpha=1\}$.
\end{Def}

\begin{Prop}\label{conversecut}
The assignment $(u,J')\mapsto C'(J')$ induces a bijection
\[
 \bigsqcup_{u\in P\cap N}\J_{u,J}/\mathbb Z\vec p_u
 \xrightarrow{\ \sim\ } \{\text{cuts of }Q(C)\},
\]
where $\mathbb Z\vec p_u$ acts by translation. The corresponding point
$u=u(C')$ is the \emph{type} of $C'$.
\end{Prop}
\begin{proof}
Summing \eqref{QCrepresentatives} around an elementary cycle gives
$\vec p_u=\varepsilon_{C'(J')}(\omega)\vec p_u$. Since $\vec p_u$ has
infinite order, $C'(J')$ is a cut, and it is invariant under translation
of $J'$ by $\mathbb Z\vec p_u$.

Conversely, fix a cut $C'$ and a vertex $h_0\in J$. Put
$\Lambda=\Ker\deg_H$. By \eqref{QClabelcriterion}, every
$a\in\Lambda\cap F_{\ge0}$ labels a closed path at $h_0$. Its cut degree
$c(a)$ is well-defined by Proposition \ref{QChomogeneous}.
Concatenation makes $c$ additive, and $c(\mathbf p)=1$.
It extends uniquely to $\Lambda$ by
$c(a)=c(a+n\mathbf p)-n$ for sufficiently large $n$.
With $K=\Ker c$ we have $\Lambda=K\oplus\mathbb Z\mathbf p$.
If $a\in\Lambda\cap F_{>0}$, then $c(a)=1+c(a-\mathbf p)\ge1$,
so $K\cap F_{>0}=\emptyset$. By Proposition \ref{liftlatticepoint},
the lift $F/K\to H$ is $G_u$ for the unique $u\in P\cap N$ satisfying
$c(A_P(m))=\langle m,u\rangle$ for all $m\in M$.

Identify $G_u=F/K$ and choose $g_{h_0}$ over $h_0$. For a path
$\gamma:h_0\to h$, set $g_h=g_{h_0}+[\delta_{C'}(\gamma)]$, with
brackets denoting the class in $F/K$. Every cycle at $h_0$ has
$\delta_{C'}$ in $K$, so appending a common return path shows that
$g_h$ is independent of $\gamma$. By \eqref{QClabelcriterion},
$q_u(g_h)=h$. Appending an arrow $\alpha$ gives
\eqref{QCrepresentatives} with $n_\alpha=\varepsilon_{C'}(\alpha)$,
hence $C'(\{g_h:h\in J\})=C'$.

If a section in $G_v$ induces $C'$, summing its arrow equalities
around a cycle at $h_0$ with label $a$ shows that the image of $a$
in $G_v$ is $c(a)\vec p_v$. This holds for all $a\in\Lambda$ by adding
and subtracting multiples of $\mathbf p$. Since $\vec p_v$ has
infinite order, $A_v(M)=K$ and $v=u$. Summing along paths from
$h_0$ determines the section from $g_{h_0}$, whose choices differ by
multiples of $\vec p_u$.
\end{proof}

\begin{Thm}\label{cutupcorr2}(Cut-upper set correspondence)
Let $u\in\Int(P)\cap N$ and $X=q_u^{-1}(J)$ with the induced order.
Then $\J_{u,J}=\J_X$, and the map
\[
 I'\longmapsto C'(I'):=C'\bigl(I'\cap(I'^c+\vec p_u)\bigr)
\]
induces a bijection from $\I_X/\mathbb Z\vec p_u$ to the cuts of
$Q(C)$ of type $u$.
\end{Thm}
\begin{proof}
Translation by $\vec p_u$ satisfies (A1)--(A3), since $\vec p_u>0$
and every element of $F$ becomes nonnegative after adding a sufficiently
large multiple of $\mathbf p$. Each non-trivial upper set thus meets
every orbit in a tail. A section with image $J'$ determines a union
of tails $I'=J'+\mathbb Z_{\ge0}\vec p_u$, and its initial elements
are precisely $I'\cap(I'^c+\vec p_u)$.

For an arbitrary section, let $n_\alpha\in\mathbb Z$ be the coefficient
of $\vec p_u$ in $g_h+\vec\lambda(\alpha)-g_{h'}$ for $\alpha:h\to h'$.
If $I'$ is upper, then $g_h+\vec\lambda(\alpha)\in I'$ for every
$\alpha$, so $n_\alpha\ge0$. Conversely, assume all $n_\alpha\ge0$.
If $g'\ge g$ in $X$, write $g'-g$ as the image of $a\in F_{\ge0}$.
\eqref{QClabelcriterion} gives a path $q_u(g)\to q_u(g')$ of label $a$.
The index of $g'$ in its orbit is the index of $g$ plus the sum of the
$n_\alpha$ along this path. Thus $g\in I'$ implies $g'\in I'$.
Every arrow lies in an elementary cycle, on which the sum of the
$n_\alpha$ is $1$. Hence nonnegativity is equivalent to the arrow
condition $n_\alpha\in\{0,1\}$. The result follows from
Theorem \ref{upJX} and Proposition \ref{conversecut}.
\end{proof}

\subsection{Higher representation infinite algebras of type $\widetilde A\widetilde A$}

Let $I(C)$ be the ideal of $kQ(C)$ generated by differences of paths
with the same endpoints and label. Arrow labels are distinct at
each source, and a diagonal label cannot be realized by two primitive
arrows out of that source. Thus $I(C)\subseteq kQ(C)_{\ge2}$.
By \cite[4.3,6.14]{Tom25d}, the algebra
$\Gamma(B,C):=kQ(C)/I(C)$ is an NCCR of the $(d+1)$-dimensional
Gorenstein toric singularity associated with $P$, and hence is
$(d+1)$-Calabi--Yau. For a cut $C'$, the degree
$|\gamma|:=\varepsilon_{C'}(\gamma)$ defines a nonnegative grading
on $\Gamma(B,C)$ by Proposition \ref{QChomogeneous}. Put
$A(B,C,C'):=\Gamma(B,C)_0$.

\begin{Prop}\label{doubleAtildeacy}
If the type $u$ of $C'$ lies in $\Int(P)$, the degree-zero quiver is
acyclic and $A(B,C,C')$ is finite dimensional. If $u$ lies on the
boundary of $P$, then $A(B,C,C')$ is infinite dimensional.
\end{Prop}
\begin{proof}
Choose $J'\in\J_{u,J}$ inducing $C'$ by Proposition \ref{conversecut}.
Summing \eqref{QCrepresentatives} around a cycle $\gamma$ shows that
$\vec\lambda(\gamma)=\varepsilon_{C'}(\gamma)\vec p_u$. Thus $\gamma$
has degree zero exactly when $\lambda(\gamma)\in A_u(M)$.

For $u\in\Int(P)$, Proposition \ref{liftlatticepoint}(3) gives
$A_u(M)\cap F_{\ge0}=0$, so the degree-zero quiver has no cycle of
positive length. As this quiver is finite, its path algebra and hence
$A(B,C,C')$ are finite dimensional. For $u\in P\setminus\Int(P)$,
choose $0\ne a\in A_u(M)\cap F_{\ge0}$. Since $\deg_H(a)=0$,
\eqref{QClabelcriterion} gives a cycle with label $a$ and degree zero.
Its powers correspond to distinct monomials under \cite[6.14]{Tom25d}
and are therefore linearly independent.
\end{proof}

\begin{Prop}\label{CYGor1}
Assume $u=u(C')\in\Int(P)$.  Put $\Gamma:=\Gamma(B,C)$ and
$A:=\Gamma_0$.  Then $\Gamma$ is a locally finite symmetric AS-regular
algebra over $A$ of dimension $d+1$ and Gorenstein parameter $1$.
Moreover, $\pvd^{\mathbb Z}\Gamma$ has Serre functor $(-1)[d+1]$.
\end{Prop}
\begin{proof}
Let $S=k[x_0,\ldots,x_l,x'_0,\ldots,x'_{l'}]$ with its $G_u$-grading,
and put $R=\bigoplus_{n\in\mathbb Z}S_{n\vec p_u}$, graded by $n$.
Since $A_u(M)\cap F_{\ge0}=0$, this grading is nonnegative with $R_0=k$.
The canonical module of the normal semigroup ring $R$ is spanned by
monomials with strictly positive exponent vectors. Subtracting
$\mathbf p$ from their exponents gives exactly the exponents in $R$.
Thus multiplication by $x^{\mathbf p}$ identifies $\omega_R$ with $R(-1)$.

Choose $J'\in\J_{u,J}$ inducing $C'$ and write $g_h$ for the
representative of $h$. Set $M=\bigoplus_{h\in J}M_h$, where
$(M_h)_n=S_{g_h+n\vec p_u}$. The monomial realization of
\cite[6.14]{Tom25d} gives $\Gamma\simeq\End_R(M)$, with the grading
determined by \eqref{QCrepresentatives}. The NCCR $\Gamma$ is finite
and maximal Cohen--Macaulay over $R$, with global dimension $d+1$.
It is locally finite since $R$ is an affine semigroup ring with $R_0=k$.
The graded trace pairing \cite[3.8(1)]{Han25} and the maximal
Cohen--Macaulay property give $\RHom_R(\Gamma,R)\simeq\Gamma$
as graded bimodules. Thus
\[
 \omega_\Gamma=\Hom_R(\Gamma,\omega_R)\simeq\Gamma(-1).
\]

For $X\in\pvd^{\mathbb Z}\Gamma$, adjunction and graded local duality give
\[
 \RHom_\Gamma(X,\Gamma)\simeq\RHom_R(X,R)\simeq DX(1)[-d-1].
\]
Since $X$ is perfect, tensor--Hom duality with a perfect $Y$ gives
\[
 \RHom_\Gamma(X,Y)\simeq
 (Y\otimes^{\mathbb L}_\Gamma DX)(1)[-d-1]
 \simeq D\RHom_\Gamma(Y,X)(1)[-d-1].
\]
This gives the Serre functor $(-1)[d+1]$ on $\pvd^{\mathbb Z}\Gamma$.
Taking $X=A$ in the first duality formula on either side gives the
AS-Gorenstein identities with parameter $1$. The acyclicity in
Proposition \ref{doubleAtildeacy} gives finite global dimension of $A$.
Together with $\omega_\Gamma\simeq\Gamma(-1)$, these facts give
symmetric AS-regularity.
\end{proof}

\begin{Thm}\label{AAdtame}
Let $C'$ be a cut of $Q(C)$ whose type lies in the interior of $P$. 
\begin{enumerate}
\item The finite dimensional algebra $A(B,C,C')$ is a $d$-representation tame algebra.
\item We have the following isomorphism as a graded algebra.
\[\Pi_{d+1}(A(B,C,C'))\cong\Gamma(B,C)\]
\end{enumerate}
\end{Thm}
\begin{proof}
By Propositions \ref{doubleAtildeacy} and \ref{CYGor1}, the algebra
$A:=\Gamma(B,C)_0$ is finite dimensional and $\Gamma(B,C)$ is symmetric
AS-regular over $A$ of dimension $d+1$ and Gorenstein parameter one.
The Minamoto--Mori correspondence \cite[4.12,4.13]{MM} therefore implies
that $A$ is $d$-representation infinite and gives a graded algebra
isomorphism
\[
 \Pi_{d+1}(A)\cong\Gamma(B,C).
\]
Since $\Gamma(B,C)$ is module-finite over the commutative noetherian
toric ring $R$, the same holds for $\Pi_{d+1}(A)$.
Hence $A$ is $d$-representation tame.
\end{proof}

\begin{Def}\label{defdrepinfAA}
Let $C'$ be a cut of $Q(C)$ whose type $u:=u(C')$ lies in the interior of $P$. We call the algebra $A(B,C,C')$ a {\it $d$-representation infinite algebra of type $\widetilde A\widetilde A$}.
\end{Def}

\begin{Rem}
The Gabriel quiver of $A(B,C,C')$ is the degree $0$ part of the graded quiver $Q(C)$, which is acyclic by Proposition \ref{doubleAtildeacy}. Thus these algebras have the acyclicity property asked about in
\cite[5.9]{HIO}, although this property fails for general higher
representation infinite algebras \cite{Tom25c}.
\end{Rem}

\begin{Prop}\label{tensorAandA}
Let $A(B_i,C_i)$ be a $d_i$-representation infinite algebra of type $\widetilde A$ for $i=1,2$. Then the tensor product
\[A(B_1,C_1)\otimes_kA(B_2,C_2)\]
is a $(d_1+d_2)$-representation infinite algebra of type $\widetilde A\widetilde A$.
\end{Prop}
\begin{proof}
Take $B=B_1\oplus B_2$ and let the ambient cut in the primitive product
quiver be $\widehat C=(\widehat C_1\times L_2)\sqcup
(L_1\times\widehat C_2)$. It satisfies the cut conditions in each
family and the balance condition, since the projection of a cycle
using all types once uses every type of each factor once. Its type
is positive because both factor cuts are bounding, so it is of the
form considered above by \cite[6.12,6.13]{Tom25d}.

Let $C'$ consist of all diagonal arrows of $Q(C)$. Expand the diagonal
arrows of an elementary cycle into pairs of deleted primitive arrows.
The projection to each factor is a cycle using every type once, hence
contains one deleted arrow. Each diagonal contributes one such arrow
to each factor, so the original cycle contains exactly one diagonal.
Thus $C'$ is a cut.

The degree-zero quiver is the product of the two cut quivers.
Commuting arrows in different factors groups each path into its two
projections. Relations between paths with equal labels then reduce to
the relations of the two factor algebras, as in the proof of
Theorem \ref{enddtiltrk1}. Hence
\[
 A(B,C,C')\cong A(B_1,C_1)\otimes_k A(B_2,C_2).
\]
The tensor product is finite dimensional, so $C'$ has interior type
by Proposition \ref{doubleAtildeacy}.
\end{proof}

\begin{Rem}
In the tensor-product construction, the
Gabriel quiver of the degree-zero algebra is the product quiver of the two
factors, whereas a general second cut $C'$ may leave mixed arrows in
degree zero. Examples \ref{ExHirz}, \ref{Exstackysurface} and \ref{Ex3fold} show that
such choices occur naturally from toric Fano stacks of Picard number two.
\end{Rem}

For fixed $B,C,u$, these algebras are closed under $d$-APR tilts and
are connected by iterated $d$-APR tilts (Theorems \ref{dAPRdoubleAtilde}
and \ref{dAPRconndoubleAtilde}).

\section{Tilting theory for smooth toric Fano stacks of Picard number two}

In this section, we prove the existence and give a classification of $d$-tilting bundles consisting of line bundles on smooth toric Fano stacks of Picard number two. Moreover, we show that the endomorphism algebras of these $d$-tilting bundles are $d$-representation infinite algebras of type $\widetilde{A}\widetilde{A}$, and that all such algebras are obtained in this way.

\subsection{Classification of $d$-tilting bundles consisting of line bundles}

Let $P$ be a $d$-dimensional simplicial lattice polytope with $d+2$
vertices and the origin in its interior. Write $G\cong\Pic\X(P)$ for
the group from Subsection \ref{polytope}, $q\colon G\to H=G/\mathbb Z\vec p$
for the quotient, and $\pi\colon H\to H/H_{\rm tors}\cong\mathbb Z$.
The integers $\pi(q(\vec x_i))$ are the coefficients of the unique
affine dependence of the vertices, up to a common factor. A zero
coefficient would put the other $d+1$ vertices in a facet, contrary to
simpliciality. Each sign occurs at least twice, since a single vertex of
one sign would be a convex combination of the others. Index the two
families of degrees as $\vec x_0,\ldots,\vec x_l$ and
$\vec x'_0,\ldots,\vec x'_{l'}$, with
\[
 \pi(q(\vec x_i))>0,\qquad \pi(q(\vec x'_{i'}))<0,\qquad
 l,l'\ge1,\quad d=l+l'.
\]
Here $\vec p=\sum_{i=0}^l\vec x_i+\sum_{i'=0}^{l'}\vec x'_{i'}$.

Here, $q(\vec{x}_0), \cdots, q(\vec{x}_l), q(-\vec{x}'_0),\cdots, q(-\vec{x}'_{l'})\in H$ satisfy (G1), (G2) and (G3). Thus $H$ has the following partial order.
\[h_1\geq h_2\Leftrightarrow h_1-h_2\in\sum_{i=0}^l\mathbb{Z}_{\geq0}q(\vec{x}_i)+\sum_{i'=0}^{l'}\mathbb{Z}_{\geq0}q(-\vec{x}'_{i'})\]
Put $s:=\sum_{i=0}^lq(\vec{x}_i)=\sum_{i'=0}^{l'}q(-\vec{x}'_{i'})\in H$. Then $\mathbb{Z}$ acts on $H$ by $n\cdot h:=h+ns$. This action satisfies the conditions (A1),(A2) and (A3).

Let $\X:=\X(P)$ be the smooth toric Fano stack corresponding to the polytope $P$. To calculate the cohomology groups of line bundles on $\X$, we first determine the topological space $X_c$ for $c\in\mathbb{Z}^{d+2}$. This is almost proved in \cite[4.2]{Tom25d}.

\begin{Lem}\label{caltop}
Let $c=(a,a')\in\mathbb{Z}^{d+2}$ where $a=(a_i)_{i=0}^l\in\mathbb{Z}^{l+1}$ and $a'=(a'_{i'})_{i'=0}^{l'}\in\mathbb{Z}^{l'+1}$.
\begin{enumerate}
\item If there exist $0\leq i_0\neq i_1\leq l$ such that $a_{i_0}\geq0$ and $a_{i_1}<0$ hold, then $X_c$ is contractible.
\item If there exist $0\leq i'_0\neq i'_1\leq l'$ such that $a'_{i'_0}\geq0$ and $a'_{i'_1}<0$ hold, then $X_c$ is contractible.
\item If $a\in\mathbb{Z}_{<0}^{l+1}$ and $a'\in\mathbb{Z}_{<0}^{l'+1}$ hold, then $X_c=\emptyset$
\item If $a\in\mathbb{Z}_{\ge0}^{l+1}$ and $a'\in\mathbb{Z}_{\ge0}^{l'+1}$ hold, then $X_c$ is homeomorphic to the $(d-1)$-dimensional sphere $S^{d-1}$.
\item If $a\in\mathbb{Z}_{\ge0}^{l+1}$ and $a'\in\mathbb{Z}_{<0}^{l'+1}$ hold, then $X_c$ is homeomorphic to the $(l-1)$-dimensional sphere $S^{l-1}$.
\item If $a\in\mathbb{Z}_{<0}^{l+1}$ and $a'\in\mathbb{Z}_{\ge0}^{l'+1}$ hold, then $X_c$ is homeomorphic to the $(l'-1)$-dimensional sphere $S^{l'-1}$.
\end{enumerate}
\end{Lem}
\begin{proof}
All the statements except for (4) are proved in \cite[4.2]{Tom25d} since our $X_c$ here coincides with that of \cite{Tom25d} since $c\notin\mathbb{Z}_{\ge0}^{d+2}$. (4) follows since $X_c$ coincides with the boundary $\partial P$ of $P$.
\end{proof}

Using Lemma \ref{caltop}, we give a vanishing criterion for certain singular cohomology groups of $X_c$.

\begin{Lem}\label{vancoh}
For $\vec{g}\in G$, the following conditions are equivalent.
\begin{enumerate}
\item For any $0\leq r\leq d-2$, and $c=(a,a')\in\mathbb{Z}^{d+2}$ where $a=(a_i)_{i=0}^l\in\mathbb{Z}^{l+1}$ and $a'=(a'_{i'})_{i'=0}^{l'}\in\mathbb{Z}^{l'+1}$ with $\sum_{i=0}^la_i\vec{x}_i+\sum_{i'=0}^{l'}a'_{i'}\vec{x}'_{i'}\in\vec{g}+\mathbb{Z}\vec{p}$, we have $\widetilde H_r(X_c;k)=0$.
\item $q(\vec{g})\ngeq s$ and $q(\vec{g})\nleq -s$ hold.
\end{enumerate}
\end{Lem}
\begin{proof}
By Lemma \ref{caltop}, $\widetilde H_r(X_c;k)$ does not vanish for some $0\leq r\leq d-2$ if and only if either of the following conditions is satisfied.
\begin{enumerate}
\item[(5)] $a\in\mathbb{Z}_{\ge0}^{l+1}$ and $a'\in\mathbb{Z}_{<0}^{l'+1}$
\item[(6)] $a\in\mathbb{Z}_{<0}^{l+1}$ and $a'\in\mathbb{Z}_{\ge0}^{l'+1}$
\end{enumerate}
Observe that for $\vec{g}\in G$, there exists $c=(a,a')\in\mathbb{Z}^{d+2}$ satisfying (5) (respectively, (6)) with $q(\vec g)=\sum_{i=0}^la_iq(\vec{x}_i)+\sum_{i'=0}^{l'}a'_{i'}q(\vec{x}'_{i'})$ if and only if $q(\vec g)\ge s$ (respectively, $q(\vec g)\le -s$) holds.
\end{proof}

As a corollary, we obtain the following vanishing criteria for cohomology groups of line bundles on $\X$.

\begin{Cor}\label{rigid}
For a subset $J'\subseteq G$, the following conditions are equivalent.
\begin{enumerate}
\item For any $\vec{g},\vec{h}\in J', n\in\mathbb{Z}$ and $0<r<d$, we have $\Ext_\X^r(\O_\X(\vec{g}),\O_\X(\vec{h}+n\vec{p}))=0$.
\item For any $\vec{g},\vec{h}\in J', 0\leq r\leq d-2$ and $a\in\mathbb{Z}^n$ with $\sum_ia_i\vec{x}_i\in\vec{h}-\vec{g}+\mathbb{Z}\vec{p}$, we have $\tilde{H}_r(X_a;k)=0$.
\item For any $\vec{g},\vec{h}\in J'$, we have $q(\vec{g})\ngeq q(\vec{h})+s$.
\item There exists $I\in\I_H$ such that $q(J')\subseteq J(I)$ holds.
\end{enumerate}
\end{Cor}
\begin{proof}
(1)$\Leftrightarrow$(2) follows from Proposition \ref{calcoh}. (2)$\Leftrightarrow$(3) follows from Lemma \ref{vancoh}. (3)$\Leftrightarrow$(4) is \cite[1.4]{Tom25d}.
\end{proof}

We show the following lemma which is an important step to prove that our tilting bundles generate the derived category.

\begin{Lem}\label{CTthick}
For any $I\in\I_H$, we have
\[\thick\{\O_\X(\vec{g})\mid q(\vec{g})\in J(I)\}=\D^b(\Coh\X).\]
\end{Lem}
\begin{proof}
Put $\T:=\thick\{\O_\X(\vec{g})\mid q(\vec{g})\in J(I)\}\subseteq\D^b(\Coh\X)$. Since $\thick\{\O_\X(\vec{g})\mid\vec{g}\in G\}=\D^b(\Coh\X)$ holds, it is enough to show that $\O_\X(\vec{g})\in\T$ holds for any $\vec{g}\in G$. Consider the graded Koszul complex of a regular sequence $x_0,\cdots, x_l\in S:=k[x_0,\cdots,x_l,x'_0,\cdots,x'_{l'}]$.
\[0\to S(-\vec{x}_0-\cdots-\vec{x}_l)\to\cdots\to S\to S/(x_0,\cdots, x_l)\to0\]
This yields an exact sequence
\[0\to\O_\X(-\vec{x}_0-\cdots-\vec{x}_l)\to\cdots\to\O_\X\to0\]
in $\Coh\X$. Let $m\in I$ be a minimal element and take $\vec{m}\in q^{-1}(m)$. We obtain an exact sequence
\[0\to\O_\X(\vec{m})\to\cdots\to\O_\X(\vec{m}+\vec{x}_0+\cdots+\vec{x}_l)\to0.\]
By the minimality of $m\in I$, for any proper subset $\Lambda\subsetneq\{0,\cdots,l\}$, we have $\vec{m}+\sum_{i\in\Lambda}\vec{x}_i\in q^{-1}(J(I))$. Thus $\O_\X(\vec{m}+\vec{x}_0+\cdots+\vec{x}_l)\in\T$ holds. This means that for any $\vec{g}\in q^{-1}(J(\mu_m^-(I)))$, we have $\O_\X(\vec{g})\in\T$. Moreover, the converse holds: for $I_0\in\I_\H$ and a minimal element $m_0\in I_0$, if $\O_\X(\vec{g})\in\T$ holds for any $\vec{g}\in q^{-1}(J(\mu_{m_0}^-(I_0)))$, then we have $\O_\X(\vec{g})\in\T$ for any $\vec{g}\in q^{-1}(J(I_0))$. Thus by \cite[1.10]{Tom25d}, we obtain $\O_\X(\vec{g})\in\T$ for any $\vec{g}\in G$.
\end{proof}

Since $G$ has a partial order, for $J\in\J_H$, we can equip $q^{-1}(J)\subseteq G$ with a partial order. Then observe that for $\vec{g},\vec{h}\in q^{-1}(J)$, we have
\[\vec{g}\leq\vec{h}\Leftrightarrow S_{\vec{h}-\vec{g}}\neq0\Leftrightarrow\Hom_\X(\O_\X(\vec{g}),\O_\X(\vec{h}))\neq0.\]
Here, $\mathbb{Z}$ acts on $q^{-1}(J)$ by $n\cdot\vec{g}:=\vec{g}+n\vec{p}$. This action satisfies the conditions (A1),(A2) and (A3).

\begin{Thm}\label{classfitiltrk2}
Let $P$ be a $d$-dimensional simplicial lattice polytope with $d+2$ vertices containing the origin as an interior point and $\X:=\X(P)$. Then for a subset $J'\subseteq G\cong\Pic\X$, the following conditions are equivalent.
\begin{enumerate}
\item $\E(J'):=\bigoplus_{\vec{g}\in J'}\O_\X(\vec{g})\in\D^b(\Coh\X)$ is a $d$-tilting bundle.
\item There exists $J\in\J_H$ containing $q(J')$ such that $J'\in\J_{q^{-1}(J)}$ holds.
\end{enumerate}
\end{Thm}
Thus by Theorem \ref{upJX}, $d$-tilting bundles on $\X$ consisting of line bundles correspond bijectively to the pairs $(I,I')$ where $I\in\I_H$ and $I'\in\I_{q^{-1}(J(I))}$.
\begin{proof}[Proof of Theorem \ref{classfitiltrk2}]
(2)$\Rightarrow$(1) Write $J=J(I)$ for $I\in\I_H$ and let $C=C(I)$.
By Theorem \ref{cutupcorr2}, $J'\in\J_{q^{-1}(J)}$ induces a cut
$C'=C'(J')$ of type $0\in\Int(P)$.
Put $Q=\bigoplus_{\vec g\in J'}S(\vec g)$. The monomial realization
in Proposition \ref{CYGor1} and the grading rule
\eqref{QCrepresentatives} identify the cut-graded algebra with
\begin{equation}\label{cutorbitalgebra}
 \Gamma(B,C)\cong
 \bigoplus_{n\in\mathbb Z}\Hom_S^G(Q,Q(n\vec p)).
\end{equation}
By Proposition \ref{calcoh}, its degree-zero part is $\End_\X(\E(J'))$,
which has global dimension at most $d$ by Theorem \ref{AAdtame}.
Corollary \ref{rigid} gives the vanishing of intermediate self-extensions
of $\E(J')$. The top self-extensions vanish by Serre duality, since
$\vec g\ngeq\vec h+\vec p$ for $\vec g,\vec h\in J'$.

For generation, write $\Gamma=\Gamma(B,C)$ and
$\T=\thick\{S(\vec g)\mid q(\vec g)\in J\}$. There is a triangle equivalence
\[
 F=\bigoplus_{n\in\mathbb Z}\RHom_S^G(Q,-(n\vec p))
 \colon \T\xrightarrow[\simeq]{}\per^{\mathbb Z}\Gamma.
\]
Indeed, $F$ identifies the free shifts $S(\vec g+n\vec p)$,
$\vec g\in J'$, with the projective shifts $e_{\vec g}\Gamma(n)$
and identifies their morphisms by \eqref{cutorbitalgebra}. These
objects generate the two categories under shifts, cones and summands.
For $Y\in\T$ we have
\[
 H^i(FY)_n=\bigoplus_{\vec g\in J'}(H^iY)_{-\vec g+n\vec p}.
\]
The degrees on the right range over $q^{-1}(-J)$. Since negation
interchanges the two entries in the defining order condition for
$\J_H$, we have $-J\in\J_H$. Thus if $FY$ has finite total
cohomology, Lemma \ref{invflvanq} applied to each $H^iY$ shows that
$Y$ sheafifies to zero.

By Proposition \ref{CYGor1}, $\Gamma$ is finite over a commutative
noetherian ring and has finite global dimension. On each side,
$\Gamma_0$ admits a graded free resolution with finitely generated
terms. A sufficiently high syzygy is projective and an ungraded
splitting exhibits it as graded projective by taking the degree-zero
part. This gives (D1). By Lemma \ref{degsuppzero}, the duality formula
in Proposition \ref{CYGor1} gives (D2) with parameter $1$.
Since $\Gamma_0$ is finite dimensional, sheafification
on $\T$ factors through
$\per^{\mathbb Z}\Gamma/\thick^{\mathbb Z}\Gamma_0$ via $F$.
By Theorem \ref{GBei} for $I=\mathbb Z_{\ge0}$, the image of $\Gamma$
generates this quotient. Hence $\thick\E(J')$ contains every
$\O_\X(\vec g)$ with $q(\vec g)\in J$, so it is all of
$\D^b(\Coh\X)$ by Lemma \ref{CTthick}.

(1)$\Rightarrow$(2) Put $A=\End_\X(\E(J'))$. By Proposition \ref{SerreRI},
$A$ is $d$-representation infinite. Under the tilting equivalence,
tensoring with $\O_\X(\vec p)$ corresponds to $\nu_d^{-1}$ by Serre
duality. Thus
$\RHom_\X(\E(J'),\E(J')(n\vec p))\cong\nu_d^{-n}A$
is concentrated in degree zero for $n\ge0$. For $n<0$, Serre duality
gives vanishing of the extensions in degrees $0<r<d$ by reducing to
the twist $(-n-1)\vec p$. Corollary \ref{rigid} now gives
$J\in\J_H$ containing $q(J')$. Also the vanishing of the top
self-extensions gives $\vec g\ngeq\vec h+\vec p$ for all
$\vec g,\vec h\in J'$. Hence $J'$ is contained in some
$J''\in\J_{q^{-1}(J)}$ by \cite[1.4]{Tom25d}. By the forward
implication, $\E(J'')$ is tilting. Both basic tilting bundles have
as many summands as the rank of $K_0(\D^b(\Coh\X))$, so $J'=J''$.
\end{proof}

\begin{Lem}\label{invflvanq}
Take $J\in\J_H$. For $M\in\mod^GS$, if $\#\{\vec{g}\in q^{-1}(J)\mid M_{\vec{g}}\neq0\}<\infty$ holds, then we have $M\in\mod^G_{SR(P)}S$.
\end{Lem}
\begin{proof}
Each facet of $P$ omits one vertex from each family, so it suffices
to show that $M$ vanishes after inverting $x_ix'_{i'}$ for each pair $i,i'$.
Fix such a pair and a homogeneous element $m\in M_{\vec g}$.
Write $J=I\setminus(I+s)$ and put $h=q(\vec g)$ and $u=q(\vec x_i)$.
Since $u$ and $s$ have a common positive multiple and $J$ meets every
$\mathbb Zs$-orbit, there is a least integer $r$ with $h+ru\in I$.
As $u\leq s$ and $h+(r-1)u\notin I$, we have $h+ru-s\notin I$,
so $h+ru\in J$.

The opposite signs of $\pi(u)$ and $\pi(q(\vec x'_{i'}))$ give positive
integers $a,b$ with $au+bq(\vec x'_{i'})=0$ in $H$, after clearing torsion.
By (G1) and (G3), the sum $a\vec x_i+b\vec x'_{i'}$ is a nonzero multiple
of $\vec p$. For $n\gg0$, the products $m x_i^{r+na}(x'_{i'})^{nb}$ have
pairwise distinct degrees in $q^{-1}(J)$. By hypothesis, some such product
is zero. Its monomial factor is invertible after inverting $x_ix'_{i'}$,
so $m$ vanishes in the localization.
\end{proof}

Next, we see that the endomorphism algebras of tilting bundles obtained in Theorem \ref{classfitiltrk2} are $d$-representation infinite algebras of type $\widetilde A\widetilde A$. Compare this with Theorem \ref{enddtiltrk1}.

\begin{Thm}\label{enddtiltrk2}
Let $P$ be a $d$-dimensional simplicial lattice polytope with $d+2$ vertices containing the origin as an interior point and $\X:=\X(P)$. Let $B\subseteq L\oplus L'$ be a cofinite subgroup corresponding to $P$. Then for $I\in\I_H$ and $I'\in\I_{q^{-1}(J(I))}$, we have
\[\End_\X(\E(J(I')))\cong A(B,C(I),C'(I')).\]
In particular, the following statements hold.
\begin{enumerate}
\item The endomorphism algebra of a $d$-tilting bundle consisting of line bundles on a $d$-dimensional smooth toric Fano stack of Picard number two is a $d$-representation infinite algebra of type $\widetilde A\widetilde A$.
\item Conversely, every $d$-representation infinite algebra of type $\widetilde A\widetilde A$ can be obtained in this way.
\end{enumerate}
\end{Thm}
\begin{proof}
The isomorphism is the degree-zero part of \eqref{cutorbitalgebra}
together with Proposition \ref{calcoh}. Theorem \ref{classfitiltrk2}
then gives (1).

For (2), let $C'$ be an interior cut of type $u$. The translated
polytope $P-u$ has the same group $H$ and the same ordered degrees in
$H$, since $A_u(M)+\mathbb Z\mathbf p=A_P(M)+\mathbb Z\mathbf p$.
Thus it gives the same $B$ and $C$. The cut $C'$ has type zero for
$P-u$, so Theorem \ref{cutupcorr2} gives an upper set $I'$ in the
corresponding lift $G_u$. Apply the isomorphism to $\X(P-u)$.
\end{proof}

This theorem shows that the smooth toric Fano stacks of Picard number two give geometric models of the higher representation infinite algebras of type $\widetilde A\widetilde A$. Moreover, together with Proposition \ref{SerreRI}, this theorem gives an alternative proof that our $A(B,C,C')$ is certainly $d$-representation infinite (Theorem \ref{AAdtame}).

As an immediate corollary, we obtain the following. We strengthen this corollary in Theorem \ref{dAPRconndoubleAtilde}.

\begin{Cor}
Let $P$ be a $d$-dimensional simplicial lattice polytope with $d+2$ vertices containing the origin as an interior point. Let $B\subseteq L\oplus L'$ be the cofinite subgroup corresponding to $P$. Take $I\in\I_H$ and cuts $C'_1,C'_2$ of $Q(C(I))$ with common type $u\in\Int(P)$. Then the two algebras $A(B,C(I),C'_1)$ and $A(B,C(I),C'_2)$ are derived equivalent.
\end{Cor}
\begin{proof}
By Theorems \ref{classfitiltrk2} and \ref{enddtiltrk2}, both algebras
are endomorphism algebras of tilting bundles on $\X(P-u)$.
\end{proof}

Using the derived equivalence $\D^b(\Coh\X)\simeq\per A$ obtained by Theorem \ref{classfitiltrk2}, we can give a description of the $d$-preprojective component and the $d$-preinjective component $\P,\I\subseteq\mod A$. Remark that we have the following commutative diagram obtained by the uniqueness of the Serre functor.
\[\xymatrix{
\D^b(\Coh\X) \ar[r]_\simeq \ar[d]_{(\vec{p})} & \per A \ar[d]^{\nu_d^{-1}}\\
\D^b(\Coh\X) \ar[r]_\simeq & \per A
}\]

\begin{Prop}
Take $I\in\I_H$ and $I'\in\I_{q^{-1}(J(I))}$. Put $A:=\End_\X(\E(J(I')))$ in the notation of Theorem \ref{classfitiltrk2}. Then the derived equivalence $\D^b(\Coh\X)\simeq\per A$ restricts to equivalences
\[\add\{\O_\X(\vec{g})\mid\vec{g}\in I'\}\simeq\P\text{ and }\add\{\O_\X(\vec{g})\mid\vec{g}\in q^{-1}(J(I))\setminus I'\}\simeq\I[-d].\]
In particular, we obtain an equivalence
\[\add\{\O_\X(\vec{g})\mid\vec{g}\in q^{-1}(J(I))\}\simeq\I[-d]\vee\P.\]
\end{Prop}
\begin{proof}
The assertion follows from the above commutative diagram.
\end{proof}

Next, we investigate the $d$-APR tilts \cite{IO11} of $d$-representation infinite algebras of type $\widetilde A\widetilde A$ through their geometric models. First, we prove that the endomorphism algebra of a $d$-APR tilting module of a $d$-representation infinite algebra of type $\widetilde A\widetilde A$ is again a $d$-representation infinite algebra of type $\widetilde A\widetilde A$ with the same $B,C$ and type $u\in P$.

\begin{Thm}\label{dAPRdoubleAtilde}
Let $P$ be a $d$-dimensional simplicial lattice polytope with $d+2$ vertices containing the origin as an interior point. Let $B\subseteq L\oplus L'$ be the cofinite subgroup corresponding to $P$. Take $I\in\I_H$ and $I'\in\I_{q^{-1}(J(I))}$. Put $A:=A(B,C(I),C'(I'))$. Take a minimal element $\vec{m}\in I'$ and let $T:=\nu_d^{-1}(e_{\vec{m}}A)\oplus\bigoplus_{\vec{g}\in J(I')\setminus\{\vec{m}\}}e_{\vec{g}}A\in\mod A$ be the $d$-APR tilting module with respect to $e_{\vec{m}}A$. Then we have
\[\End_A(T)\cong A(B,C(I),C'(\mu^-_{\vec{m}}(I'))).\]
\end{Thm}
\begin{proof}
By minimality of $\vec m$, the projective $e_{\vec m}A$ is simple.
The tilting equivalence sends $T$ to the bundle obtained from
$\E(J(I'))$ by replacing $\O_\X(\vec m)$ with $\O_\X(\vec m+\vec p)$.
Its summand degrees form $J(\mu^-_{\vec m}(I'))$, so
Theorem \ref{enddtiltrk2} identifies its endomorphism algebra.
\end{proof}

We emphasize that Theorem \ref{dAPRdoubleAtilde} is difficult to prove without using their geometric models. Thanks to Theorem \ref{dAPRdoubleAtilde}, we can prove that all $d$-representation infinite algebras of type $\widetilde A\widetilde A$ having same $B,C$ and the type $u$ are connected by a finite sequence of iterated $d$-APR tilts.

\begin{Thm}\label{dAPRconndoubleAtilde}
Let $P$ be a $d$-dimensional simplicial lattice polytope with $d+2$ vertices containing the origin as an interior point. Let $B\subseteq L\oplus L'$ be the cofinite subgroup corresponding to $P$. Take $I\in\I_H$ and cuts $C'_1,C'_2$ of $Q(C(I))$ with common type $u\in\Int(P)$. Then the two algebras $A(B,C(I),C'_1)$ and $A(B,C(I),C'_2)$ are connected by a finite sequence of iterated $d$-APR tilts.
\end{Thm}
\begin{proof}
Work with $P-u$, so that $C'_1$ and $C'_2$ have type zero and
the lift is $q_u\colon G_u\to H$. Choose upper sets $I'_1,I'_2$ in
$q_u^{-1}(J(I))$ inducing the cuts by Theorem \ref{cutupcorr2}.
As in the proof of Theorem \ref{dAPRconnAtilde}, translate $I'_2$
by a multiple of $\vec p_u$ so that $I'_2\subseteq I'_1$.
Successively deleting minimal elements from $I'_1$ to reach $I'_2$
then gives a sequence of $d$-APR tilts by \cite[1.9]{Tom25d} and
Theorem \ref{dAPRdoubleAtilde}.
\end{proof}

\subsection{Examples of dimension $2$}

We classify $2$-tilting bundles consisting of line bundles on some toric stacky surfaces, including $\mathbb{P}^1\times\mathbb{P}^1$ and $\Sigma_1$, and determine their quivers by using Theorem \ref{classfitiltrk2}. We remark that all the endomorphism algebras of the obtained $2$-tilting bundles are $2$-representation infinite algebras of type $\widetilde A\widetilde A$.

\begin{Ex}($d=2$)\label{ExHirz}
We see that Theorem \ref{classfitiltrk2} gives a classification of $2$-tilting bundles consisting of line bundles on the Hirzebruch surfaces $\mathbb{P}^1\times\mathbb{P}^1$ and $\Sigma_1$. This is also known as the classification of geometric helices.

(1) Put $G:=\mathbb{Z}^2$ and $\vec{x}=\vec{y}=(1,0),\vec{z}=\vec{w}=(0,1)\in G$. We view $S:=k[x,y,z,w]$ as a $G$-graded $k$-algebra. Then the resulting toric stack $\X$ is isomorphic to $\mathbb{P}^1\times\mathbb{P}^1$. We have $\vec{p}=(2,2)$ and $H=G/\mathbb{Z}\vec{p}\cong\mathbb{Z}\oplus(\mathbb{Z}/2\mathbb{Z}); (a,b)+\mathbb{Z}\vec{p}\mapsto(a-b,b+2\mathbb{Z})$. If we equip $H$ with our partial order, then the quiver of $H$ becomes the following.
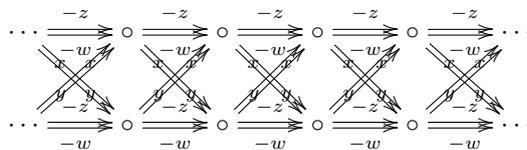
\[\xymatrix{
\cdots \ar@2[r]^x_y \ar@2[dr]^{-z}_{-w} & \circ \ar@2[r]^x_y \ar@2[dr]^{-z}_{-w} & \circ \ar@2[r]^x_y \ar@2[dr]^{-z}_{-w} & \circ \ar@2[r]^x_y \ar@2[dr]^{-z}_{-w} & \circ \ar@2[r]^x_y \ar@2[dr]^{-z}_{-w} & \cdots \\
\cdots \ar@2[r]^x_y \ar@2[ur]^{-z}_{-w} & \circ \ar@2[r]^x_y \ar@2[ur]^{-z}_{-w} & \circ \ar@2[r]^x_y \ar@2[ur]^{-z}_{-w} & \circ \ar@2[r]^x_y \ar@2[ur]^{-z}_{-w} & \circ \ar@2[r]^x_y \ar@2[ur]^{-z}_{-w} & \cdots
}\]
Then there are the following two kinds of non-trivial upper sets in $H$ up to translations.

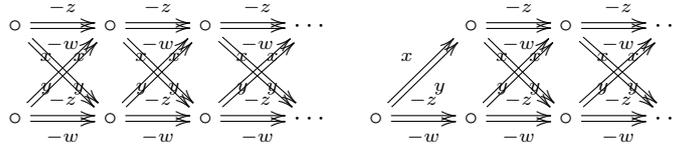
\[\begin{array}{c c}
\xymatrix{
\circ \ar@2[r]^x_y \ar@2[dr]^{-z}_{-w} & \circ \ar@2[r]^x_y \ar@2[dr]^{-z}_{-w} & \circ \ar@2[r]^x_y \ar@2[dr]^{-z}_{-w} & \cdots \\
\circ \ar@2[r]^x_y \ar@2[ur]^{-z}_{-w} & \circ \ar@2[r]^x_y \ar@2[ur]^{-z}_{-w} & \circ \ar@2[r]^x_y \ar@2[ur]^{-z}_{-w} & \cdots
}&\xymatrix{
\circ \ar@2[r]^x_y \ar@2[dr]^{-z}_{-w} & \circ \ar@2[r]^x_y \ar@2[dr]^{-z}_{-w} & \circ \ar@2[r]^x_y \ar@2[dr]^{-z}_{-w} & \cdots \\
 & \circ \ar@2[r]^x_y \ar@2[ur]^{-z}_{-w} & \circ \ar@2[r]^x_y \ar@2[ur]^{-z}_{-w} & \cdots
}\end{array}\]
Remark that the isomorphism $H\cong\mathbb{Z}\oplus(\mathbb{Z}/2\mathbb{Z})$ sends $s$ to $(2,0)$. If we let $\widehat{C}_1,\widehat{C}_2$ be the cuts of $\widehat{Q}$ corresponding to the two kinds of upper sets in $H$, then $\widehat{Q}(\widehat{C}_1)$ and $\widehat{Q}(\widehat{C}_2)$ are as follows. Here, the black vertices correspond to elements of $B$.
\[
\begin{tikzcd}[row sep=large]
\circ \arrow[d, "w"] 
& \circ \arrow[r, "x"] \arrow[l, "y"] 
& \circ \arrow[d, "w"] 
& \circ \arrow[l, "y"] 
\\
\bullet \arrow[r, "x"] 
& \circ \arrow[u, "z"] \arrow[d, "w"] 
& \bullet \arrow[r, "x"] \arrow[l, "y"] 
& \circ \arrow[u, "z"] \arrow[d, "w"] 
\\
\circ \arrow[u, "z"] \arrow[d, "w"] 
& \circ \arrow[r, "x"] \arrow[l, "y"] 
& \circ \arrow[u, "z"] \arrow[d, "w"] 
& \circ \arrow[l, "y"] 
\\
\bullet \arrow[r, "x"] 
& \circ \arrow[u, "z"] 
& \bullet \arrow[r, "x"] \arrow[l, "y"] 
& \circ \arrow[u, "z"]
\end{tikzcd}
\quad
\begin{tikzcd}[row sep=large]
\circ \arrow[r, "x"] \arrow[d, "w"]
& \circ \arrow[d, "w"]
& \circ \arrow[r, "x"] \arrow[l, "y"] \arrow[d, "w"]
& \circ \arrow[d, "w"]
\\
\bullet \arrow[r, "x"]
& \circ \arrow[ur, "xz"] \arrow[ul, "yz"] \arrow[dr, "xw"] \arrow[dl, "yw"]
& \bullet \arrow[r, "x"] \arrow[l, "y"]
& \circ \arrow[ul, "yz"] \arrow[dl, "yw"]
\\
\circ \arrow[r, "x"] \arrow[u, "z"] \arrow[d, "w"]
& \circ \arrow[u, "z"] \arrow[d, "w"]
& \circ \arrow[r, "x"] \arrow[l, "y"] \arrow[u, "z"] \arrow[d, "w"]
& \circ \arrow[u, "z"] \arrow[d, "w"]
\\
\bullet \arrow[r, "x"]
& \circ \arrow[ur, "xz"] \arrow[ul, "yz"]
& \bullet \arrow[r, "x"] \arrow[l, "y"] 
& \circ \arrow[ul, "yz"]
\end{tikzcd}
\]

Thus there are the following two kinds of sets in $\J_H$ up to translations with the following quivers $Q(C_1)$ and $Q(C_2)$.

\[\begin{array}{c c}
\xymatrix{
\circ \ar@2[r]^x_y & \circ \ar@2[dl]^z_w \\
\circ \ar@2[r]^x_y & \circ \ar@2[ul]^z_w
}&\xymatrix{
\circ \ar@2[r]^x_y & \circ \ar[d]^4 \\
 & \circ \ar@2[r]^x_y \ar@2[ul]^z_w & \circ \ar@2[ul]^z_w
}\end{array}\]
Here, the arrow $\xrightarrow{4}$ means that there are $4$ arrows consisting of $xz,xw,yz$ and $yw$. The quivers of $q^{-1}(J)$ are as follows.
\[\xymatrix{
 & & & \circ \ar@2[r]^x_y & \cdots \\
 & & \circ \ar@2[r]^x_y & \circ \ar@2[u]^z_w \\
 & \circ \ar@2[r]^x_y & \circ \ar@2[u]^z_w \\
\cdots \ar@2[r]^x_y & \circ \ar@2[u]^z_w
}\xymatrix{
 & & & & & \rotatebox{90}{$\ddots$} \\
 & & & \circ \ar@2[r]^x_y & \circ \ar[ur]^4 \\
 & & & \circ \ar@2[r]^x_y \ar@2[u]^z_w & \circ \ar@2[u]^z_w \\
 & \circ \ar@2[r]^x_y & \circ \ar[ur]^4 \\
 & \circ \ar@2[r]^x_y \ar@2[u]^z_w & \circ \ar@2[u]^z_w \\
\rotatebox{90}{$\ddots$} \ar[ur]^4
}\]
In the first case, $q^{-1}(J)$ is a chain in which translation by
$\vec p=(2,2)$ advances four steps. There are four non-trivial upper
sets modulo $\mathbb Z\vec p$. The two diagrams below represent them
modulo translation by $(1,1)$, which advances two steps and preserves
$q^{-1}(J)$.

\[\xymatrix{
 & & \circ \ar@2[r]^x_y & \cdots \\
 & \circ \ar@2[r]^x_y & \circ \ar@2[u]^z_w \\
\circ \ar@2[r]^x_y & \circ \ar@2[u]^z_w \\
\circ \ar@2[u]^z_w
}\xymatrix{
 & & \circ \ar@2[r]^x_y & \cdots \\
 & \circ \ar@2[r]^x_y & \circ \ar@2[u]^z_w \\
\circ \ar@2[r]^x_y & \circ \ar@2[u]^z_w
}\]

These upper sets correspond to the following cuts of $\widehat Q(\widehat C_1)$.
\[
\begin{tikzcd}[row sep=large]
\circ \arrow[d, "w"] 
& \circ \arrow[r, line width=2pt, no head] \arrow[l, line width=2pt, no head]
& \circ \arrow[d, "w"] 
& \circ \arrow[l, line width=2pt, no head]
\\
\bullet \arrow[r, "x"] 
& \circ \arrow[u, "z"] \arrow[d, "w"] 
& \bullet \arrow[r, "x"] \arrow[l, "y"] 
& \circ \arrow[u, "z"] \arrow[d, "w"] 
\\
\circ \arrow[u, "z"] \arrow[d, "w"] 
& \circ \arrow[r, line width=2pt, no head] \arrow[l, line width=2pt, no head]
& \circ \arrow[u, "z"] \arrow[d, "w"] 
& \circ \arrow[l, line width=2pt, no head]
\\
\bullet \arrow[r, "x"] 
& \circ \arrow[u, "z"] 
& \bullet \arrow[r, "x"] \arrow[l, "y"] 
& \circ \arrow[u, "z"]
\end{tikzcd}
\quad\begin{tikzcd}[row sep=large]
\circ  \arrow[d, line width=2pt, no head]
& \circ \arrow[r, "x"] \arrow[l, "y"] 
& \circ  \arrow[d, line width=2pt, no head]
& \circ \arrow[l, "y"] 
\\
\bullet \arrow[r, "x"] 
& \circ \arrow[u, "z"] \arrow[d, "w"] 
& \bullet \arrow[r, "x"] \arrow[l, "y"] 
& \circ \arrow[u, "z"] \arrow[d, "w"] 
\\
\circ \arrow[u, line width=2pt, no head]  \arrow[d, line width=2pt, no head] 
& \circ \arrow[r, "x"] \arrow[l, "y"] 
& \circ \arrow[u, line width=2pt, no head]  \arrow[d, line width=2pt, no head]
& \circ \arrow[l, "y"] 
\\
\bullet \arrow[r, "x"] 
& \circ \arrow[u, "z"] 
& \bullet \arrow[r, "x"] \arrow[l, "y"] 
& \circ \arrow[u, "z"]
\end{tikzcd}
\]

Therefore there are the following two kinds of $2$-tilting bundles up to translations. Their endomorphism algebras are $2$-representation infinite algebras of type $\widetilde A\widetilde A$. Observe that by mutations of non-trivial upper sets in $q^{-1}(J)$, they are mutated to each other, which correspond to $2$-APR tilting mutations.

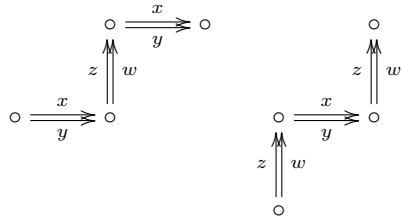
\[\begin{array}{c c}
\xymatrix{
 & \circ \\
\circ \ar@2[r]^x_y & \circ \ar@2[u]^z_w \\
\circ \ar@2[u]^z_w
}&\xymatrix{
 & \circ \ar@2[r]^x_y & \circ \\
\circ \ar@2[r]^x_y & \circ \ar@2[u]^z_w
}\end{array}\]
In the second case, there are the following five kinds of non-trivial upper sets in $q^{-1}(J)$ up to translations by $\vec{p}$.

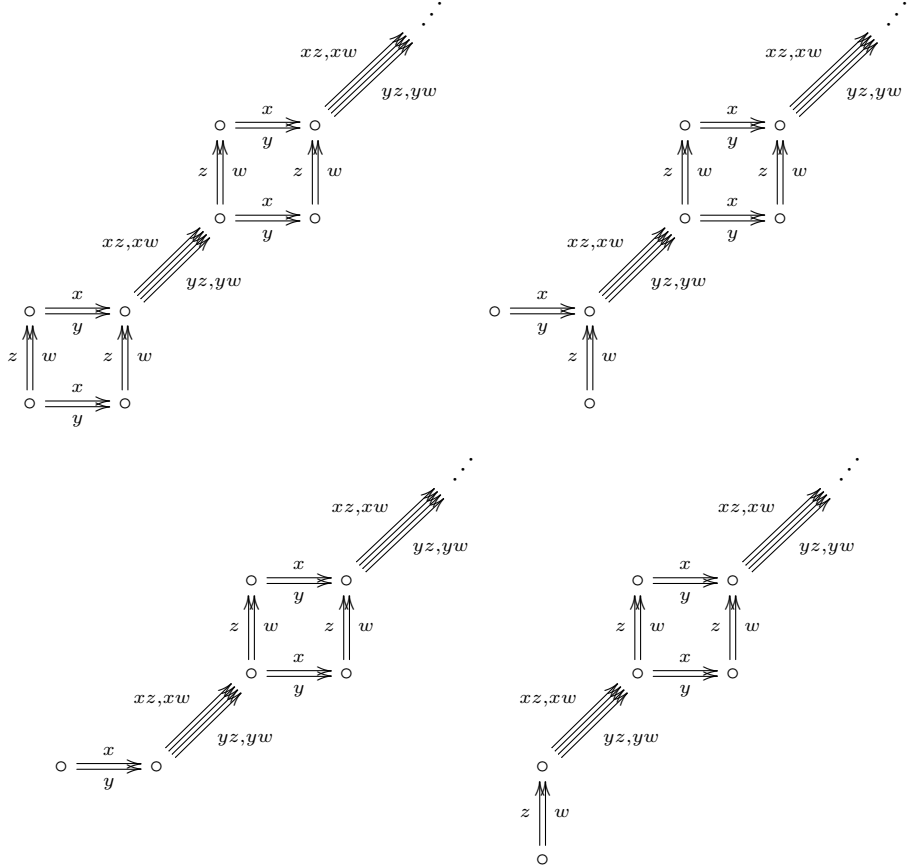
\[\xymatrix{
 & & & & \rotatebox{90}{$\ddots$} \\
 & & \circ \ar@2[r]^x_y & \circ \ar[ur]^4 \\
 & & \circ \ar@2[r]^x_y \ar@2[u]^z_w & \circ \ar@2[u]^z_w \\
\circ \ar@2[r]^x_y & \circ \ar[ur]^4 \\
\circ \ar@2[r]^x_y \ar@2[u]^z_w & \circ \ar@2[u]^z_w
}\xymatrix{
 & & & & \rotatebox{90}{$\ddots$} \\
 & & \circ \ar@2[r]^x_y & \circ \ar[ur]^4 \\
 & & \circ \ar@2[r]^x_y \ar@2[u]^z_w & \circ \ar@2[u]^z_w \\
\circ \ar@2[r]^x_y & \circ \ar[ur]^4 \\
 & \circ \ar@2[u]^z_w
}\]
\[\xymatrix{
 & & & & \rotatebox{90}{$\ddots$} \\
 & & \circ \ar@2[r]^x_y & \circ \ar[ur]^4 \\
 & & \circ \ar@2[r]^x_y \ar@2[u]^z_w & \circ \ar@2[u]^z_w \\
\circ \ar@2[r]^x_y & \circ \ar[ur]^4
}\xymatrix{
 & & & \rotatebox{90}{$\ddots$} \\
 & \circ \ar@2[r]^x_y & \circ \ar[ur]^4 \\
 & \circ \ar@2[r]^x_y \ar@2[u]^z_w & \circ \ar@2[u]^z_w \\
\circ \ar[ur]^4 \\
\circ \ar@2[u]^z_w
}\xymatrix{
 & & & \rotatebox{90}{$\ddots$} \\
 & \circ \ar@2[r]^x_y & \circ \ar[ur]^4 \\
 & \circ \ar@2[r]^x_y \ar@2[u]^z_w & \circ \ar@2[u]^z_w \\
\circ \ar[ur]^4
}\]

These upper sets correspond to the following cuts of $\widehat Q(\widehat C_2)$.
\[
\begin{tikzcd}[row sep=large]
\circ \arrow[r, "x"] \arrow[d, "w"]
& \circ \arrow[d, "w"]
& \circ \arrow[r, "x"] \arrow[l, "y"] \arrow[d, "w"]
& \circ \arrow[d, "w"]
\\
\bullet \arrow[r, "x"]
& \circ \arrow[ur, line width=2pt, no head] \arrow[ul, line width=2pt, no head] \arrow[dr, line width=2pt, no head] \arrow[dl, line width=2pt, no head]
& \bullet \arrow[r, "x"] \arrow[l, "y"]
& \circ \arrow[ul, line width=2pt, no head] \arrow[dl, line width=2pt, no head]
\\
\circ \arrow[r, "x"] \arrow[u, "z"] \arrow[d, "w"]
& \circ \arrow[u, "z"] \arrow[d, "w"]
& \circ \arrow[r, "x"] \arrow[l, "y"] \arrow[u, "z"] \arrow[d, "w"]
& \circ \arrow[u, "z"] \arrow[d, "w"]
\\
\bullet \arrow[r, "x"]
& \circ \arrow[ur, line width=2pt, no head] \arrow[ul, line width=2pt, no head]
& \bullet \arrow[r, "x"] \arrow[l, "y"] 
& \circ \arrow[ul, line width=2pt, no head]
\end{tikzcd}
\quad
\begin{tikzcd}[row sep=large]
\circ \arrow[r, line width=2pt, no head] \arrow[d, line width=2pt, no head]
& \circ \arrow[d, "w"]
& \circ \arrow[r, line width=2pt, no head] \arrow[l, line width=2pt, no head] \arrow[d, line width=2pt, no head]
& \circ \arrow[d, "w"]
\\
\bullet \arrow[r, "x"]
& \circ \arrow[ur, "xz"] \arrow[ul, "yz"] \arrow[dr, "xw"] \arrow[dl, "yw"]
& \bullet \arrow[r, "x"] \arrow[l, "y"]
& \circ \arrow[ul, "yz"] \arrow[dl, "yw"]
\\
\circ \arrow[r, line width=2pt, no head] \arrow[u, line width=2pt, no head] \arrow[d, line width=2pt, no head]
& \circ \arrow[u, "z"] \arrow[d, "w"]
& \circ \arrow[r, line width=2pt, no head] \arrow[l, line width=2pt, no head] \arrow[u, line width=2pt, no head] \arrow[d, line width=2pt, no head]
& \circ \arrow[u, "z"] \arrow[d, "w"]
\\
\bullet \arrow[r, "x"]
& \circ \arrow[ur, "xz"] \arrow[ul, "yz"]
& \bullet \arrow[r, "x"] \arrow[l, "y"] 
& \circ \arrow[ul, "yz"]
\end{tikzcd}
\]

\[
\begin{tikzcd}[row sep=large]
\circ \arrow[r, "x"] \arrow[d, line width=2pt, no head]
& \circ \arrow[d, line width=2pt, no head]
& \circ \arrow[r, "x"] \arrow[l, "y"] \arrow[d, line width=2pt, no head]
& \circ \arrow[d, line width=2pt, no head]
\\
\bullet \arrow[r, "x"]
& \circ \arrow[ur, "xz"] \arrow[ul, "yz"] \arrow[dr, "xw"] \arrow[dl, "yw"]
& \bullet \arrow[r, "x"] \arrow[l, "y"]
& \circ \arrow[ul, "yz"] \arrow[dl, "yw"]
\\
\circ \arrow[r, "x"] \arrow[u, line width=2pt, no head] \arrow[d, line width=2pt, no head]
& \circ \arrow[u, line width=2pt, no head] \arrow[d, line width=2pt, no head]
& \circ \arrow[r, "x"] \arrow[l, "y"] \arrow[u, line width=2pt, no head] \arrow[d, line width=2pt, no head]
& \circ \arrow[u, line width=2pt, no head] \arrow[d, line width=2pt, no head]
\\
\bullet \arrow[r, "x"]
& \circ \arrow[ur, "xz"] \arrow[ul, "yz"]
& \bullet \arrow[r, "x"] \arrow[l, "y"] 
& \circ \arrow[ul, "yz"]
\end{tikzcd}
\quad
\begin{tikzcd}[row sep=large]
\circ \arrow[r, line width=2pt, no head] \arrow[d, "w"]
& \circ \arrow[d, "w"]
& \circ \arrow[r, line width=2pt, no head] \arrow[l, line width=2pt, no head] \arrow[d, "w"]
& \circ \arrow[d, "w"]
\\
\bullet \arrow[r, line width=2pt, no head]
& \circ \arrow[ur, "xz"] \arrow[ul, "yz"] \arrow[dr, "xw"] \arrow[dl, "yw"]
& \bullet \arrow[r, line width=2pt, no head] \arrow[l, line width=2pt, no head]
& \circ \arrow[ul, "yz"] \arrow[dl, "yw"]
\\
\circ \arrow[r, line width=2pt, no head] \arrow[u, "z"] \arrow[d, "w"]
& \circ \arrow[u, "z"] \arrow[d, "w"]
& \circ \arrow[r, line width=2pt, no head] \arrow[l, line width=2pt, no head] \arrow[u, "z"] \arrow[d, "w"]
& \circ \arrow[u, "z"] \arrow[d, "w"]
\\
\bullet \arrow[r, line width=2pt, no head]
& \circ \arrow[ur, "xz"] \arrow[ul, "yz"]
& \bullet \arrow[r, line width=2pt, no head] \arrow[l, line width=2pt, no head]
& \circ \arrow[ul, "yz"]
\end{tikzcd}
\quad
\begin{tikzcd}[row sep=large]
\circ \arrow[r, line width=2pt, no head] \arrow[d, "w"]
& \circ \arrow[d, line width=2pt, no head]
& \circ \arrow[r, line width=2pt, no head] \arrow[l, line width=2pt, no head] \arrow[d, "w"]
& \circ \arrow[d, line width=2pt, no head]
\\
\bullet \arrow[r, "x"]
& \circ \arrow[ur, "xz"] \arrow[ul, "yz"] \arrow[dr, "xw"] \arrow[dl, "yw"]
& \bullet \arrow[r, "x"] \arrow[l, "y"]
& \circ \arrow[ul, "yz"] \arrow[dl, "yw"]
\\
\circ \arrow[r, line width=2pt, no head] \arrow[u, "z"] \arrow[d, "w"]
& \circ \arrow[u, line width=2pt, no head] \arrow[d, line width=2pt, no head]
& \circ \arrow[r, line width=2pt, no head] \arrow[l, line width=2pt, no head] \arrow[u, "z"] \arrow[d, "w"]
& \circ \arrow[u, line width=2pt, no head] \arrow[d, line width=2pt, no head]
\\
\bullet \arrow[r, "x"]
& \circ \arrow[ur, "xz"] \arrow[ul, "yz"]
& \bullet \arrow[r, "x"] \arrow[l, "y"] 
& \circ \arrow[ul, "yz"]
\end{tikzcd}
\]

Therefore there are the following five kinds of $2$-tilting bundles up to translations. Their endomorphism algebras are $2$-representation infinite algebras of type $\widetilde A\widetilde A$. Observe that by mutations of non-trivial upper sets in $q^{-1}(J)$, they are mutated to each other, which correspond to $2$-APR tilting mutations.

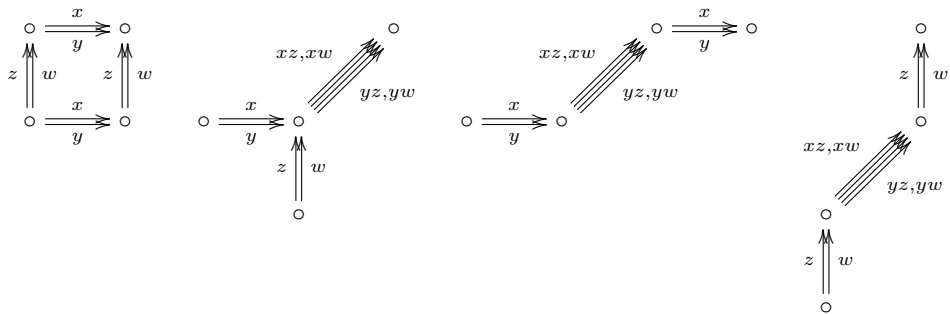
\[\begin{array}{c c}
\xymatrix{
\circ \ar@2[r]^x_y & \circ \\
\circ \ar@2[r]^x_y \ar@2[u]^z_w & \circ \ar@2[u]^z_w
}&\xymatrix{
 & & \circ \\
\circ \ar@2[r]^x_y & \circ \ar[ur]^4 \\
 & \circ \ar@2[u]^z_w
}\end{array}\]

\[\begin{array}{c c c}
\xymatrix{
 & & \circ \ar@2[r]^x_y & \circ \\
\circ \ar@2[r]^x_y & \circ \ar[ur]^4
}&\xymatrix{
 & \circ \\
 & \circ \ar@2[u]^z_w \\
\circ \ar[ur]^4 \\
\circ \ar@2[u]^z_w
}&\xymatrix{
 & \circ \\
 & \circ \ar@2[r]^x_y \ar@2[u]^z_w & \circ \\
\circ \ar[ur]^4
}\end{array}\]

(2) Put $G:=\mathbb{Z}^2$ and $\vec{x}=\vec{y}=(1,0),\vec{z}=(1,1),\vec{w}=(0,1)\in G$. We view $S:=k[x,y,z,w]$ as a $G$-graded $k$-algebra. Then the resulting toric stack $\X$ is isomorphic to $\Sigma_1$. We have $\vec{p}=(3,2)$ and $H=G/\mathbb{Z}\vec{p}\cong\mathbb{Z}; (a,b)+\mathbb{Z}\vec{p}\mapsto 2a-3b$. If we equip $H$ with our partial order, then the quiver of $H$ becomes the following.
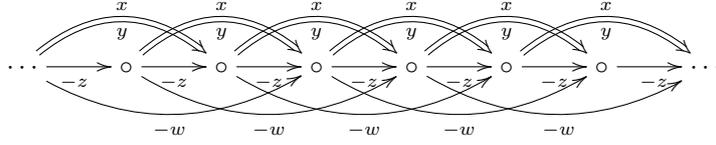
\[\xymatrix{
\cdots \ar[r]_{-z} \ar@2@/^18pt/[rr]^x_y \ar@/^-18pt/[rrr]_{-w} & \circ \ar[r]_{-z} \ar@2@/^18pt/[rr]^x_y \ar@/^-18pt/[rrr]_{-w} & \circ \ar[r]_{-z} \ar@2@/^18pt/[rr]^x_y \ar@/^-18pt/[rrr]_{-w} & \circ \ar[r]_{-z} \ar@2@/^18pt/[rr]^x_y \ar@/^-18pt/[rrr]_{-w} & \circ \ar[r]_{-z} \ar@2@/^18pt/[rr]^x_y \ar@/^-18pt/[rrr]_{-w} & \circ \ar[r]_{-z} \ar@2@/^18pt/[rr]^x_y & \circ \ar[r]_{-z} & \cdots
}\]
Then there is the following just one kind of non-trivial upper sets in $H$ up to translations.
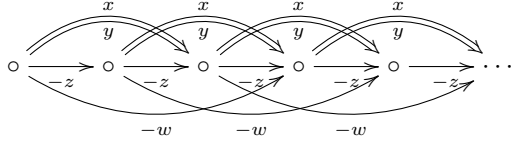
\[\xymatrix{
\circ \ar[r]_{-z} \ar@2@/^18pt/[rr]^x_y \ar@/^-18pt/[rrr]_{-w} & \circ \ar[r]_{-z} \ar@2@/^18pt/[rr]^x_y \ar@/^-18pt/[rrr]_{-w} & \circ \ar[r]_{-z} \ar@2@/^18pt/[rr]^x_y \ar@/^-18pt/[rrr]_{-w} & \circ \ar[r]_{-z} \ar@2@/^18pt/[rr]^x_y & \circ \ar[r]_{-z} & \cdots
}\]
Remark that the isomorphism $H\cong\mathbb{Z}$ sends $s$ to $4$. If we let $\widehat{C}$ be the cut of $\widehat{Q}$ corresponding to an upper set in $H$, then $\widehat{Q}(\widehat{C})$ is as follows. Here, the black vertices correspond to elements of $B$.
\[
\begin{tikzcd}[row sep=large]
\bullet \arrow[r, "x"]
& \circ \arrow[dr, "xw"] \arrow[dl, "yw"]
& \bullet \arrow[r, "x"] \arrow[l, "y"]
& \circ \arrow[dr, "xw"] \arrow[dl, "yw"]
& \bullet \arrow[l, "y"]
\\
\circ \arrow[r, "x"] \arrow[u, "z"]
& \circ \arrow[u, "z"] \arrow[d, "w"]
& \circ \arrow[r, "x"] \arrow[l, "y"] \arrow[u, "z"]
& \circ \arrow[u, "z"] \arrow[d, "w"]
& \circ \arrow[l, "y"] \arrow[u, "z"]
\\
\circ \arrow[u, "z"] \arrow[dr, "xw"]
& \bullet \arrow[r, "x"] \arrow[l, "y"]
& \circ \arrow[u, "z"] \arrow[dr, "xw"] \arrow[dl, "yw"]
& \bullet \arrow[r, "x"] \arrow[l, "y"]
& \circ \arrow[u, "z"] \arrow[dl, "yw"]
\\
\circ \arrow[u, "z"] \arrow[d, "w"]
& \circ \arrow[r, "x"] \arrow[l, "y"] \arrow[u, "z"]
& \circ \arrow[u, "z"] \arrow[d, "w"]
& \circ \arrow[r, "x"] \arrow[l, "y"] \arrow[u, "z"]
& \circ \arrow[u, "z"] \arrow[d, "w"]
\\
\bullet \arrow[r, "x"]
& \circ \arrow[u, "z"]
& \bullet \arrow[r, "x"] \arrow[l, "y"]
& \circ \arrow[u, "z"]
& \bullet \arrow[l, "y"]
\end{tikzcd}
\]

Thus there are the following just one kind of non-trivial upper set in $\J_H$ up to translations with the following quiver $Q(C)$.
\[\xymatrix{
\circ \ar@2@/^18pt/[rr]^x_y & \circ \ar[l]_z \ar@2@/^18pt/[rr]^x_y & \circ \ar@3[l]^{xw,yw}_z & \circ \ar[l]_z \ar@/^18pt/[lll]_w
}\]
The quiver of $q^{-1}(J)$ is as follows.
\[\xymatrix{
 & & & & & & & \rotatebox{90}{$\ddots$} \\
 & & & & & \circ \ar@2[r]^x_y & \circ \ar@3[ur]^z_{xw,yw} \\
 & & & & \circ \ar@2[r]^x_y \ar[ur]^z & \circ \ar[u]^w \ar[ur]_z \\
 & & \circ \ar@2[r]^x_y & \circ \ar@3[ur]^z_{xw,yw} \\
 & \circ \ar@2[r]^x_y \ar[ur]^z & \circ \ar[u]^w \ar[ur]_z \\
\rotatebox{90}{$\ddots$} \ar@3[ur]^z_{xw,yw}
}\]
There are the following four kinds of non-trivial upper sets in $q^{-1}(J)$ up to translations by $\vec{p}$.
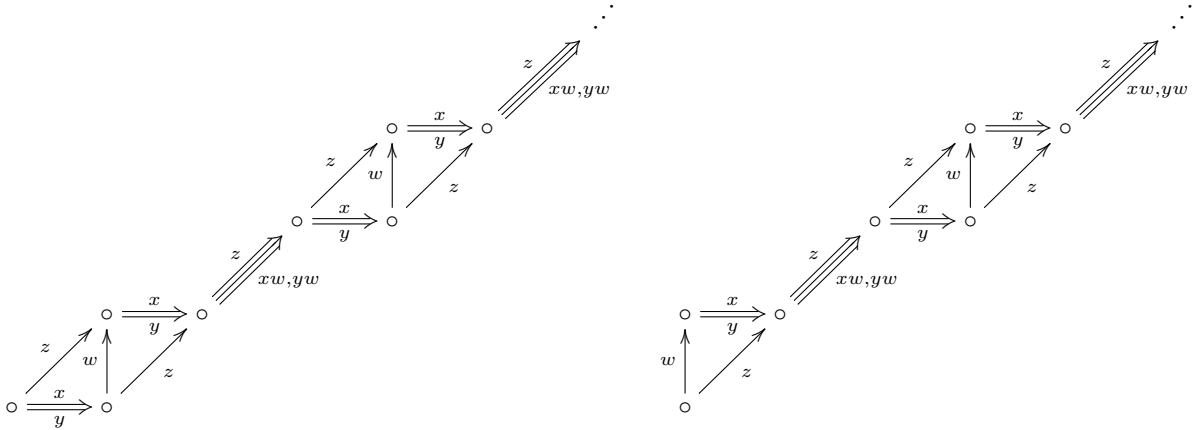
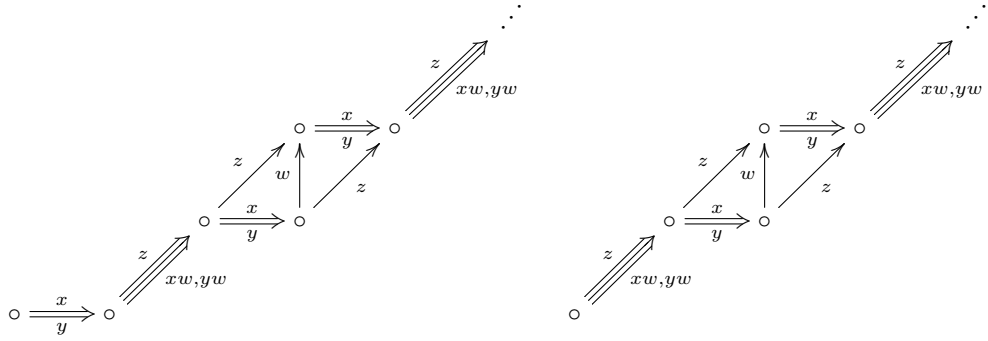
\[\adjustbox{max width=\linewidth}{$\displaystyle \begin{array}{c c}
\xymatrix{
 & & & & & & \rotatebox{90}{$\ddots$} \\
 & & & & \circ \ar@2[r]^x_y & \circ \ar@3[ur]^z_{xw,yw} \\
 & & & \circ \ar@2[r]^x_y \ar[ur]^z & \circ \ar[u]^w \ar[ur]_z \\
 & \circ \ar@2[r]^x_y & \circ \ar@3[ur]^z_{xw,yw} \\
\circ \ar@2[r]^x_y \ar[ur]^z & \circ \ar[u]^w \ar[ur]_z
}&\xymatrix{
 & & & & & \rotatebox{90}{$\ddots$} \\
 & & & \circ \ar@2[r]^x_y & \circ \ar@3[ur]^z_{xw,yw} \\
 & & \circ \ar@2[r]^x_y \ar[ur]^z & \circ \ar[u]^w \ar[ur]_z \\
\circ \ar@2[r]^x_y & \circ \ar@3[ur]^z_{xw,yw} \\
\circ \ar[u]^w \ar[ur]_z
}\end{array}$}\]
\[\begin{array}{c c}
\xymatrix{
 & & & & & \rotatebox{90}{$\ddots$} \\
 & & & \circ \ar@2[r]^x_y & \circ \ar@3[ur]^z_{xw,yw} \\
 & & \circ \ar@2[r]^x_y \ar[ur]^z & \circ \ar[u]^w \ar[ur]_z \\
\circ \ar@2[r]^x_y & \circ \ar@3[ur]^z_{xw,yw}
}&\xymatrix{
 & & & & \rotatebox{90}{$\ddots$} \\
 & & \circ \ar@2[r]^x_y & \circ \ar@3[ur]^z_{xw,yw} \\
 & \circ \ar@2[r]^x_y \ar[ur]^z & \circ \ar[u]^w \ar[ur]_z \\
\circ \ar@3[ur]^z_{xw,yw}
}\end{array}\]

These upper sets correspond to the following cuts of $\widehat Q(\widehat C)$.

\[
\begin{tikzcd}[row sep=large]
\bullet \arrow[r, "x"]
& \circ \arrow[dr, line width=2pt, no head] \arrow[dl, line width=2pt, no head]
& \bullet \arrow[r, "x"] \arrow[l, "y"]
& \circ \arrow[dr, line width=2pt, no head] \arrow[dl, line width=2pt, no head]
& \bullet \arrow[l, "y"]
\\
\circ \arrow[r, "x"] \arrow[u, "z"]
& \circ \arrow[u, "z"] \arrow[d, "w"]
& \circ \arrow[r, "x"] \arrow[l, "y"] \arrow[u, "z"]
& \circ \arrow[u, "z"] \arrow[d, "w"]
& \circ \arrow[l, "y"] \arrow[u, "z"]
\\
\circ \arrow[u, line width=2pt, no head] \arrow[dr, line width=2pt, no head]
& \bullet \arrow[r, "x"] \arrow[l, "y"]
& \circ \arrow[u, line width=2pt, no head] \arrow[dr, line width=2pt, no head] \arrow[dl, line width=2pt, no head]
& \bullet \arrow[r, "x"] \arrow[l, "y"]
& \circ \arrow[u, line width=2pt, no head] \arrow[dl, line width=2pt, no head]
\\
\circ \arrow[u, "z"] \arrow[d, "w"]
& \circ \arrow[r, "x"] \arrow[l, "y"] \arrow[u, "z"]
& \circ \arrow[u, "z"] \arrow[d, "w"]
& \circ \arrow[r, "x"] \arrow[l, "y"] \arrow[u, "z"]
& \circ \arrow[u, "z"] \arrow[d, "w"]
\\
\bullet \arrow[r, "x"]
& \circ \arrow[u, line width=2pt, no head]
& \bullet \arrow[r, "x"] \arrow[l, "y"]
& \circ \arrow[u, line width=2pt, no head]
& \bullet \arrow[l, "y"]
\end{tikzcd}
\quad
\begin{tikzcd}[row sep=large]
\bullet \arrow[r, "x"]
& \circ \arrow[dr, "xw"] \arrow[dl, "yw"]
& \bullet \arrow[r, "x"] \arrow[l, "y"]
& \circ \arrow[dr, "xw"] \arrow[dl, "yw"]
& \bullet \arrow[l, "y"]
\\
\circ \arrow[r, line width=2pt, no head] \arrow[u, line width=2pt, no head]
& \circ \arrow[u, "z"] \arrow[d, "w"]
& \circ \arrow[r, line width=2pt, no head] \arrow[l, line width=2pt, no head] \arrow[u, line width=2pt, no head]
& \circ \arrow[u, "z"] \arrow[d, "w"]
& \circ \arrow[l, line width=2pt, no head] \arrow[u, line width=2pt, no head]
\\
\circ \arrow[u, "z"] \arrow[dr, "xw"]
& \bullet \arrow[r, "x"] \arrow[l, "y"]
& \circ \arrow[u, "z"] \arrow[dr, "xw"] \arrow[dl, "yw"]
& \bullet \arrow[r, "x"] \arrow[l, "y"]
& \circ \arrow[u, "z"] \arrow[dl, "yw"]
\\
\circ \arrow[u, "z"] \arrow[d, "w"]
& \circ \arrow[r, line width=2pt, no head] \arrow[l, line width=2pt, no head] \arrow[u, line width=2pt, no head]
& \circ \arrow[u, "z"] \arrow[d, "w"]
& \circ \arrow[r, line width=2pt, no head] \arrow[l, line width=2pt, no head] \arrow[u, line width=2pt, no head]
& \circ \arrow[u, "z"] \arrow[d, "w"]
\\
\bullet \arrow[r, "x"]
& \circ \arrow[u, "z"]
& \bullet \arrow[r, "x"] \arrow[l, "y"]
& \circ \arrow[u, "z"]
& \bullet \arrow[l, "y"]
\end{tikzcd}
\]

\[
\begin{tikzcd}[row sep=large]
\bullet \arrow[r, "x"]
& \circ \arrow[dr, "xw"] \arrow[dl, "yw"]
& \bullet \arrow[r, "x"] \arrow[l, "y"]
& \circ \arrow[dr, "xw"] \arrow[dl, "yw"]
& \bullet \arrow[l, "y"]
\\
\circ \arrow[r, "x"] \arrow[u, line width=2pt, no head]
& \circ \arrow[u, line width=2pt, no head] \arrow[d, line width=2pt, no head]
& \circ \arrow[r, "x"] \arrow[l, "y"] \arrow[u, line width=2pt, no head]
& \circ \arrow[u, line width=2pt, no head] \arrow[d, line width=2pt, no head]
& \circ \arrow[l, "y"] \arrow[u, line width=2pt, no head]
\\
\circ \arrow[u, "z"] \arrow[dr, "xw"]
& \bullet \arrow[r, "x"] \arrow[l, "y"]
& \circ \arrow[u, "z"] \arrow[dr, "xw"] \arrow[dl, "yw"]
& \bullet \arrow[r, "x"] \arrow[l, "y"]
& \circ \arrow[u, "z"] \arrow[dl, "yw"]
\\
\circ \arrow[u, line width=2pt, no head] \arrow[d, line width=2pt, no head]
& \circ \arrow[r, "x"] \arrow[l, "y"] \arrow[u, line width=2pt, no head]
& \circ \arrow[u, line width=2pt, no head] \arrow[d, line width=2pt, no head]
& \circ \arrow[r, "x"] \arrow[l, "y"] \arrow[u, line width=2pt, no head]
& \circ \arrow[u, line width=2pt, no head] \arrow[d, line width=2pt, no head]
\\
\bullet \arrow[r, "x"]
& \circ \arrow[u, "z"]
& \bullet \arrow[r, "x"] \arrow[l, "y"]
& \circ \arrow[u, "z"]
& \bullet \arrow[l, "y"]
\end{tikzcd}
\quad
\begin{tikzcd}[row sep=large]
\bullet \arrow[r, line width=2pt, no head]
& \circ \arrow[dr, "xw"] \arrow[dl, "yw"]
& \bullet \arrow[r, line width=2pt, no head] \arrow[l, line width=2pt, no head]
& \circ \arrow[dr, "xw"] \arrow[dl, "yw"]
& \bullet \arrow[l, line width=2pt, no head]
\\
\circ \arrow[r, "x"] \arrow[u, "z"]
& \circ \arrow[u, line width=2pt, no head] \arrow[d, "w"]
& \circ \arrow[r, "x"] \arrow[l, "y"] \arrow[u, "z"]
& \circ \arrow[u, line width=2pt, no head] \arrow[d, "w"]
& \circ \arrow[l, "y"] \arrow[u, "z"]
\\
\circ \arrow[u, "z"] \arrow[dr, "xw"]
& \bullet \arrow[r, line width=2pt, no head] \arrow[l, line width=2pt, no head]
& \circ \arrow[u, "z"] \arrow[dr, "xw"] \arrow[dl, "yw"]
& \bullet \arrow[r, line width=2pt, no head] \arrow[l, line width=2pt, no head]
& \circ \arrow[u, "z"] \arrow[dl, "yw"]
\\
\circ \arrow[u, line width=2pt, no head] \arrow[d, "w"]
& \circ \arrow[r, "x"] \arrow[l, "y"] \arrow[u, "z"]
& \circ \arrow[u, line width=2pt, no head] \arrow[d, "w"]
& \circ \arrow[r, "x"] \arrow[l, "y"] \arrow[u, "z"]
& \circ \arrow[u, line width=2pt, no head] \arrow[d, "w"]
\\
\bullet \arrow[r, line width=2pt, no head]
& \circ \arrow[u, "z"]
& \bullet \arrow[r, line width=2pt, no head] \arrow[l, line width=2pt, no head]
& \circ \arrow[u, "z"]
& \bullet \arrow[l, line width=2pt, no head]
\end{tikzcd}
\]

Therefore there are the following four kinds of $2$-tilting bundles up to translations. Their endomorphism algebras are $2$-representation infinite algebras of type $\widetilde A\widetilde A$. Observe that by mutations of non-trivial upper sets in $q^{-1}(J)$, they are mutated to each other, which correspond to $2$-APR tilting mutations.
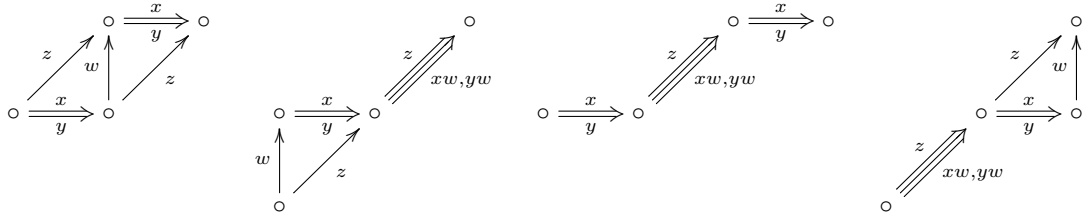
\[\begin{array}{c c c c}
\xymatrix{
 & \circ \ar@2[r]^x_y & \circ \\
\circ \ar@2[r]^x_y \ar[ur]^z & \circ \ar[u]^w \ar[ur]_z
}&\xymatrix{
 & & \circ \\
\circ \ar@2[r]^x_y & \circ \ar@3[ur]^z_{xw,yw} \\
\circ \ar[u]^w \ar[ur]_z
}&\xymatrix{
 & & \circ \ar@2[r]^x_y & \circ \\
\circ \ar@2[r]^x_y & \circ \ar@3[ur]^z_{xw,yw}
}&\xymatrix{
 & & \circ \\
 & \circ \ar@2[r]^x_y \ar[ur]^z & \circ \ar[u]^w \\
\circ \ar@3[ur]^z_{xw,yw}
}\end{array}\]
\end{Ex}

Finally, we see a stacky example. 

\begin{Ex}($d=2$)\label{Exstackysurface}
Put $G:=\mathbb{Z}^2$ and $\vec{x}=(1,-1),\vec{y}=(1,0),\vec{z}=(1,1),\vec{w}=(0,1)\in G$. We view $S:=k[x,y,z,w]$ as a $G$-graded $k$-algebra. We have $\vec{p}=(3,1)$ and $H=G/\mathbb{Z}\vec{p}\cong\mathbb{Z}; (a,b)+\mathbb{Z}\vec{p}\mapsto a-3b$. If we equip $H$ with our partial order, then the quiver of $H$ becomes the following.
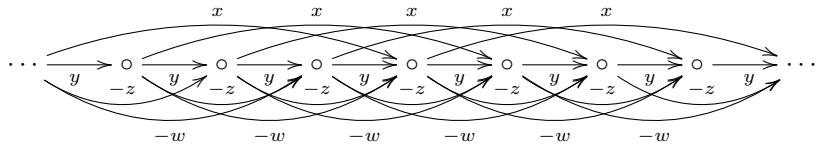
\[\xymatrix{
\cdots \ar[r]_y \ar@/^-15pt/[rr]^{-z} \ar@/^-21pt/[rrr]_{-w} \ar@/^15pt/[rrrr]^x & \circ \ar[r]_y \ar@/^-15pt/[rr]^{-z} \ar@/^-21pt/[rrr]_{-w} \ar@/^15pt/[rrrr]^x & \circ \ar[r]_y \ar@/^-15pt/[rr]^{-z} \ar@/^-21pt/[rrr]_{-w}  \ar@/^15pt/[rrrr]^x & \circ \ar[r]_y \ar@/^-15pt/[rr]^{-z} \ar@/^-21pt/[rrr]_{-w} \ar@/^15pt/[rrrr]^x & \circ \ar[r]_y \ar@/^-15pt/[rr]^{-z} \ar@/^-21pt/[rrr]_{-w} \ar@/^15pt/[rrrr]^x & \circ \ar[r]_y \ar@/^-15pt/[rr]^{-z} \ar@/^-21pt/[rrr]_{-w} & \circ \ar[r]_y \ar@/^-15pt/[rr]^{-z} & \circ \ar[r]_y & \cdots
}\]
Then there is the following just one kind of non-trivial upper set in $H$ up to translations.
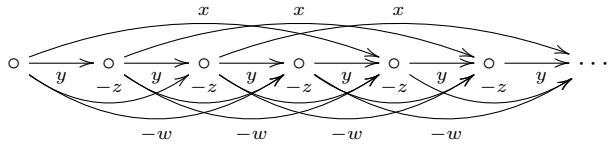
\[\xymatrix{
\circ \ar[r]_y \ar@/^-15pt/[rr]^{-z} \ar@/^-21pt/[rrr]_{-w}  \ar@/^15pt/[rrrr]^x & \circ \ar[r]_y \ar@/^-15pt/[rr]^{-z} \ar@/^-21pt/[rrr]_{-w} \ar@/^15pt/[rrrr]^x & \circ \ar[r]_y \ar@/^-15pt/[rr]^{-z} \ar@/^-21pt/[rrr]_{-w} \ar@/^15pt/[rrrr]^x & \circ \ar[r]_y \ar@/^-15pt/[rr]^{-z} \ar@/^-21pt/[rrr]_{-w} & \circ \ar[r]_y \ar@/^-15pt/[rr]^{-z} & \circ \ar[r]_y & \cdots
}\]
Remark that the isomorphism $H\cong\mathbb{Z}$ sends $s$ to $5$. If we let $\widehat{C}$ be the cut of $\widehat{Q}$ corresponding to an upper set in $H$, then $\widehat{Q}(\widehat{C})$ is as follows. Here, the black vertices correspond to elements of $B$.
\[
\begin{tikzcd}[row sep=large]
\circ \arrow[d, "w"]
& \circ \arrow[l, "y"] \arrow[dr, "xw"]
& \circ \arrow[l, "y"] \arrow[dr, "xw"]
& \bullet \arrow[r, "x"] \arrow[l, "y"]
& \circ \arrow[d, "w"]
& \circ \arrow[l, "y"] \arrow[d, "w"]
& \circ \arrow[l, "y"]
\\
\bullet \arrow[r, "x"]
& \circ \arrow[u, "z"] \arrow[d, "w"]
& \circ \arrow[l, "y"] \arrow[u, "z"] \arrow[d, "w"]
& \circ \arrow[l, "y"] \arrow[u, "z"] \arrow[dr, "xw"]
& \circ \arrow[l, "y"] \arrow[ur, "xz"] \arrow[dr, "xw"]
& \bullet \arrow[r, "x"] \arrow[l, "y"]
& \circ \arrow[u, "z"] \arrow[d, "w"]
\\
\circ \arrow[u, "z"] \arrow[dr, "xw"]
& \circ \arrow[l, "y"] \arrow[ur, "xz"] \arrow[dr, "xw"]
& \bullet \arrow[r, "x"] \arrow[l, "y"]
& \circ \arrow[u, "z"] \arrow[d, "w"]
& \circ \arrow[l, "y"] \arrow[u, "z"] \arrow[d, "w"]
& \circ \arrow[l, "y"] \arrow[u, "z"] \arrow[dr, "xw"]
& \circ \arrow[l, "y"]
\\
\circ \arrow[u, "z"]\arrow[d, "w"]
& \circ \arrow[l, "y"] \arrow[u, "z"] \arrow[d, "w"]
& \circ \arrow[l, "y"] \arrow[u, "z"] \arrow[dr, "xw"]
& \circ \arrow[l, "y"] \arrow[ur, "xz"] \arrow[dr, "xw"]
& \bullet \arrow[r, "x"] \arrow[l, "y"]
& \circ \arrow[u, "z"] \arrow[d, "w"]
& \circ \arrow[l, "y"] \arrow[u, "z"] \arrow[d, "w"]
\\
\circ \arrow[ur, "xz"] \arrow[dr, "xw"]
& \bullet \arrow[r, "x"] \arrow[l, "y"]
& \circ \arrow[u, "z"] \arrow[d, "w"]
& \circ \arrow[l, "y"] \arrow[u, "z"] \arrow[d, "w"]
& \circ \arrow[l, "y"] \arrow[u, "z"] \arrow[dr, "xw"]
& \circ \arrow[l, "y"] \arrow[ur, "xz"] \arrow[dr, "xw"]
& \bullet \arrow[l, "y"]
\\
\circ \arrow[u, "z"] \arrow[d, "w"]
& \circ \arrow[l, "y"] \arrow[u, "z"] \arrow[dr, "xw"]
& \circ \arrow[l, "y"] \arrow[ur, "xz"] \arrow[dr, "xw"]
& \bullet \arrow[r, "x"] \arrow[l, "y"]
& \circ \arrow[u, "z"] \arrow[d, "w"]
& \circ \arrow[l, "y"] \arrow[u, "z"] \arrow[d, "w"]
& \circ \arrow[l, "y"] \arrow[u, "z"]
\\
\bullet \arrow[r, "x"]
& \circ \arrow[u, "z"]
& \circ \arrow[l, "y"] \arrow[u, "z"]
& \circ \arrow[l, "y"] \arrow[u, "z"]
& \circ \arrow[l, "y"] \arrow[ur, "xz"]
& \bullet \arrow[r, "x"] \arrow[l, "y"]
& \circ \arrow[u, "z"]
\\
\end{tikzcd}
\]

Thus there are the following just one kind of non-trivial upper set in $\J_H$ up to translations with the following quiver $Q(C)$.
\[\xymatrix{
\circ \ar[r]^y \ar@/^20pt/[rrrr]^x & \circ \ar@2[r]^y_{xw} \ar@/^12pt/[rr]^{xz} & \circ \ar@2[r]^y_{xw} \ar@/^15pt/[ll]_z & \circ \ar[r]^y \ar@/^15pt/[ll]_z \ar@/^21pt/[lll]^w & \circ \ar@/^15pt/[ll]_z \ar@/^21pt/[lll]^w
}\]
The quiver of $q^{-1}(J)$ is as follows.
\[\xymatrix{
 & & & & & \circ \ar[r]^y \ar[dr]^x & \circ \ar@2[r]^y_{xw} \ar@/^18pt/[rr]^{xz} & \circ \ar@2[r]^y_{xw} & \cdots\\
 & & \circ \ar[r]^y \ar[dr]^x & \circ \ar@2[r]^y_{xw} \ar@/^18pt/[rr]^{xz} & \circ \ar@2[r]^y_{xw} \ar[ur]^z & \circ \ar[r]^y \ar[ur]^z \ar[u]^w & \circ \ar[ur]_z \ar[u]^w \\
\cdots \ar@2[r]^y_{xw} \ar@/^18pt/[rr]^{xz} & \circ \ar@2[r]^y_{xw} \ar[ur]^z & \circ \ar[r]^y \ar[ur]^z \ar[u]^w & \circ \ar[ur]_z \ar[u]^w
}\]
There are the following five kinds of non-trivial upper sets in $q^{-1}(J)$ up to translations by $\vec{p}$.
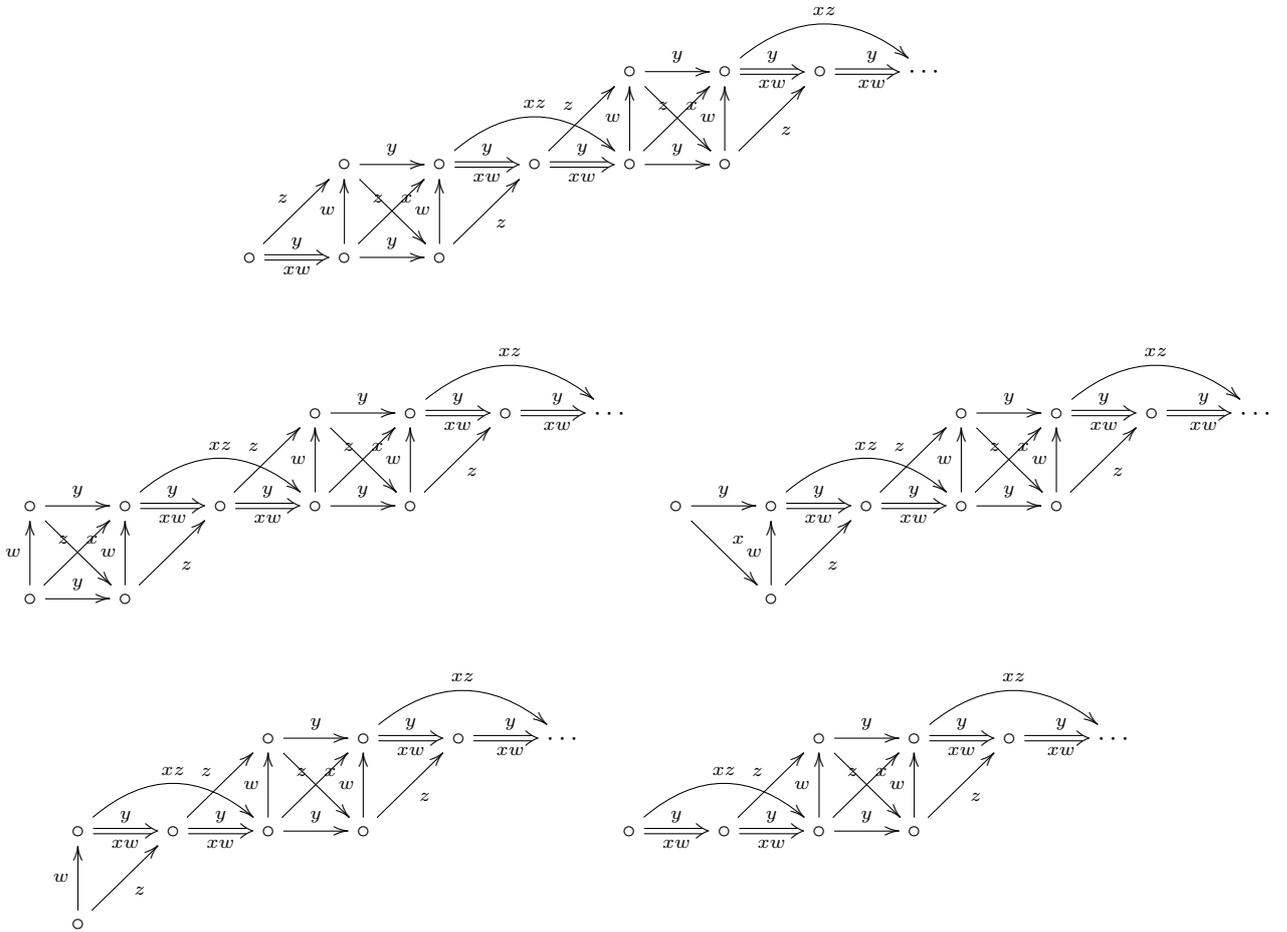
\[\xymatrix{
 & & & & \circ \ar[r]^y \ar[dr]^x & \circ \ar@2[r]^y_{xw} \ar@/^18pt/[rr]^{xz} & \circ \ar@2[r]^y_{xw} & \cdots\\
 & \circ \ar[r]^y \ar[dr]^x & \circ \ar@2[r]^y_{xw} \ar@/^18pt/[rr]^{xz} & \circ \ar@2[r]^y_{xw} \ar[ur]^z & \circ \ar[r]^y \ar[ur]^z \ar[u]^w & \circ \ar[ur]_z \ar[u]^w \\
\circ \ar@2[r]^y_{xw} \ar[ur]^z & \circ \ar[r]^y \ar[ur]^z \ar[u]^w & \circ \ar[ur]_z \ar[u]^w
}\]
\[\adjustbox{max width=\linewidth}{$\displaystyle \begin{array}{c c}
\xymatrix{
 & & & \circ \ar[r]^y \ar[dr]^x & \circ \ar@2[r]^y_{xw} \ar@/^18pt/[rr]^{xz} & \circ \ar@2[r]^y_{xw} & \cdots\\
\circ \ar[r]^y \ar[dr]^x & \circ \ar@2[r]^y_{xw} \ar@/^18pt/[rr]^{xz} & \circ \ar@2[r]^y_{xw} \ar[ur]^z & \circ \ar[r]^y \ar[ur]^z \ar[u]^w & \circ \ar[ur]_z \ar[u]^w \\
\circ \ar[r]^y \ar[ur]^z \ar[u]^w & \circ \ar[ur]_z \ar[u]^w
}&\xymatrix{
 & & & \circ \ar[r]^y \ar[dr]^x & \circ \ar@2[r]^y_{xw} \ar@/^18pt/[rr]^{xz} & \circ \ar@2[r]^y_{xw} & \cdots\\
\circ \ar[r]^y \ar[dr]^x & \circ \ar@2[r]^y_{xw} \ar@/^18pt/[rr]^{xz} & \circ \ar@2[r]^y_{xw} \ar[ur]^z & \circ \ar[r]^y \ar[ur]^z \ar[u]^w & \circ \ar[ur]_z \ar[u]^w \\
 & \circ \ar[ur]_z \ar[u]^w
}\end{array}$}\]
\[\begin{array}{c c}
\xymatrix{
 & & \circ \ar[r]^y \ar[dr]^x & \circ \ar@2[r]^y_{xw} \ar@/^18pt/[rr]^{xz} & \circ \ar@2[r]^y_{xw} & \cdots\\
\circ \ar@2[r]^y_{xw} \ar@/^18pt/[rr]^{xz} & \circ \ar@2[r]^y_{xw} \ar[ur]^z & \circ \ar[r]^y \ar[ur]^z \ar[u]^w & \circ \ar[ur]_z \ar[u]^w \\
\circ \ar[ur]_z \ar[u]^w
}&\xymatrix{
 & & \circ \ar[r]^y \ar[dr]^x & \circ \ar@2[r]^y_{xw} \ar@/^18pt/[rr]^{xz} & \circ \ar@2[r]^y_{xw} & \cdots\\
\circ \ar@2[r]^y_{xw} \ar@/^18pt/[rr]^{xz} & \circ \ar@2[r]^y_{xw} \ar[ur]^z & \circ \ar[r]^y \ar[ur]^z \ar[u]^w & \circ \ar[ur]_z \ar[u]^w
}\end{array}\]

These upper sets correspond to the following cuts of $\widehat Q(\widehat C)$.

\[
\begin{tikzcd}[row sep=1.8em, column sep=1.8em]
\circ \arrow[d, "w"]
& \circ \arrow[l, "y"] \arrow[dr, "xw"]
& \circ \arrow[l, line width=2pt, no head] \arrow[dr, line width=2pt, no head]
& \bullet \arrow[r, "x"] \arrow[l, "y"]
& \circ \arrow[d, "w"]
& \circ \arrow[l, "y"] \arrow[d, "w"]
& \circ \arrow[l, "y"]
\\
\bullet \arrow[r, "x"]
& \circ \arrow[u, line width=2pt, no head] \arrow[d, "w"]
& \circ \arrow[l, "y"] \arrow[u, "z"] \arrow[d, "w"]
& \circ \arrow[l, "y"] \arrow[u, "z"] \arrow[dr, "xw"]
& \circ \arrow[l, line width=2pt, no head] \arrow[ur, line width=2pt, no head] \arrow[dr, line width=2pt, no head]
& \bullet \arrow[r, "x"] \arrow[l, "y"]
& \circ \arrow[u, line width=2pt, no head] \arrow[d, "w"]
\\
\circ \arrow[u, "z"] \arrow[dr, "xw"]
& \circ \arrow[l, line width=2pt, no head] \arrow[ur, line width=2pt, no head] \arrow[dr, line width=2pt, no head]
& \bullet \arrow[r, "x"] \arrow[l, "y"]
& \circ \arrow[u, line width=2pt, no head] \arrow[d, "w"]
& \circ \arrow[l, "y"] \arrow[u, "z"] \arrow[d, "w"]
& \circ \arrow[l, "y"] \arrow[u, "z"] \arrow[dr, "xw"]
& \circ \arrow[l, line width=2pt, no head]
\\
\circ \arrow[u, line width=2pt, no head]\arrow[d, "w"]
& \circ \arrow[l, "y"] \arrow[u, "z"] \arrow[d, "w"]
& \circ \arrow[l, "y"] \arrow[u, "z"] \arrow[dr, "xw"]
& \circ \arrow[l, line width=2pt, no head] \arrow[ur, line width=2pt, no head] \arrow[dr, line width=2pt, no head]
& \bullet \arrow[r, "x"] \arrow[l, "y"]
& \circ \arrow[u, line width=2pt, no head] \arrow[d, "w"]
& \circ \arrow[l, "y"] \arrow[u, "z"] \arrow[d, "w"]
\\
\circ \arrow[ur, line width=2pt, no head] \arrow[dr, line width=2pt, no head]
& \bullet \arrow[r, "x"] \arrow[l, "y"]
& \circ \arrow[u, line width=2pt, no head] \arrow[d, "w"]
& \circ \arrow[l, "y"] \arrow[u, "z"] \arrow[d, "w"]
& \circ \arrow[l, "y"] \arrow[u, "z"] \arrow[dr, "xw"]
& \circ \arrow[l, line width=2pt, no head] \arrow[ur, line width=2pt, no head] \arrow[dr, line width=2pt, no head]
& \bullet \arrow[l, "y"]
\\
\circ \arrow[u, "z"] \arrow[d, "w"]
& \circ \arrow[l, "y"] \arrow[u, "z"] \arrow[dr, "xw"]
& \circ \arrow[l, line width=2pt, no head] \arrow[ur, line width=2pt, no head] \arrow[dr, line width=2pt, no head]
& \bullet \arrow[r, "x"] \arrow[l, "y"]
& \circ \arrow[u, line width=2pt, no head] \arrow[d, "w"]
& \circ \arrow[l, "y"] \arrow[u, "z"] \arrow[d, "w"]
& \circ \arrow[l, "y"] \arrow[u, "z"]
\\
\bullet \arrow[r, "x"]
& \circ \arrow[u, line width=2pt, no head]
& \circ \arrow[l, "y"] \arrow[u, "z"]
& \circ \arrow[l, "y"] \arrow[u, "z"]
& \circ \arrow[l, line width=2pt, no head] \arrow[ur, line width=2pt, no head]
& \bullet \arrow[r, "x"] \arrow[l, "y"]
& \circ \arrow[u, line width=2pt, no head]
\\
\end{tikzcd}
\]
\[
\begin{tikzcd}[row sep=1.8em, column sep=1.8em]
\circ \arrow[d, "w"]
& \circ \arrow[l, line width=2pt, no head] \arrow[dr, line width=2pt, no head]
& \circ \arrow[l, "y"] \arrow[dr, "xw"]
& \bullet \arrow[r, "x"] \arrow[l, "y"]
& \circ \arrow[d, "w"]
& \circ \arrow[l, "y"] \arrow[d, "w"]
& \circ \arrow[l, line width=2pt, no head]
\\
\bullet \arrow[r, "x"]
& \circ \arrow[u, "z"] \arrow[d, "w"]
& \circ \arrow[l, "y"] \arrow[u, "z"] \arrow[d, "w"]
& \circ \arrow[l, line width=2pt, no head] \arrow[u, line width=2pt, no head] \arrow[dr, line width=2pt, no head]
& \circ \arrow[l, "y"] \arrow[ur, line width=2pt, no head] \arrow[dr, "xw"]
& \bullet \arrow[r, "x"] \arrow[l, "y"]
& \circ \arrow[u, "z"] \arrow[d, "w"]
\\
\circ \arrow[u, line width=2pt, no head] \arrow[dr, line width=2pt, no head]
& \circ \arrow[l, "y"] \arrow[ur, line width=2pt, no head] \arrow[dr, "xw"]
& \bullet \arrow[r, "x"] \arrow[l, "y"]
& \circ \arrow[u, "z"] \arrow[d, "w"]
& \circ \arrow[l, "y"] \arrow[u, "z"] \arrow[d, "w"]
& \circ \arrow[l, line width=2pt, no head] \arrow[u, line width=2pt, no head] \arrow[dr, line width=2pt, no head]
& \circ \arrow[l, "y"]
\\
\circ \arrow[u, "z"]\arrow[d, "w"]
& \circ \arrow[l, "y"] \arrow[u, "z"] \arrow[d, "w"]
& \circ \arrow[l, line width=2pt, no head] \arrow[u, line width=2pt, no head] \arrow[dr, line width=2pt, no head]
& \circ \arrow[l, "y"] \arrow[ur, line width=2pt, no head] \arrow[dr, "xw"]
& \bullet \arrow[r, "x"] \arrow[l, "y"]
& \circ \arrow[u, "z"] \arrow[d, "w"]
& \circ \arrow[l, "y"] \arrow[u, "z"] \arrow[d, "w"]
\\
\circ \arrow[ur, line width=2pt, no head] \arrow[dr, "xw"]
& \bullet \arrow[r, "x"] \arrow[l, "y"]
& \circ \arrow[u, "z"] \arrow[d, "w"]
& \circ \arrow[l, "y"] \arrow[u, "z"] \arrow[d, "w"]
& \circ \arrow[l, line width=2pt, no head] \arrow[u, line width=2pt, no head] \arrow[dr, line width=2pt, no head]
& \circ \arrow[l, "y"] \arrow[ur, line width=2pt, no head] \arrow[dr, "xw"]
& \bullet \arrow[l, "y"]
\\
\circ \arrow[u, "z"] \arrow[d, "w"]
& \circ \arrow[l, line width=2pt, no head] \arrow[u, line width=2pt, no head] \arrow[dr, line width=2pt, no head]
& \circ \arrow[l, "y"] \arrow[ur, line width=2pt, no head] \arrow[dr, "xw"]
& \bullet \arrow[r, "x"] \arrow[l, "y"]
& \circ \arrow[u, "z"] \arrow[d, "w"]
& \circ \arrow[l, "y"] \arrow[u, "z"] \arrow[d, "w"]
& \circ \arrow[l, line width=2pt, no head] \arrow[u, line width=2pt, no head]
\\
\bullet \arrow[r, "x"]
& \circ \arrow[u, "z"]
& \circ \arrow[l, "y"] \arrow[u, "z"]
& \circ \arrow[l, line width=2pt, no head] \arrow[u, line width=2pt, no head]
& \circ \arrow[l, "y"] \arrow[ur, line width=2pt, no head]
& \bullet \arrow[r, "x"] \arrow[l, "y"]
& \circ \arrow[u, "z"]
\\
\end{tikzcd}
\quad
\begin{tikzcd}[row sep=1.8em, column sep=1.8em]
\circ \arrow[d, line width=2pt, no head]
& \circ \arrow[l, "y"] \arrow[dr, "xw"]
& \circ \arrow[l, "y"] \arrow[dr, "xw"]
& \bullet \arrow[r, "x"] \arrow[l, "y"]
& \circ \arrow[d, "w"]
& \arrow[l, line width=2pt, no head] \arrow[d, line width=2pt, no head]
& \circ \arrow[l, "y"]
\\
\bullet \arrow[r, "x"]
& \circ \arrow[u, "z"] \arrow[d, "w"]
& \arrow[l, line width=2pt, no head] \arrow[u, line width=2pt, no head] \arrow[d, line width=2pt, no head]
& \circ \arrow[l, "y"] \arrow[u, line width=2pt, no head] \arrow[dr, "xw"]
& \circ \arrow[l, "y"] \arrow[ur, "xz"] \arrow[dr, "xw"]
& \bullet \arrow[r, "x"] \arrow[l, "y"]
& \circ \arrow[u, "z"] \arrow[d, "w"]
\\
\circ \arrow[u, line width=2pt, no head] \arrow[dr, "xw"]
& \circ \arrow[l, "y"] \arrow[ur, "xz"] \arrow[dr, "xw"]
& \bullet \arrow[r, "x"] \arrow[l, "y"]
& \circ \arrow[u, "z"] \arrow[d, "w"]
& \circ \arrow[l, line width=2pt, no head] \arrow[u, line width=2pt, no head] \arrow[d, line width=2pt, no head]
& \circ \arrow[l, "y"] \arrow[u, line width=2pt, no head] \arrow[dr, "xw"]
& \circ \arrow[l, "y"]
\\
\circ \arrow[u, "z"]\arrow[d, "w"]
& \circ \arrow[l, line width=2pt, no head] \arrow[u, line width=2pt, no head] \arrow[d, line width=2pt, no head]
& \circ \arrow[l, "y"] \arrow[u, line width=2pt, no head] \arrow[dr, "xw"]
& \circ \arrow[l, "y"] \arrow[ur, "xz"] \arrow[dr, "xw"]
& \bullet \arrow[r, "x"] \arrow[l, "y"]
& \circ \arrow[u, "z"] \arrow[d, "w"]
& \circ \arrow[l, line width=2pt, no head] \arrow[u, line width=2pt, no head] \arrow[d, line width=2pt, no head]
\\
\circ \arrow[ur, "xz"] \arrow[dr, "xw"]
& \bullet \arrow[r, "x"] \arrow[l, "y"]
& \circ \arrow[u, "z"] \arrow[d, "w"]
& \circ \arrow[l, line width=2pt, no head] \arrow[u, line width=2pt, no head] \arrow[d, line width=2pt, no head]
& \circ \arrow[l, "y"] \arrow[u, line width=2pt, no head] \arrow[dr, "xw"]
& \circ \arrow[l, "y"] \arrow[ur, "xz"] \arrow[dr, "xw"]
& \bullet \arrow[l, "y"]
\\
\circ \arrow[u, line width=2pt, no head] \arrow[d, line width=2pt, no head]
& \circ \arrow[l, "y"] \arrow[u, line width=2pt, no head] \arrow[dr, "xw"]
& \circ \arrow[l, "y"] \arrow[ur, "xz"] \arrow[dr, "xw"]
& \bullet \arrow[r, "x"] \arrow[l, "y"]
& \circ \arrow[u, "z"] \arrow[d, "w"]
& \circ \arrow[l, line width=2pt, no head] \arrow[u, line width=2pt, no head] \arrow[d, line width=2pt, no head]
& \circ \arrow[l, "y"] \arrow[u, line width=2pt, no head]
\\
\bullet \arrow[r, "x"]
& \circ \arrow[u, "z"]
& \circ \arrow[l, line width=2pt, no head] \arrow[u, line width=2pt, no head]
& \circ \arrow[l, "y"] \arrow[u, line width=2pt, no head]
& \circ \arrow[l, "y"] \arrow[ur, "xz"]
& \bullet \arrow[r, "x"] \arrow[l, "y"]
& \circ \arrow[u, "z"]
\\
\end{tikzcd}
\]
\[
\begin{tikzcd}[row sep=1.8em, column sep=1.8em]
\circ \arrow[d, "w"]
& \circ \arrow[l, "y"] \arrow[dr, "xw"]
& \circ \arrow[l, "y"] \arrow[dr, "xw"]
& \bullet \arrow[r, line width=2pt, no head] \arrow[l, line width=2pt, no head]
& \circ \arrow[d, "w"]
& \circ \arrow[l, line width=2pt, no head] \arrow[d, "w"]
& \circ \arrow[l, "y"]
\\
\bullet \arrow[r, line width=2pt, no head]
& \circ \arrow[u, "z"] \arrow[d, "w"]
& \circ \arrow[l, line width=2pt, no head] \arrow[u, line width=2pt, no head] \arrow[d, "w"]
& \circ \arrow[l, "y"] \arrow[u, "z"] \arrow[dr, "xw"]
& \circ \arrow[l, "y"] \arrow[ur, "xz"] \arrow[dr, "xw"]
& \bullet \arrow[r, line width=2pt, no head] \arrow[l, line width=2pt, no head]
& \circ \arrow[u, "z"] \arrow[d, "w"]
\\
\circ \arrow[u, "z"] \arrow[dr, "xw"]
& \circ \arrow[l, "y"] \arrow[ur, "xz"] \arrow[dr, "xw"]
& \bullet \arrow[r, line width=2pt, no head] \arrow[l, line width=2pt, no head]
& \circ \arrow[u, "z"] \arrow[d, "w"]
& \circ \arrow[l, line width=2pt, no head] \arrow[u, line width=2pt, no head] \arrow[d, "w"]
& \circ \arrow[l, "y"] \arrow[u, "z"] \arrow[dr, "xw"]
& \circ \arrow[l, "y"]
\\
\circ \arrow[u, "z"]\arrow[d, "w"]
& \circ \arrow[l, line width=2pt, no head] \arrow[u, line width=2pt, no head] \arrow[d, "w"]
& \circ \arrow[l, "y"] \arrow[u, "z"] \arrow[dr, "xw"]
& \circ \arrow[l, "y"] \arrow[ur, "xz"] \arrow[dr, "xw"]
& \bullet \arrow[r, line width=2pt, no head] \arrow[l, line width=2pt, no head]
& \circ \arrow[u, "z"] \arrow[d, "w"]
& \circ \arrow[l, line width=2pt, no head] \arrow[u, line width=2pt, no head] \arrow[d, "w"]
\\
\circ \arrow[ur, "xz"] \arrow[dr, "xw"]
& \bullet \arrow[r, line width=2pt, no head] \arrow[l, line width=2pt, no head]
& \circ \arrow[u, "z"] \arrow[d, "w"]
& \circ \arrow[l, line width=2pt, no head] \arrow[u, line width=2pt, no head] \arrow[d, "w"]
& \circ \arrow[l, "y"] \arrow[u, "z"] \arrow[dr, "xw"]
& \circ \arrow[l, "y"] \arrow[ur, "xz"] \arrow[dr, "xw"]
& \bullet \arrow[l, line width=2pt, no head]
\\
\circ \arrow[u, line width=2pt, no head] \arrow[d, "w"]
& \circ \arrow[l, "y"] \arrow[u, "z"] \arrow[dr, "xw"]
& \circ \arrow[l, "y"] \arrow[ur, "xz"] \arrow[dr, "xw"]
& \bullet \arrow[r, line width=2pt, no head] \arrow[l, line width=2pt, no head]
& \circ \arrow[u, "z"] \arrow[d, "w"]
& \circ \arrow[l, line width=2pt, no head] \arrow[u, line width=2pt, no head] \arrow[d, "w"]
& \circ \arrow[l, "y"] \arrow[u, "z"]
\\
\bullet \arrow[r, line width=2pt, no head]
& \circ \arrow[u, "z"]
& \circ \arrow[l, line width=2pt, no head] \arrow[u, line width=2pt, no head]
& \circ \arrow[l, "y"] \arrow[u, "z"]
& \circ \arrow[l, "y"] \arrow[ur, "xz"]
& \bullet \arrow[r, line width=2pt, no head] \arrow[l, line width=2pt, no head]
& \circ \arrow[u, "z"]
\\
\end{tikzcd}
\quad
\begin{tikzcd}[row sep=1.8em, column sep=1.8em]
\circ \arrow[d, "w"]
& \circ \arrow[l, "y"] \arrow[dr, "xw"]
& \circ \arrow[l, "y"] \arrow[dr, "xw"]
& \bullet \arrow[r, "x"] \arrow[l, line width=2pt, no head]
& \circ \arrow[d, line width=2pt, no head]
& \circ \arrow[l, "y"] \arrow[d, "w"]
& \circ \arrow[l, "y"]
\\
\bullet \arrow[r, "x"]
& \circ \arrow[u, line width=2pt, no head] \arrow[d, line width=2pt, no head]
& \circ \arrow[l, "y"] \arrow[u, line width=2pt, no head] \arrow[d, "w"]
& \circ \arrow[l, "y"] \arrow[u, "z"] \arrow[dr, "xw"]
& \circ \arrow[l, "y"] \arrow[ur, "xz"] \arrow[dr, "xw"]
& \bullet \arrow[r, "x"] \arrow[l, line width=2pt, no head]
& \circ \arrow[u, line width=2pt, no head] \arrow[d, line width=2pt, no head]
\\
\circ \arrow[u, "z"] \arrow[dr, "xw"]
& \circ \arrow[l, "y"] \arrow[ur, "xz"] \arrow[dr, "xw"]
& \bullet \arrow[r, "x"] \arrow[l, line width=2pt, no head]
& \circ \arrow[u, line width=2pt, no head] \arrow[d, line width=2pt, no head]
& \circ \arrow[l, "y"] \arrow[u, line width=2pt, no head] \arrow[d, "w"]
& \circ \arrow[l, "y"] \arrow[u, "z"] \arrow[dr, "xw"]
& \circ \arrow[l, "y"]
\\
\circ \arrow[u, line width=2pt, no head] \arrow[d, line width=2pt, no head]
& \circ \arrow[l, "y"] \arrow[u, line width=2pt, no head] \arrow[d, "w"]
& \circ \arrow[l, "y"] \arrow[u, "z"] \arrow[dr, "xw"]
& \circ \arrow[l, "y"] \arrow[ur, "xz"] \arrow[dr, "xw"]
& \bullet \arrow[r, "x"] \arrow[l, line width=2pt, no head]
& \circ \arrow[u, line width=2pt, no head] \arrow[d, line width=2pt, no head]
& \circ \arrow[l, "y"] \arrow[u, line width=2pt, no head] \arrow[d, "w"]
\\
\circ \arrow[ur, "xz"] \arrow[dr, "xw"]
& \bullet \arrow[r, "x"] \arrow[l, line width=2pt, no head]
& \circ \arrow[u, line width=2pt, no head] \arrow[d, line width=2pt, no head]
& \circ \arrow[l, "y"] \arrow[u, line width=2pt, no head] \arrow[d, "w"]
& \circ \arrow[l, "y"] \arrow[u, "z"] \arrow[dr, "xw"]
& \circ \arrow[l, "y"] \arrow[ur, "xz"] \arrow[dr, "xw"]
& \bullet \arrow[l, line width=2pt, no head]
\\
\circ \arrow[u, line width=2pt, no head] \arrow[d, "w"]
& \circ \arrow[l, "y"] \arrow[u, "z"] \arrow[dr, "xw"]
& \circ \arrow[l, "y"] \arrow[ur, "xz"] \arrow[dr, "xw"]
& \bullet \arrow[r, "x"] \arrow[l, line width=2pt, no head]
& \circ \arrow[u, line width=2pt, no head] \arrow[d, line width=2pt, no head]
& \circ \arrow[l, "y"] \arrow[u, line width=2pt, no head] \arrow[d, "w"]
& \circ \arrow[l, "y"] \arrow[u, "z"]
\\
\bullet \arrow[r, "x"]
& \circ \arrow[u, "z"]
& \circ \arrow[l, "y"] \arrow[u, line width=2pt, no head]
& \circ \arrow[l, "y"] \arrow[u, line width=2pt, no head]
& \circ \arrow[l, "y"] \arrow[ur, "xz"]
& \bullet \arrow[r, "x"] \arrow[l, line width=2pt, no head]
& \circ \arrow[u, "z"]
\\
\end{tikzcd}
\]

Therefore there are the following five kinds of $2$-tilting bundles up to translations. Their endomorphism algebras are $2$-representation infinite algebras of type $\widetilde A\widetilde A$. Observe that by mutations of non-trivial upper sets in $q^{-1}(J)$, they are mutated to each other, which correspond to $2$-APR tilting mutations.
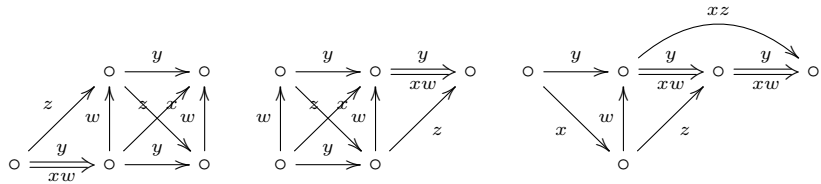
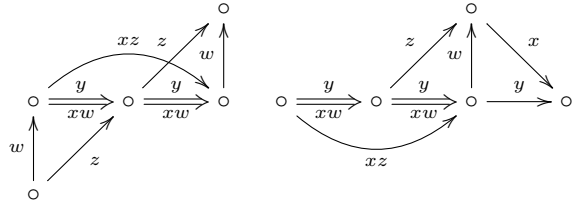
\[\begin{array}{c c c}
\xymatrix{
 & \circ \ar[r]^y \ar[dr]^x & \circ \\
\circ \ar@2[r]^y_{xw} \ar[ur]^z & \circ \ar[r]^y \ar[ur]^z \ar[u]^w & \circ \ar[u]^w
}&\xymatrix{
\circ \ar[r]^y \ar[dr]^x & \circ \ar@2[r]^y_{xw} & \circ \\
\circ \ar[r]^y \ar[ur]^z \ar[u]^w & \circ \ar[ur]_z \ar[u]^w
}&\xymatrix{
\circ \ar[r]^y \ar[dr]_x & \circ \ar@2[r]^y_{xw} \ar@/^18pt/[rr]^{xz} & \circ \ar@2[r]^y_{xw} & \circ \\
 & \circ \ar[ur]_z \ar[u]^w
}\end{array}\]
\[\begin{array}{c c}
\xymatrix{
 & & \circ \\
\circ \ar@2[r]^y_{xw} \ar@/^18pt/[rr]^{xz} & \circ \ar@2[r]^y_{xw} \ar[ur]^z & \circ \ar[u]^w \\
\circ \ar[ur]_z \ar[u]^w
}&\xymatrix{
 & & \circ \ar[dr]^x \\
\circ \ar@2[r]^y_{xw} \ar@/^-18pt/[rr]_{xz} & \circ \ar@2[r]^y_{xw} \ar[ur]^z & \circ \ar[r]^y \ar[u]^w & \circ
}\end{array}\]
\end{Ex}

\subsection{Examples of dimension $3$}

In this subsection, we exhibit the classification of $3$-tilting bundles consisting of line bundles on the $\mathbb{P}^1$-bundle $\mathbb{P}_{\mathbb{P}^2}(\O_{\mathbb{P}^2}\oplus\O_{\mathbb{P}^2}(-2))$ over $\mathbb{P}^2$ as an application of Theorem \ref{classfitiltrk2}.

\begin{Ex}($d=3$)\label{Ex3fold}
Put $G:=\mathbb{Z}^2$ and $\vec{x}=\vec{y}=\vec{z}=(1,0),\vec{u}=(-1,1),\vec{v}=(1,1)\in G$. We view $S:=k[x,y,z,u,v]$ as a $G$-graded $k$-algebra. Then the resulting toric stack $\X$ is isomorphic to the $\mathbb{P}^1$-bundle $\mathbb{P}_{\mathbb{P}^2}(\O_{\mathbb{P}^2}\oplus\O_{\mathbb{P}^2}(-2))$ over $\mathbb{P}^2$. We have $\vec{p}=(3,2)$ and $H=G/\mathbb{Z}\vec{p}\cong\mathbb{Z}; (a,b)+\mathbb{Z}\vec{p}\mapsto 2a-3b$. If we equip $H$ with our partial order, then the quiver of $H$ becomes the following.
\[\xymatrix{
\cdots \ar@3@/^15pt/[rr]^{x,y,z} \ar@/^-15pt/[rrrrr]_{-u} \ar[r]_{-v} & \circ \ar@3@/^15pt/[rr]^{x,y,z} \ar@/^-15pt/[rrrrr]_{-u} \ar[r]_{-v} & \circ \ar@3@/^15pt/[rr]^{x,y,z} \ar@/^-15pt/[rrrrr]_{-u} \ar[r]_{-v} & \circ \ar@3@/^15pt/[rr]^{x,y,z} \ar[r]_{-v} & \circ \ar@3@/^15pt/[rr]^{x,y,z} \ar[r]_{-v} & \circ \ar@3@/^15pt/[rr]^{x,y,z} \ar[r]_{-v} & \circ \ar[r]_{-v} & \cdots
}\]
Then there is the following just one kind of non-trivial upper sets in $H$ up to translations.
\[\xymatrix{
\circ \ar@3@/^15pt/[rr]^{x,y,z} \ar@/^-15pt/[rrrrr]_{-u} \ar[r]_{-v} & \circ \ar@3@/^15pt/[rr]^{x,y,z} \ar@/^-15pt/[rrrrr]_{-u} \ar[r]_{-v} & \circ \ar@3@/^15pt/[rr]^{x,y,z} \ar[r]_{-v} & \circ \ar@3@/^15pt/[rr]^{x,y,z} \ar[r]_{-v} & \circ \ar@3@/^15pt/[rr]^{x,y,z} \ar[r]_{-v} & \circ \ar[r]_{-v} & \cdots
}\]
Remark that the isomorphism $H\cong\mathbb{Z}$ sends $s$ to $6$. If we let $\widehat{C}$ be the cut of $\widehat{Q}$ corresponding to an upper set in $H$, then $\widehat{Q}(\widehat{C})$ is as follows. Here, the black vertices correspond to elements of $B$.

\begin{center}
\begin{tikzpicture}[
  x={(2.4cm,0cm)},
  y={(-0.72cm,0.96cm)},
  z={(0cm,2.4cm)},
  >={Stealth[length=2.2mm,width=1.8mm]},
  line cap=round,
  line join=round,
  edge/.style={->, line width=0.95pt},
  rededge/.style={->, red!85!black, line width=1.35pt},
  vertex/.style={circle, fill=black, inner sep=1.25pt},
  every node/.style={font=\scriptsize}
]

\foreach \i in {0,...,4}{
  \foreach \j in {0,...,2}{
    \foreach \k in {0,...,6}{
      \coordinate (\i-\j-\k) at (\i,\j,\k);
    }
  }
}

\foreach \i in {0,...,4}{
  \foreach \j in {2}{
    \foreach \k in {0,...,6}{
      \pgfmathtruncatemacro{\labf}{mod(2*\i+2*\j+\k,6)}
      \ifnum\labf=0
        \ifnum\i<4
          \draw[->, thick] ($(\i-\j-\k)+(0.04,0,0)$) -- ($(\i-\j-\k)+(0.96,0,0)$);
        \fi
        \ifnum\i>0
          \draw[->, thick] ($(\i-\j-\k)+(-0.04,-0.04,0)$) -- ($(\i-\j-\k)+(-0.96,-0.96,0)$);
        \fi
      \fi
      \ifnum\labf=1
        \ifnum\i<4
          \draw[->, thick] ($(\i-\j-\k)+(0.04,0,0)$) -- ($(\i-\j-\k)+(0.96,0,0)$);
        \fi
        \ifnum\i>0
          \draw[->, thick] ($(\i-\j-\k)+(-0.04,-0.04,0)$) -- ($(\i-\j-\k)+(-0.96,-0.96,0)$);
        \fi
        \ifnum\k>0
          \draw[->, thick] ($(\i-\j-\k)+(0,0,-0.04)$) -- ($(\i-\j-\k)+(0,0,-0.96)$);
        \fi
      \fi
      \ifnum\labf=2
        \ifnum\i<4
          \draw[->, thick] ($(\i-\j-\k)+(0.04,0,0)$) -- ($(\i-\j-\k)+(0.96,0,0)$);
        \fi
        \ifnum\i>0
          \draw[->, thick] ($(\i-\j-\k)+(-0.04,-0.04,0)$) -- ($(\i-\j-\k)+(-0.96,-0.96,0)$);
        \fi
        \ifnum\k>0
          \draw[->, thick] ($(\i-\j-\k)+(0,0,-0.04)$) -- ($(\i-\j-\k)+(0,0,-0.96)$);
        \fi
      \fi
      \ifnum\labf=3
        \ifnum\i<4
          \draw[->, thick] ($(\i-\j-\k)+(0.04,0,0)$) -- ($(\i-\j-\k)+(0.96,0,0)$);
        \fi
        \ifnum\i>0
          \draw[->, thick] ($(\i-\j-\k)+(-0.04,-0.04,0)$) -- ($(\i-\j-\k)+(-0.96,-0.96,0)$);
        \fi
        \ifnum\k>0
          \draw[->, thick] ($(\i-\j-\k)+(0,0,-0.04)$) -- ($(\i-\j-\k)+(0,0,-0.96)$);
        \fi
      \fi
      \ifnum\labf=4
        \ifnum\k>0
          \draw[->, thick] ($(\i-\j-\k)+(0,0,-0.04)$) -- ($(\i-\j-\k)+(0,0,-0.96)$);
        \fi
        \ifnum\i>0
        \ifnum\k<6
          \draw[->, thick] ($(\i-\j-\k)+(-0.04,-0.04,0.04)$) -- ($(\i-\j-\k)+(-0.96,-0.96,0.96)$);
        \fi
        \fi
        \ifnum\i<4
        \ifnum\k<6
          \draw[->, thick] ($(\i-\j-\k)+(0.04,0,0.04)$) -- ($(\i-\j-\k)+(0.96,0,0.96)$);
        \fi
        \fi
      \fi
      \ifnum\labf=5
        \ifnum\k>0
          \draw[->, thick] ($(\i-\j-\k)+(0,0,-0.04)$) -- ($(\i-\j-\k)+(0,0,-0.96)$);
        \fi
        \ifnum\k<6
          \draw[->, thick] ($(\i-\j-\k)+(0,0,0.04)$) -- ($(\i-\j-\k)+(0,0,0.96)$);
        \fi
      \fi
    }
  }
}

\foreach \k in {0,...,6}{
  \foreach \j in {1}{
    \foreach \i in {0,...,4}{
      \pgfmathtruncatemacro{\labf}{mod(2*\i+2*\j+\k,6)}
      \ifnum\labf=0
        \ifnum\i<4
          \draw[white, line width=4pt] ($(\i-\j-\k)+(0.04,0,0)$) -- ($(\i-\j-\k)+(0.96,0,0)$);
          \draw[->, thick] ($(\i-\j-\k)+(0.04,0,0)$) -- ($(\i-\j-\k)+(0.96,0,0)$);
        \fi
        \ifnum\i>0
          \draw[white, line width=4pt] ($(\i-\j-\k)+(-0.4,-0.4,0)$) -- ($(\i-\j-\k)+(-0.96,-0.96,0)$);
          \draw[->, thick] ($(\i-\j-\k)+(-0.04,-0.04,0)$) -- ($(\i-\j-\k)+(-0.96,-0.96,0)$);
        \fi
        \draw[->, thick] ($(\i-\j-\k)+(0,0.04,0)$) -- ($(\i-\j-\k)+(0,0.96,0)$);
      \fi
      \ifnum\labf=1
        \ifnum\i<4
          \draw[white, line width=4pt] ($(\i-\j-\k)+(0.4,0,0)$) -- ($(\i-\j-\k)+(0.96,0,0)$);
          \draw[->, thick] ($(\i-\j-\k)+(0.04,0,0)$) -- ($(\i-\j-\k)+(0.96,0,0)$);
        \fi
        \ifnum\i>0
          \draw[white, line width=4pt] ($(\i-\j-\k)+(-0.4,-0.4,0)$) -- ($(\i-\j-\k)+(-0.96,-0.96,0)$);
          \draw[->, thick] ($(\i-\j-\k)+(-0.04,-0.04,0)$) -- ($(\i-\j-\k)+(-0.96,-0.96,0)$);
        \fi
        \ifnum\k>0
          \draw[white, line width=4pt] ($(\i-\j-\k)+(0,0,-0.4)$) -- ($(\i-\j-\k)+(0,0,-0.8)$);
          \draw[->, thick] ($(\i-\j-\k)+(0,0,-0.04)$) -- ($(\i-\j-\k)+(0,0,-0.96)$);
        \fi
        \draw[->, thick] ($(\i-\j-\k)+(0,0.04,0)$) -- ($(\i-\j-\k)+(0,0.96,0)$);
      \fi
      \ifnum\labf=2
        \ifnum\i<4
          \draw[white, line width=4pt] ($(\i-\j-\k)+(0.4,0,0)$) -- ($(\i-\j-\k)+(0.96,0,0)$);
          \draw[->, thick] ($(\i-\j-\k)+(0.04,0,0)$) -- ($(\i-\j-\k)+(0.96,0,0)$);
        \fi
        \ifnum\i>0
          \draw[white, line width=4pt] ($(\i-\j-\k)+(-0.4,-0.4,0)$) -- ($(\i-\j-\k)+(-0.96,-0.96,0)$);
          \draw[->, thick] ($(\i-\j-\k)+(-0.04,-0.04,0)$) -- ($(\i-\j-\k)+(-0.96,-0.96,0)$);
        \fi
        \ifnum\k>0
          \draw[white, line width=4pt] ($(\i-\j-\k)+(0,0,-0.4)$) -- ($(\i-\j-\k)+(0,0,-0.8)$);
          \draw[->, thick] ($(\i-\j-\k)+(0,0,-0.04)$) -- ($(\i-\j-\k)+(0,0,-0.96)$);
        \fi
        \draw[->, thick] ($(\i-\j-\k)+(0,0.04,0)$) -- ($(\i-\j-\k)+(0,0.96,0)$);
      \fi
      \ifnum\labf=3
        \ifnum\i<4
          \draw[white, line width=4pt] ($(\i-\j-\k)+(0.2,0,0)$) -- ($(\i-\j-\k)+(0.96,0,0)$);
          \draw[->, thick] ($(\i-\j-\k)+(0.04,0,0)$) -- ($(\i-\j-\k)+(0.96,0,0)$);
        \fi
        \ifnum\i>0
          \draw[white, line width=4pt] ($(\i-\j-\k)+(-0.04,-0.04,0)$) -- ($(\i-\j-\k)+(-0.96,-0.96,0)$);
          \draw[->, thick] ($(\i-\j-\k)+(-0.04,-0.04,0)$) -- ($(\i-\j-\k)+(-0.96,-0.96,0)$);
        \fi
        \ifnum\k>0
          \draw[white, line width=4pt] ($(\i-\j-\k)+(0,0,-0.2)$) -- ($(\i-\j-\k)+(0,0,-0.8)$);
          \draw[->, thick] ($(\i-\j-\k)+(0,0,-0.04)$) -- ($(\i-\j-\k)+(0,0,-0.96)$);
        \fi
        \draw[->, thick] ($(\i-\j-\k)+(0,0.04,0)$) -- ($(\i-\j-\k)+(0,0.96,0)$);
      \fi
      \ifnum\labf=4
        \ifnum\k>0
          \draw[->, thick] ($(\i-\j-\k)+(0,0,-0.04)$) -- ($(\i-\j-\k)+(0,0,-0.96)$);
        \fi
        \ifnum\i>0
        \ifnum\k<6
          \draw[white, line width=4pt] ($(\i-\j-\k)+(-0.2,-0.2,0.2)$) -- ($(\i-\j-\k)+(-0.96,-0.96,0.96)$);
          \draw[->, thick] ($(\i-\j-\k)+(-0.04,-0.04,0.04)$) -- ($(\i-\j-\k)+(-0.96,-0.96,0.96)$);
        \fi
        \fi
        \ifnum\i<4
        \ifnum\k<6
          \draw[white, line width=4pt] ($(\i-\j-\k)+(0.3,0,0.3)$) -- ($(\i-\j-\k)+(0.8,0,0.8)$);
          \draw[->, thick] ($(\i-\j-\k)+(0.04,0,0.04)$) -- ($(\i-\j-\k)+(0.96,0,0.96)$);
        \fi
        \fi
        \ifnum\k<6
          \draw[white, line width=4pt] ($(\i-\j-\k)+(0,0.2,0.2)$) -- ($(\i-\j-\k)+(0,0.8,0.8)$);
          \draw[->, thick] ($(\i-\j-\k)+(0,0.04,0.04)$) -- ($(\i-\j-\k)+(0,0.96,0.96)$);
        \fi
      \fi
      \ifnum\labf=5
        \ifnum\k>0
          \draw[white, line width=4pt] ($(\i-\j-\k)+(0,0,-0.04)$) -- ($(\i-\j-\k)+(0,0,-0.75)$);
          \draw[->, thick] ($(\i-\j-\k)+(0,0,-0.04)$) -- ($(\i-\j-\k)+(0,0,-0.96)$);
        \fi
        \ifnum\k<6
          \draw[white, line width=4pt] ($(\i-\j-\k)+(0,0,0.3)$) -- ($(\i-\j-\k)+(0,0,0.96)$);
          \draw[->, thick] ($(\i-\j-\k)+(0,0,0.04)$) -- ($(\i-\j-\k)+(0,0,0.96)$);
        \fi
      \fi
    }
  }
}

\foreach \k in {0,...,6}{
  \foreach \j in {0}{
    \foreach \i in {0,...,4}{
      \pgfmathtruncatemacro{\labf}{mod(2*\i+2*\j+\k,6)}
      \ifnum\labf=0
        \ifnum\i<4
          \draw[white, line width=4pt] ($(\i-\j-\k)+(0.04,0,0)$) -- ($(\i-\j-\k)+(0.96,0,0)$);
          \draw[->, thick] ($(\i-\j-\k)+(0.04,0,0)$) -- ($(\i-\j-\k)+(0.96,0,0)$);
        \fi
        \draw[->, thick] ($(\i-\j-\k)+(0,0.04,0)$) -- ($(\i-\j-\k)+(0,0.96,0)$);
      \fi
      \ifnum\labf=1
        \ifnum\i<4
          \draw[white, line width=4pt] ($(\i-\j-\k)+(0.4,0,0)$) -- ($(\i-\j-\k)+(0.96,0,0)$);
          \draw[->, thick] ($(\i-\j-\k)+(0.04,0,0)$) -- ($(\i-\j-\k)+(0.96,0,0)$);
        \fi
        \ifnum\k>0
          \draw[white, line width=4pt] ($(\i-\j-\k)+(0,0,-0.4)$) -- ($(\i-\j-\k)+(0,0,-0.8)$);
          \draw[->, thick] ($(\i-\j-\k)+(0,0,-0.04)$) -- ($(\i-\j-\k)+(0,0,-0.96)$);
        \fi
        \draw[white, line width=4pt] ($(\i-\j-\k)+(0,0.4,0)$) -- ($(\i-\j-\k)+(0,0.85,0)$);
        \draw[->, thick] ($(\i-\j-\k)+(0,0.04,0)$) -- ($(\i-\j-\k)+(0,0.96,0)$);
      \fi
      \ifnum\labf=2
        \ifnum\i<4
          \draw[white, line width=4pt] ($(\i-\j-\k)+(0.4,0,0)$) -- ($(\i-\j-\k)+(0.96,0,0)$);
          \draw[->, thick] ($(\i-\j-\k)+(0.04,0,0)$) -- ($(\i-\j-\k)+(0.96,0,0)$);
        \fi
        \ifnum\k>0
          \draw[white, line width=4pt] ($(\i-\j-\k)+(0,0,-0.4)$) -- ($(\i-\j-\k)+(0,0,-0.8)$);
          \draw[->, thick] ($(\i-\j-\k)+(0,0,-0.04)$) -- ($(\i-\j-\k)+(0,0,-0.96)$);
        \fi
        \draw[->, thick] ($(\i-\j-\k)+(0,0.04,0)$) -- ($(\i-\j-\k)+(0,0.96,0)$);
      \fi
      \ifnum\labf=3
        \ifnum\i<4
          \draw[white, line width=4pt] ($(\i-\j-\k)+(0.2,0,0)$) -- ($(\i-\j-\k)+(0.96,0,0)$);
          \draw[->, thick] ($(\i-\j-\k)+(0.04,0,0)$) -- ($(\i-\j-\k)+(0.96,0,0)$);
        \fi
        \ifnum\k>0
          \draw[white, line width=4pt] ($(\i-\j-\k)+(0,0,-0.2)$) -- ($(\i-\j-\k)+(0,0,-0.8)$);
          \draw[->, thick] ($(\i-\j-\k)+(0,0,-0.04)$) -- ($(\i-\j-\k)+(0,0,-0.96)$);
        \fi
        \draw[->, thick] ($(\i-\j-\k)+(0,0.04,0)$) -- ($(\i-\j-\k)+(0,0.96,0)$);
      \fi
      \ifnum\labf=4
        \ifnum\k>0
          \draw[white, line width=4pt] ($(\i-\j-\k)+(0,0,-0.1)$) -- ($(\i-\j-\k)+(0,0,-0.9)$);
          \draw[->, thick] ($(\i-\j-\k)+(0,0,-0.04)$) -- ($(\i-\j-\k)+(0,0,-0.96)$);
        \fi
        \ifnum\i<4
        \ifnum\k<6
          \draw[white, line width=4pt] ($(\i-\j-\k)+(0.3,0,0.3)$) -- ($(\i-\j-\k)+(0.8,0,0.8)$);
          \draw[->, thick] ($(\i-\j-\k)+(0.04,0,0.04)$) -- ($(\i-\j-\k)+(0.96,0,0.96)$);
        \fi
        \fi
        \ifnum\k<6
          \draw[white, line width=4pt] ($(\i-\j-\k)+(0,0.2,0.2)$) -- ($(\i-\j-\k)+(0,0.65,0.65)$);
          \draw[->, thick] ($(\i-\j-\k)+(0,0.04,0.04)$) -- ($(\i-\j-\k)+(0,0.96,0.96)$);
        \fi
      \fi
      \ifnum\labf=5
        \ifnum\k>0
          \draw[white, line width=4pt] ($(\i-\j-\k)+(0,0,-0.04)$) -- ($(\i-\j-\k)+(0,0,-0.75)$);
          \draw[->, thick] ($(\i-\j-\k)+(0,0,-0.04)$) -- ($(\i-\j-\k)+(0,0,-0.96)$);
        \fi
        \ifnum\k<6
          \draw[white, line width=4pt] ($(\i-\j-\k)+(0,0,0.1)$) -- ($(\i-\j-\k)+(0,0,0.96)$);
          \draw[->, thick] ($(\i-\j-\k)+(0,0,0.04)$) -- ($(\i-\j-\k)+(0,0,0.96)$);
        \fi
      \fi
    }
  }
}

\foreach \i in {0,...,4}{
  \foreach \j in {0,...,2}{
    \foreach \k in {0,...,6}{
    \pgfmathtruncatemacro{\labf}{mod(2*\i+2*\j+\k,6)}
      \ifnum\labf=0
        \node at (\i-\j-\k) {$\bullet$};
      \else
        \node at (\i-\j-\k) {$\circ$};
      \fi
    }
  }
}

\end{tikzpicture}
\end{center}

Here, each arrow of each direction represents the following monomials.

\begin{center}
\begin{tikzpicture}[
  x={(2.4cm,0cm)},
  y={(-0.72cm,0.96cm)},
  z={(0cm,2.4cm)},
  >={Stealth[length=2.2mm,width=1.8mm]},
  line cap=round,
  line join=round,
  edge/.style={->, line width=0.95pt},
  rededge/.style={->, red!85!black, line width=1.35pt},
  vertex/.style={circle, fill=black, inner sep=1.25pt},
  every node/.style={font=\scriptsize}
]

\node at (0,0,0) {$\circ$};

\draw[->, thick] (0.04,0,0) -- node[midway,above] {$x$} (0.96,0,0);
\draw[->, thick] (0,0.04,0) -- node[midway,above] {$y$} (0,0.96,0);
\draw[->, thick] (-0.04,-0.04,0) -- node[midway,above] {$z$} (-0.96,-0.96,0);
\draw[->, thick] (0,0,0.04) -- node[midway,right] {$u$} (0,0,0.96);
\draw[->, thick] (0,0,-0.04) -- node[midway,right] {$v$} (0,0,-0.96);
\draw[->, thick] (0.04,0,0.04) -- node[midway,left] {$xu$} (0.96,0,0.96);
\draw[->, thick] (0,0.04,0.04) -- node[midway,left] {$yu$} (0,0.96,0.96);
\draw[->, thick] (-0.04,-0.04,0.04) -- node[midway,left] {$zu$} (-0.96,-0.96,0.96);
\end{tikzpicture}
\end{center}
Thus there is the following just one kind of non-trivial upper sets in $\J_H$ up to translations with the following quiver $Q(C)$.
\[\xymatrix{
\circ \ar@3@/^15pt/[rr]^{x,y,z} & \circ \ar@3@/^15pt/[rr]^{x,y,z} \ar[l]^v & \circ \ar@3@/^15pt/[rr]^{x,y,z} \ar[l]^v & \circ \ar@3@/^15pt/[rr]^{x,y,z} \ar[l]^v & \circ \ar[l]^v \ar@3@/^10pt/[lll]^{xu,yu,zu} & \circ \ar@/^20pt/[lllll]^u \ar[l]^v
}\]
The quiver of $q^{-1}(J)$ is as follows.
\[\xymatrix{
 & & & & & & \circ \ar@3[r]^{x,y,z} & \circ \ar@3[r]_{x,y,z} & \cdots \\
 & & & & \circ \ar@3[r]^{x,y,z} & \circ \ar@3[r]_{x,y,z} \ar[ur]^v & \circ \ar@3[u]^{xu,yu,zu} \ar[ur]_v \\
 & & & \circ \ar@3[r]_{x,y,z} \ar[ur]^v & \circ \ar@3[r]_{x,y,z} \ar[ur]^v & \circ \ar[ul]^u \ar[ur]_v \\
 & \circ \ar@3[r]^{x,y,z} & \circ \ar@3[r]_{x,y,z} \ar[ur]^v & \circ \ar@3[u]^{xu,yu,zu} \ar[ur]_v \\
\cdots \ar@3[r]_{x,y,z} \ar[ur]^v & \circ \ar@3[r]_{x,y,z} \ar[ur]^v & \circ \ar[ul]^u \ar[ur]_v
}\]
There are the following six kinds of non-trivial upper sets in $q^{-1}(J)$ up to translations by $\vec{p}$.
\[\adjustbox{max width=\linewidth}{$\displaystyle \xymatrix{
 & & & & & \circ \ar@3[r]^{x,y,z} & \circ \ar@3[r]_{x,y,z} & \cdots \\
 & & & \circ \ar@3[r]^{x,y,z} & \circ \ar@3[r]_{x,y,z} \ar[ur]^v & \circ \ar@3[u]^{xu,yu,zu} \ar[ur]_v \\
 & & \circ \ar@3[r]_{x,y,z} \ar[ur]^v & \circ \ar@3[r]_{x,y,z} \ar[ur]^v & \circ \ar[ul]^u \ar[ur]_v \\
\circ \ar@3[r]^{x,y,z} & \circ \ar@3[r]_{x,y,z} \ar[ur]^v & \circ \ar@3[u]^{xu,yu,zu} \ar[ur]_v
}\xymatrix{
 & & & & \circ \ar@3[r]^{x,y,z} & \circ \ar@3[r]_{x,y,z} & \cdots \\
 & & \circ \ar@3[r]^{x,y,z} & \circ \ar@3[r]_{x,y,z} \ar[ur]^v & \circ \ar@3[u]^{xu,yu,zu} \ar[ur]_v \\
 & \circ \ar@3[r]_{x,y,z} \ar[ur]^v & \circ \ar@3[r]_{x,y,z} \ar[ur]^v & \circ \ar[ul]^u \ar[ur]_v \\
\circ \ar@3[r]_{x,y,z} \ar[ur]^v & \circ \ar@3[u]^{xu,yu,zu} \ar[ur]_v
}$}\]
\[\xymatrix{
 & & & \circ \ar@3[r]^{x,y,z} & \circ \ar@3[r]_{x,y,z} & \cdots \\
 & \circ \ar@3[r]^{x,y,z} & \circ \ar@3[r]_{x,y,z} \ar[ur]^v & \circ \ar@3[u]^{xu,yu,zu} \ar[ur]_v \\
\circ \ar@3[r]_{x,y,z} \ar[ur]^v & \circ \ar@3[r]_{x,y,z} \ar[ur]^v & \circ \ar[ul]^u \ar[ur]_v \\
\circ \ar@3[u]^{xu,yu,zu} \ar[ur]_v
}\xymatrix{
 & & & \circ \ar@3[r]^{x,y,z} & \circ \ar@3[r]_{x,y,z} & \cdots \\
 & \circ \ar@3[r]^{x,y,z} & \circ \ar@3[r]_{x,y,z} \ar[ur]^v & \circ \ar@3[u]^{xu,yu,zu} \ar[ur]_v \\
\circ \ar@3[r]_{x,y,z} \ar[ur]^v & \circ \ar@3[r]_{x,y,z} \ar[ur]^v & \circ \ar[ul]^u \ar[ur]_v
}\]
\[\xymatrix{
 & & \circ \ar@3[r]^{x,y,z} & \circ \ar@3[r]_{x,y,z} & \cdots \\
\circ \ar@3[r]^{x,y,z} & \circ \ar@3[r]_{x,y,z} \ar[ur]^v & \circ \ar@3[u]^{xu,yu,zu} \ar[ur]_v \\
\circ \ar@3[r]_{x,y,z} \ar[ur]^v & \circ \ar[ul]^u \ar[ur]_v
}\xymatrix{
 & & \circ \ar@3[r]^{x,y,z} & \circ \ar@3[r]_{x,y,z} & \cdots \\
\circ \ar@3[r]^{x,y,z} & \circ \ar@3[r]_{x,y,z} \ar[ur]^v & \circ \ar@3[u]^{xu,yu,zu} \ar[ur]_v \\
 & \circ \ar[ul]^u \ar[ur]_v
}\]
These upper sets correspond to the following cuts of $\widehat Q(\widehat C)$.



Therefore there are the following six kinds of $3$-tilting bundles up to translations. Their endomorphism algebras are $3$-representation infinite algebras of type $\widetilde A\widetilde A$. Observe that by mutations of non-trivial upper sets in $q^{-1}(J)$, they are mutated to each other, which correspond to $3$-APR tilting mutations.
\[\begin{array}{c c c}
\xymatrix{
 & & \circ \ar@3[r]_{x,y,z} & \circ \ar@3[r]_{x,y,z} & \circ \\
\circ \ar@3[r]^{x,y,z} & \circ \ar@3[r]_{x,y,z} \ar[ur]^v & \circ \ar@3[u]^{xu,yu,zu} \ar[ur]_v
}&\xymatrix{
 & & \circ \\
 & \circ \ar@3[r]_{x,y,z} \ar[ur]^v & \circ \ar@3[r]_{x,y,z} & \circ \ar[ul]_u \\
\circ \ar@3[r]_{x,y,z} \ar[ur]^v & \circ \ar@3[u]^{xu,yu,zu} \ar[ur]_v
}&\xymatrix{
 & \circ \ar@3[r]^{x,y,z} & \circ \\
\circ \ar@3[r]_{x,y,z} \ar[ur]^v & \circ \ar@3[r]_{x,y,z} \ar[ur]^v & \circ \ar[ul]^u \\
\circ \ar@3[u]^{xu,yu,zu} \ar[ur]_v
}\end{array}\]
\[\begin{array}{c c c}
\xymatrix{
 & \circ \ar@3[r]^{x,y,z} & \circ \ar@3[r]^{x,y,z} & \circ \\
\circ \ar@3[r]_{x,y,z} \ar[ur]^v & \circ \ar@3[r]_{x,y,z} \ar[ur]^v & \circ \ar[ul]^u \ar[ur]_v
}&\xymatrix{
 & & \circ \\
\circ \ar@3[r]^{x,y,z} & \circ \ar@3[r]_{x,y,z} \ar[ur]^v & \circ \ar@3[u]^{xu,yu,zu} \\
\circ \ar@3[r]_{x,y,z} \ar[ur]^v & \circ \ar[ul]^u \ar[ur]_v
}&\xymatrix{
 & & \circ \ar@3[r]^{x,y,z} & \circ \\
\circ \ar@3[r]^{x,y,z} & \circ \ar@3[r]_{x,y,z} \ar[ur]^v & \circ \ar@3[u]^{xu,yu,zu} \ar[ur]_v \\
 & \circ \ar[ul]^u \ar[ur]_v
}\end{array}\]
\end{Ex}

\begin{appendix}

\section{Beilinson-type theorem for \texorpdfstring{$G$}{G}-graded dg rings}\label{appenBei}

We prove a Beilinson-type theorem for graded dg rings with a generalized
Gorenstein parameter and use it to extend the Minamoto--Mori correspondence.

Let $G$ be a finitely generated abelian group of rank one with a
translation-invariant partial order such that $G=\mathbb ZG_{\ge0}$.
Choose the quotient map $\pi\colon G\to G/G_{\rm tors}\cong\mathbb Z$
positive on $G_{>0}$. Such a choice is possible since a nonzero element
of $G_{\ge0}$ cannot be torsion and a positive combination of elements
with opposite numerical signs would be torsion. For $u>0$, the image of
$G_{\ge0}$ in the finite group $G/\mathbb Zu$ is a generating submonoid
and hence the whole group. Choosing a nonnegative representative in
each class shows that every element of sufficiently large numerical
degree lies in $G_{\ge0}$. Consequently every non-trivial upper set
contains all sufficiently large degrees and excludes all sufficiently
small ones.

Let $\Gamma$ be a $G_{\ge0}$-graded dg ring. For an upper set $I$ and
a graded dg module $X$, write $X_I$ for its submodule in degrees in
$I$, and $X_{I^c}=X/X_I$.

\begin{Def}\label{subGder}
Let $I\subseteq G$ be a subset.
\begin{enumerate}
\item For $X\in\per^G\!\Gamma$, define $\thick^IX:=\thick\{X(-g)\mid g\in I\}\subseteq\per^G\!\Gamma$.
\item $\per^I\Gamma:=\thick^I\Gamma\subseteq\per^G\!\Gamma$
\item $\D^G(\Gamma)_I:=\{X\in\D^G(\Gamma)\mid H^nX_{g}=0\text{ for all }n\in\mathbb{Z}\text{ and }g\in I^c\}$
\item $(\per^G\!\Gamma)_I:=\per^G\!\Gamma\cap\D^G(\Gamma)_I$
\end{enumerate}
\end{Def}

\subsection{Beilinson-type theorem}

Beilinson's celebrated result \cite{Bei} shows that $\bigoplus_{i=0}^d\O_{\mathbb{P}^d}(i)\in\Coh\mathbb{P}^d$ is a tilting bundle. This is significantly generalized to the setting of non-commutative projective schemes of $\mathbb{Z}_{\geq0}$-graded dg rings (or dg categories) in \cite[A.7]{Han24c}. In this section, we give a further generalization to the setting of $G_{\geq0}$-graded dg rings.

As in \cite[A.1]{Han24c}, we say $\Gamma$ has {\it generalized Gorenstein parameter} $p\in G$ if $\RHom_\Gamma(-,\Gamma)$ takes $\D^G(\Gamma)_0$ to $\D^G(\Gamma^{\op})_{-p}$ and $\RHom_{\Gamma^{\op}}(-,\Gamma)$ takes $\D^G(\Gamma^{\op})_0$ to $\D^G(\Gamma)_{-p}$. Then we have an analogous statement to \cite[A.4]{Han24c}.

\begin{Lem}\label{degsuppzero}
Let $\Gamma$ be a $G_{\geq0}$-graded dg ring.
\begin{enumerate}
\item The natural functor $\D(\Gamma_0)\to \D^G(\Gamma)_0$ is a triangle equivalence. In particular, we have $\D^G(\Gamma)_0=\Loc\Gamma_0$.
\item $\Gamma$ has generalized Gorenstein parameter $p$ if and only if
\[
 \RHom_\Gamma(\Gamma_0,\Gamma)\in\D^G(\Gamma^{\op})_{-p},
 \qquad
 \RHom_{\Gamma^{\op}}(\Gamma_0,\Gamma)\in\D^G(\Gamma)_{-p}.
\]
\end{enumerate}
\end{Lem}
\begin{proof}
For $X\in\D^G(\Gamma)_0$, the inclusion $X_{\ge0}\to X$ and the
projection $X_{\ge0}\to X_0$ are quasi-isomorphisms. Thus taking the
degree-zero component is a quasi-inverse to the natural functor.
The localizing subcategory assertion follows since $\Gamma_0$ generates
$\D(\Gamma_0)$.

For (2), the two derived Hom functors send coproducts to products,
and the subcategories of objects supported in degree $-p$ are closed
under products. Thus (1) reduces the conditions on the duality functors
to testing $\Gamma_0$ on both sides.
\end{proof}

For an upper set $I$, the free shifts give a stable $t$-structure
$\per^G\Gamma=\per^{I^c}\Gamma\perp\per^I\Gamma$. A nonzero morphism
from $\Gamma(-g)$ to a cohomological shift of $\Gamma(-h)$ requires
$g\ge h$, which is impossible for $g\in I^c$ and $h\in I$.
The following is the analogue of \cite[A.8(1)]{Han24c}.

\begin{Lem}\label{perupper}
For any upper set $I\subseteq G$, we have $\per^I\Gamma=(\per^G\!\Gamma)_I$.
\end{Lem}
\begin{proof}
The support condition is equivalent to right orthogonality to
$\Gamma(-g)[n]$ for all $g\notin I$ and $n\in\mathbb Z$. The stable
$t$-structure above identifies this orthogonal with $\per^I\Gamma$.
\end{proof}

Next, we consider the condition $\Gamma_0\in\per^G\!\Gamma$. Remark that for any upper set $I\subseteq G$, we have a stable $t$-structure $\D^G(\Gamma)=\D^G(\Gamma)_I\perp\D^G(\Gamma)_{I^c}$.

\begin{Prop}\label{tstres}
For a $G_{\geq0}$-graded dg ring $\Gamma$, the following conditions are equivalent.
\begin{enumerate}
\item $\Gamma_0\in\per^G\!\Gamma$
\item $\Gamma_{g}\in\per^G\!\Gamma$ for every $g\in G_{\ge0}$.
\item For any upper set $I\subseteq G$, the stable $t$-structure $\D^G(\Gamma)=\D^G(\Gamma)_I\perp\D^G(\Gamma)_{I^c}$ restricts to $\per^G\!\Gamma$.
\end{enumerate}
If these conditions are satisfied, then we have $\Gamma_g\in\per\Gamma_0$ for every $g\in G_{\ge0}$.
\end{Prop}
\begin{proof}
Applying (3) to $I=G_{>0}$ and truncating $\Gamma$ gives (1).

Assume (1). The triangle $\Gamma_{>0}\to\Gamma\to\Gamma_0\dashrightarrow$
and Lemma \ref{perupper} give $\Gamma_{>0}\in\per^{G_{>0}}\Gamma$.
For $g>0$, taking the component of degree $g$ gives
\[
 \Gamma_g\in\thick_{\D(\Gamma_0)}
 \{\Gamma_{g-h}\mid h>0,\ g-h\ge0\}.
\]
Since $\pi(g-h)<\pi(g)$, induction on $\pi(g)$ gives
$\Gamma_g\in\per\Gamma_0$. Inflating along $\Gamma\to\Gamma_0$
and placing each component in its degree gives (2).

Assume (2). For a non-trivial upper set $I$, the quotient
$\Gamma_{I^c}$ has finite internal support. Its filtration by individual
degrees has factors $\Gamma_g$ in their respective degrees, so
$\Gamma_{I^c}$ is perfect. The same holds for the truncations of each
shift of $\Gamma$. Since the truncation functors are exact and preserve
direct summands, they preserve perfect objects. This gives (3), with
the cases $I=\emptyset,G$ immediate.
\end{proof}

For a $G_{\geq0}$-graded dg ring $\Gamma$, we consider the following conditions.

\begin{enumerate}
\item[(D1)] $\Gamma_0\in\per^G\!\Gamma$ and $\Gamma_0^{\op}\in\per^G\!\Gamma^{\op}$.
\item[(D2)]  $\Gamma$ has generalized Gorenstein parameter $p\in G_{>0}$.
\end{enumerate}

As \cite[A.6]{Han24c}, we define $G$-graded cluster categories.

\begin{Def}
Under the setting (D1) and (D2), we define the {\it graded cluster category} as the Verdier quotient
\[\C^G(\Gamma):=\per^G\!\Gamma/\thick^G\!\Gamma_0.\]
\end{Def}

Remark that $\mathbb{Z}$ acts on $G$ by $n\cdot g:=g+np$ and this action satisfies the conditions (A1),(A2) and (A3). Thus for $I\in\I_G$, we write $J(I)=I\cap(I^c+p)\in\J_G$ as in Theorem \ref{upJX}.

We state one of the main theorems in this section which is a generalization of \cite[A.7]{Han24c}.

\begin{Thm}\label{GBei}(Beilinson-type theorem for $G$-graded dg rings)
Assume the setting (D1) and (D2). Take $I\in\I_G$.
\begin{enumerate}
\item We have a weak semi-orthogonal decomposition
\[\per^G\!\Gamma=\thick^I\Gamma_0\perp\per^{J(I)}\Gamma\perp\thick^{I^c}\Gamma_0.\]
\item The composition
\[\per^{J(I)}\Gamma\hookrightarrow\per^G\!\Gamma\to\C^G(\Gamma)\]
is a triangle equivalence. Thus if we put $A:=[\Gamma_{h-g}]_{g,h\in J(I)}$, then we have a triangle equivalence
\[\C^G(\Gamma)\simeq\per A.\]
\end{enumerate}
\end{Thm}

As in the case of (non-commutative) projective geometry, we write $\O_\Gamma\in\C^G(\Gamma)$ for the image of $\Gamma\in\D^G(\Gamma)$ in $\C^G(\Gamma)$. Remark that Theorem \ref{GBei}(2) says that $\bigoplus_{g\in J(I)}\O_\Gamma(-g)$ is a thick generator of $\C^G(\Gamma)$.

Towards this theorem, we make some preparations.

\begin{Lem}\label{Gordual}
Under the setting (D1) and (D2),
\begin{enumerate}
\item $(\per^G\!\Gamma)_{g}=\thick\Gamma_0(-g)$
\item The duality $\RHom_\Gamma(-,\Gamma)\colon\per^G\!\Gamma\xrightarrow[\simeq]{}\per^G\!\Gamma^{\op}$ restricts to a duality $(\per^G\!\Gamma)_{g}\xrightarrow[\simeq]{}(\per^G\!\Gamma^{\op})_{-p-g}$.
\end{enumerate}
\end{Lem}
\begin{proof}
These assertions can be shown in the same way as \cite[A.9]{Han24c}.
\end{proof}

We recall powerful results from \cite{IYang20}.

\begin{Prop}\label{IYang}
Let $\T$ be a triangulated category and $\S\subseteq\T$ a thick subcategory. Assume we have a $t$-structure $\S=\X\perp\Y$. Consider the Verdier quotient $\pi\colon\T\to\T/\S$.
\begin{enumerate}
\item For $M\in{}^\perp\Y[1]\subseteq\T$ and $N\in\X^\perp\subseteq\T$, the map
\[\pi\colon\T(M,N)\to(\T/\S)(\pi(M),\pi(N))\]
is bijective.
\item\cite[1.3]{IYang20} Assume further that $\S=\X\perp\Y$ is a
stable $t$-structure and that $\T$ has stable $t$-structures
$\T=\X\perp\X^\perp={}^\perp\Y\perp\Y$. Put
$\Z:=\X^\perp\cap{}^\perp\Y$. Then
$\T=\X\perp\Z\perp\Y$ is a weak semi-orthogonal decomposition.
In particular, the composition
\[\Z\hookrightarrow\T\xrightarrow{\pi}\T/\S\]
is a triangle equivalence.
\end{enumerate}
\end{Prop}
\begin{proof}
For (1), the same proof as \cite[3.3]{IYang20} works.
\end{proof}

Now we can prove Theorem \ref{GBei}.

\begin{proof}[Proof of Theorem \ref{GBei}]
Internal truncation at $I$ splits $\thick^G\Gamma_0$ as
$\thick^I\Gamma_0\perp\thick^{I^c}\Gamma_0$. By Proposition
\ref{tstres}, it also restricts to the perfect category. The upper
part is $\per^I\Gamma$ by Lemma \ref{perupper}. The lower truncations of
free shifts have finite internal support with components perfect
over $\Gamma_0$, so the lower part is $\thick^{I^c}\Gamma_0$. Thus
\[
 \per^G\Gamma=\per^I\Gamma\perp\thick^{I^c}\Gamma_0.
\]
Dualizing the corresponding decomposition over $\Gamma^{\op}$ at
the upper set $-I^c-p$ gives, by Lemma \ref{Gordual},
\[
 \per^G\Gamma=\thick^I\Gamma_0\perp\per^{I^c+p}\Gamma.
\]
By Proposition \ref{IYang}(2), the quotient is equivalent to the
intersection $\per^I\Gamma\cap\per^{I^c+p}\Gamma$. The free shifts give
the decomposition of $\per^{I^c+p}\Gamma$ as
$\per^{I^c}\Gamma\perp\per^{J(I)}\Gamma$, so this intersection is
$\per^{J(I)}\Gamma$. This proves (1) and the equivalence in (2).
The $(g,h)$-entry of the dg endomorphism algebra of
$\bigoplus_{g\in J(I)}\Gamma(-g)$ is
$\operatorname{Hom}^G_\Gamma(\Gamma(-h),\Gamma(-g))=\Gamma_{h-g}$.
\end{proof}

\subsection{Minamoto--Mori correspondence}

We extend the Minamoto--Mori correspondence \cite[4.12]{MM}
to $G$-graded connective dg algebras. Fix an integer $d\ge1$ and a
$G_{\ge0}$-graded connective dg $k$-algebra $\Gamma$ with
$H^0\Gamma_0\ne0$. We may replace $\Gamma$ by its nonpositive
cohomological truncation, so $H^0\Gamma$-modules are viewed as dg
$\Gamma$-modules via $\Gamma\to H^0\Gamma$. Write
\[
 \pvd^G\Gamma=
 \left\{X\in\D^G(\Gamma)\ \middle|\ \sum_{i\in\mathbb Z}\dim_k H^iX<\infty\right\}.
\]
We consider the following conditions.
\begin{enumerate}
\item[(C1)] $\Gamma_g\in\pvd\Gamma_0$ holds for every $g\in G_{\ge0}$.
\item[(C2)] $\pvd^G\!\Gamma$ admits a Serre functor $(-p)[d+1]$.
\end{enumerate}

\begin{Rem}\label{lrsymm}
Since we have a duality $D\colon\pvd^G\!\Gamma\xrightarrow[\simeq]{}\pvd^G\!\Gamma^{\op}$, the condition (C2) is left-right symmetric.
\end{Rem}

\begin{Ex}
Assume a $G_{\geq0}$-graded connective dg algebra $\Gamma$ satisfies the following conditions.
\begin{enumerate}
\item $\Gamma_0$ is proper as a dg $k$-algebra, that is, we have $\sum_{i\in\mathbb{Z}}\dim_kH^i\Gamma_0<\infty$.
\item $\Gamma$ is homologically smooth, that is, $\Gamma\in\per^G\!\Gamma^e$ holds where $\Gamma^e:=\Gamma^{\op}\otimes_k\Gamma$.
\item $\Gamma$ is bimodule $(d+1)$-Calabi--Yau of Gorenstein parameter $p$, that is, $\RHom_{\Gamma^e}(\Gamma,\Gamma^e)\cong\Gamma(p)[-d-1]$ holds in $\D^G(\Gamma^e)$.
\end{enumerate}
Then $\Gamma$ satisfies (C1) and (C2). Indeed, smoothness makes the
proper module $\Gamma_0$ perfect over $\Gamma$, so Proposition
\ref{tstres} gives (C1). Bimodule Calabi--Yau duality restricts to the
Serre duality in (C2).
\end{Ex}

Consider the following thick subcategory $\T$ of $\D^G(\Gamma)$ which offers us a suitable place for arguments.
\[\T:=\{X\in\D^G(\Gamma)\mid X_{\le g}\in\pvd^G\!\Gamma\text{ holds for every }g\in G\}\]
First, we see that for each object in $\T^{\le0}:=\{X\in\T\mid H^{>0}X=0\}$, we can do an operation like taking a projective cover. A role of projective-like objects is played by objects in
\[\P:=\{\bigoplus_{g\in G}P_g(-g)\mid P_g\in\add\Gamma\subseteq\per^G\!\Gamma, P_{\ll0}=0\}\subseteq\T.\]

\begin{Lem}\label{grprojcover}
Assume a $G_{\geq0}$-graded connective dg algebra $\Gamma$ satisfies (C1). Take $X_0\in\T^{\le0}$. Then we have triangles
\[X_{i+1}\to P_i\to X_i\dashrightarrow(i\ge0)\]
satisfying the following conditions for each $i\ge0$.
\begin{enumerate}
\item $X_i\in\T^{\le0}$
\item The morphism $P_i\to X_i$ is a right $\P$-approximation.
\item If we take a simple module $S\in\simp^G\!H^0\Gamma$, then the homomorphism $\T(X_i,S)\to\T(P_i,S)$ is an isomorphism.
\end{enumerate}
\end{Lem}
\begin{proof}
It suffices to construct the first triangle. Put $A=H^0\Gamma$ and
$R=\rad A_0\oplus A_{>0}$. The graded module $H^0X_0$ is bounded
below in numerical degree and has finite-dimensional components.
Choose the graded projective cover of $H^0X_0$ by lifting the
projective covers of the components of its top $H^0X_0/(H^0X_0)R$
over $A_0$. The lifts generate $H^0X_0$ by graded Nakayama, using
the nilpotence of $\rad A_0$ and the positive numerical degrees of
$A_{>0}$. This cover lifts to a map $P_0\to X_0$ with $P_0\in\P$,
since morphisms from shifts of $\Gamma$ are computed by cohomology.
Surjectivity on $H^0$ makes it a right $\P$-approximation and gives
$X_1\in\T^{\le0}$. Maps from $X_0$ and $P_0$ to a simple $S$ are
computed on $H^0$, where the cover induces an isomorphism on tops.
\end{proof}

As an immediate application, we can verify that the condition (C1) ensures that $\pvd^G\!\Gamma$ is Hom-finite. 

\begin{Cor}
Assume a $G_{\geq0}$-graded connective dg algebra $\Gamma$ satisfies (C1). Then $\pvd^G\!\Gamma$ is Hom-finite.
\end{Cor}
\begin{proof}
Take $X,Y\in\pvd^G\!\Gamma$. Shifting both, we may assume
$X\in\T^{\le0}$. Apply Lemma \ref{grprojcover} to $X_0=X$ to obtain
\[X\in P_0*P_1[1]*\cdots*P_{i-1}[i-1]*X_i[i].\]
Here, since $Y\in\pvd^G\!\Gamma$, we have $\dim_k\T(P[n],Y)<\infty$ for each $P\in\P$ and $n\in\mathbb{Z}$. In addition, we have $\T(\T^{\le-i},Y)=0$ for $i\gg0$. Thus we get the assertion.
\end{proof}

As a next application, we prove $\Gamma_0\in\per^G\!\Gamma$ using the condition (C2).

\begin{Prop}\label{CY0thper}
Assume a $G_{\geq0}$-graded connective dg algebra $\Gamma$ satisfies (C1) and (C2). Then we have 
\[\Gamma_0\in\per^G\!\Gamma.\]
\end{Prop}
\begin{proof}
Apply Lemma \ref{grprojcover} to $X_0=\Gamma_0$. For a simple
$H^0\Gamma_0$-module $S$ in internal degree zero, dimension shifting
and (C2) give
\[
 \D^G(\Gamma)(P_n,S(g))
 \cong \D^G(\Gamma)(\Gamma_0,S(g)[n])
 \cong D\D^G(\Gamma)(S(g),\Gamma_0(-p)[d+1-n]).
\]
This vanishes for $n\gg0$ by boundedness of $H^*\Gamma_0$ and
cohomological orthogonality. It also vanishes for $\pi(-g)>\pi(p)$
by internal truncation orthogonality, since $S(g)$ is then in an upper
set of degrees disjoint from the support of $\Gamma_0(-p)$.
Thus each $P_n$ is a finite sum of projective shifts, and $P_n=0$
for all $n\ge N$ for some $N$. The remaining triangles give
$X_N\cong X_{N+r}[r]$ for every $r\ge0$. Since each $X_{N+r}$ has
cohomology in degrees at most zero, $X_N=0$. The finite tower up to
$P_{N-1}$ now shows that $\Gamma_0$ is perfect.
\end{proof}

Next, we see that $\gl\Gamma_0\le d+1<\infty$ holds thanks to the condition (C2).

\begin{Lem}\label{properperfect}
Assume a $G_{\geq0}$-graded connective dg algebra $\Gamma$ satisfies (C1) and (C2). Then we have the following assertions.
\begin{enumerate}
\item $\pvd\Gamma_0=\per\Gamma_0$
\item $\pvd^G\!\Gamma=\thick^G\!\Gamma_0\subseteq\per^G\!\Gamma$
\end{enumerate}
\end{Lem}
\begin{proof}
(1) Take $X,Y\in\mod H^0\Gamma_0\subseteq\pvd\Gamma_0$. By (C2), we have
\[\D(\Gamma_0)(X,Y[n])=\D^G(\Gamma)(X,Y[n])\cong D\D^G(\Gamma)(Y[n],X(-p)[d+1]).\]
Thus $\D(\Gamma_0)(X,Y[>\!d+1])=0$ holds. This implies $\gl\Gamma_0\le d+1<\infty$ and $\pvd\Gamma_0=\per\Gamma_0$ by \cite[3.3,3.6]{Tom25c}.

(2) This follows by (1) and Proposition \ref{CY0thper}.
\end{proof}

Our next objective is to prove the relative Calabi--Yau property (Proposition \ref{relCY}). The next lemma holds for arbitrary $G_{\ge0}$-graded dg algebras.

\begin{Lem}
Let $\Gamma$ be a $G_{\geq0}$-graded dg algebra. Take $M\in\D^G(\Gamma)$.
\begin{enumerate}
\item For $P\in\per^G\!\Gamma$, the natural morphism
\[\D^G(\Gamma)(P,M)\to\varprojlim_{g\in G}\D^G(\Gamma)(P,M_{\leq g})\]
is an isomorphism.
\item For $X\in\pvd^G\!\Gamma$, the natural morphism
\[\varinjlim_{g\in G}\D^G(\Gamma)(M_{\leq g},X)\to\D^G(\Gamma)(M,X)\]
is an isomorphism.
\end{enumerate}
\end{Lem}
\begin{proof}
For (1), choose $g$ sufficiently large that all the internal degrees
of the finitely many free shifts used to build $P$ lie in
$\{h\mid h\le g\}$. Then $\D^G(\Gamma)(P,M_{\nleq g}[n])=0$ for
all $n$, so $\D^G(\Gamma)(P,M)\to\D^G(\Gamma)(P,M_{\le g})$ is an
isomorphism for all sufficiently large $g$.

For (2), choose $g$ large enough that the internal cohomology support
of $X$ is contained in $\{h\mid h\le g\}$. Internal truncation
orthogonality gives $\D^G(\Gamma)(M_{\nleq g},X[n])=0$ for every $n$.
Thus $\D^G(\Gamma)(M_{\le g},X)\to\D^G(\Gamma)(M,X)$ is an isomorphism
for all sufficiently large $g$.
\end{proof}

Then we can prove the relative Calabi--Yau property. This is an analogue of \cite[3.1(8)]{IR}.

\begin{Prop}\label{relCY}
Assume a $G_{\geq0}$-graded connective dg algebra $\Gamma$ satisfies the conditions (C1) and (C2). Then for $X\in\pvd^G\!\Gamma$ and $Y\in\T$, we have a functorial isomorphism
\[\D^G(\Gamma)(X,Y)\cong D\D^G(\Gamma)(Y,X(-p)[d+1]).\]
\end{Prop}
\begin{proof}
For $g\in G$, we have an isomorphism
\[\D^G(\Gamma)(X,Y_{\le g})\cong D\D^G(\Gamma)(Y_{\le g},X(-p)[d+1]).\]
By taking a limit, we obtain
\begin{align*}
\D^G(\Gamma)(X,Y)&\cong\varprojlim_{g\in G}\D^G(\Gamma)(X,Y_{\le g}) \\
&\cong\varprojlim_{g\in G}D\D^G(\Gamma)(Y_{\le g},X(-p)[d+1]) \\
&\cong D\varinjlim_{g\in G}\D^G(\Gamma)(Y_{\le g},X(-p)[d+1]) \\
&\cong D\D^G(\Gamma)(Y,X(-p)[d+1]).\qedhere
\end{align*}
\end{proof}

The next two corollaries verify the hypotheses of the Beilinson-type
theorem.

\begin{Cor}
Assume $\Gamma$ satisfies (C1) and (C2). Then it satisfies (D1) and (D2).
\end{Cor}
\begin{proof}
Proposition \ref{CY0thper} and Remark \ref{lrsymm} give (D1).
Proposition \ref{relCY} gives
\[
 H^i\bigl(\RHom_\Gamma(\Gamma_0,\Gamma)\bigr)_g
 \cong \D^G(\Gamma)(\Gamma_0,\Gamma(g)[i])
 \cong D H^{d+1-i}(\Gamma_0)_{-g-p}.
\]
Thus the graded dual of $\Gamma_0$ is supported in internal degree
$-p$, and the same holds over $\Gamma^{\op}$. Lemma
\ref{degsuppzero} and Corollary \ref{posgorparam} give (D2).
\end{proof}

\begin{Cor}\label{posgorparam}
Assume $\Gamma$ satisfies (C1) and (C2). Then $p>0$.
\end{Cor}
\begin{proof}
The identity of $\Gamma_0$ and Serre duality give
$\D^G(\Gamma)(\Gamma_0,\Gamma_0(-p)[d+1])\ne0$. If $p\ngeq0$, this
contradicts internal truncation orthogonality at $G_{\ge0}$. If $p=0$,
Lemma \ref{degsuppzero} identifies this Hom space with
$H^{d+1}\Gamma_0=0$, since $\Gamma$ is connective. Hence $p>0$.
\end{proof}

With these preparations, we can state the following theorem. Recall that under the conditions (C1) and (C2), we define $\C^G(\Gamma)=\per^G\!\Gamma/\pvd^G\!\Gamma$.

\begin{Thm}\label{MMcorr}(Minamoto--Mori correspondence for $G$-graded Calabi--Yau dg algebras)
Let $\Gamma$ be a $G_{\geq0}$-graded connective dg algebra satisfying (C1) and (C2). Take a non-trivial upper set $I\subseteq G$ and put $\E:=\bigoplus_{g\in J(I)}\O_\Gamma(-g)\in\C^G(\Gamma)$.
\begin{enumerate}
\item The auto-equivalence $(-p)[d]$ of $\C^G(\Gamma)$ is a Serre functor.
\item $\E\in\C^G(\Gamma)$ is a $d$-silting object.
\end{enumerate}
Moreover, we assume $H^{<0}\Gamma=0$.
\begin{enumerate}
\setcounter{enumi}{2}
\item $\E\in\C^G(\Gamma)$ is a $d$-tilting object and $\End_{\C^G(\Gamma)}(\E)\cong[H^0\Gamma_{h-g}]_{g,h\in J(I)}$ is a $d$-representation infinite algebra.
\end{enumerate}
\end{Thm}
\begin{proof}
(1) Put $\S=\pvd^G\Gamma$. Theorem \ref{GBei} gives Hom-finiteness
of $\C^G(\Gamma)$, since the window is finite and each $\Gamma_g$
is proper by (C1). For $X,Y\in\per^G\Gamma$, choose $g$ sufficiently
large that $\D^G(\Gamma)(X,Y_{\ge g})=0$, using the finitely many
free shifts that build $X$. Then $Y_{\ngeq g}\in\S$ by (C1),
first for free shifts and hence for all perfect $Y$.
Thus $Y\to Y_{\ngeq g}$ is a local $\S$-envelope of $Y$ relative to
$X$. Maps from a perfect object to $\S$ are finite dimensional, as
are the reverse maps by Proposition \ref{relCY}. The same proposition
gives the relative Serre duality for $\S$ in $\per^G\Gamma$. By
\cite[1.3,1.4]{Ami09}, it induces Serre duality on the quotient with
Serre functor $(-p)[d]$.

(2) By Theorem \ref{GBei} and connectivity, $\E$ is silting.
For $a,b\in G$ and $i<d$, put $M=\Gamma(a)$ and $N=\Gamma(b)[i]$.
For the cohomological $t$-structure
$\S=\S^{\le1}\perp\S^{\ge2}$, we have
$\D^G(\Gamma)(M,\S^{\ge2}[1])=0$. Also
$\D^G(\Gamma)(\S^{\le1},N)=0$ by Proposition \ref{relCY}, since
$d+1-i\ge2$. Proposition \ref{IYang}(1) therefore gives
\begin{equation}\label{windowhom}
 \C^G(\Gamma)(\O_\Gamma(a),\O_\Gamma(b)[i])
 \cong H^i\Gamma_{b-a}\qquad (i<d).
\end{equation}
To prove $d$-silting, it suffices by \cite[4.1]{HI} to show
$\C^G(\Gamma)(\E,\E(p)[>0])=0$, since $\nu_d^{-1}=(p)$ by (1).
For degrees $0<i<d$ this follows from \eqref{windowhom} and
connectivity. For $g,h\in J(I)$ and $j\ge0$, Serre duality and
\eqref{windowhom} give
\[
 \C^G(\Gamma)(\O_\Gamma(-g),\O_\Gamma(-h+p)[d+j])
 \cong D H^{-j}\Gamma_{h-g-2p}=0 .
\]
Indeed, $h-p\notin I$ and $g+p\in I$, so $h-g-2p\ngeq0$.

(3) Now $\E$ is tilting by \eqref{windowhom}. Its endomorphism algebra
$A=[H^0\Gamma_{h-g}]_{g,h\in J(I)}$ is finite dimensional and
triangular, with diagonal $H^0\Gamma_0$ of finite global dimension
by Lemma \ref{properperfect} and its proof. Hence $A$ has finite global dimension.
Under the tilting equivalence, $\nu_d^{-1}A$ corresponds to $\E(p)$.
Part (2) and \eqref{windowhom} give that
$\nu_d^{-1}A=\RHom_A(DA,A)[d]$ is concentrated in degree zero, so
$\gl A=\pd_A DA\le d$. Thus $\nu_d^{-1}$ preserves nonpositive
cohomological degrees, and \eqref{windowhom} gives vanishing in
negative degrees for $\nu_d^{-n}A$ for each $n\ge0$. This proves that
$A$ is $d$-representation infinite.
\end{proof}

\end{appendix}

\section*{Use of AI}
The author used ChatGPT-5.5 Pro and ChatGPT-5.5 Thinking as auxiliary tools for discussing possible approaches to specific arguments in Section~\ref{doubleAtilde}, including Propositions \ref{pathindep} and \ref{conversecut}, for language polishing and for improving the exposition of preliminary drafts.
The author also used OpenAI Codex to revise the exposition and proofs
in Section~\ref{doubleAtilde}, to rewrite the proofs of
Theorem~\ref{classfitiltrk2} and Lemma~\ref{invflvanq}, and to check
and revise arguments elsewhere in the manuscript, including
Appendix~\ref{appenBei}. Codex was also used to verify and update
citations and bibliographic entries.
All mathematical arguments were independently verified and finalized by the author.
The author takes full responsibility for the accuracy, originality, and integrity of the paper.

\bibliographystyle{alpha} 
\bibliography{reference}

\newcommand{\etalchar}[1]{$^{#1}$}
\begin{thebibliography}{BMR{\etalchar{+}}06}

\bibitem[AIR15]{AIR15}
Claire Amiot, Osamu Iyama, and Idun Reiten.
\newblock Stable categories of {C}ohen-{M}acaulay modules and cluster
  categories: {D}edicated to {R}agnar-{O}laf {B}uchweitz on the occasion of his
  sixtieth birthday.
\newblock {\em American Journal of Mathematics}, 137(3):813--857, 2015.

\bibitem[Ami09]{Ami09}
Claire Amiot.
\newblock Cluster categories for algebras of global dimension 2 and quivers
  with potential.
\newblock {\em Annales de l'institut Fourier}, 59:2525--2590, 2009.

\bibitem[BCS05]{BCS}
Lev Borisov, Linda Chen, and Gregory Smith.
\newblock The orbifold {C}how ring of toric {D}eligne-{M}umford stacks.
\newblock {\em Journal of the American Mathematical Society}, 18(1):193--215,
  2005.

\bibitem[Bei78]{Bei}
Aleksandr~Aleksandrovich Beilinson.
\newblock Coherent sheaves on ${P}^n$ and problems of linear algebra.
\newblock {\em Funktsional. Anal. i Prilozhen}, 12(3):68--69, 1978.

\bibitem[BH]{BuH}
Ragnar-Olaf Buchweitz and Lutz Hille.
\newblock Endomorphism rings of geometric tilting objects.
\newblock in preparation.

\bibitem[BH09]{BH}
Lev Borisov and Zheng Hua.
\newblock On the conjecture of {K}ing for smooth toric {D}eligne--{M}umford
  stacks.
\newblock {\em Advances in Mathematics}, 221(1):277--301, 2009.

\bibitem[BHI]{BHI}
Ragnar-Olaf Buchweitz, Lutz Hille, and Osamu Iyama.
\newblock Cluster tilting for projective varieties and {H}orrocks type theorem.
\newblock in preparation.

\bibitem[BMR{\etalchar{+}}06]{BMRRT}
Aslak~Bakke Buan, Bethany Marsh, Markus Reineke, Idun Reiten, and Gordana
  Todorov.
\newblock Tilting theory and cluster combinatorics.
\newblock {\em Advances in mathematics}, 204(2):572--618, 2006.

\bibitem[DG24]{DG}
Darius Dramburg and Oleksandra Gasanova.
\newblock A classification of $n$-representation infinite algebras of type
  \~{A}.
\newblock arXiv:2409.06553v3, 2024.

\bibitem[Efi14]{Efi}
Alexander~I Efimov.
\newblock Maximal lengths of exceptional collections of line bundles.
\newblock {\em Journal of the London Mathematical Society}, 90(2):350--372,
  2014.

\bibitem[Gin06]{Gin06}
Victor Ginzburg.
\newblock Calabi--{Y}au algebras.
\newblock arXiv:0612139, 2006.

\bibitem[GL87]{GL87}
Werner Geigle and Helmut Lenzing.
\newblock A class of weighted projective curves arising in representation
  theory of finite-dimensional algebras.
\newblock In {\em Singularities, representation of algebras, and vector bundles
  ({L}ambrecht, 1985)}, volume 1273 of {\em Lecture Notes in Math.}, pages
  265--297. Springer, Berlin, 1987.

\bibitem[Han24a]{Han24c}
Norihiro Hanihara.
\newblock Calabi--{Y}au completions for roots of dualizing dg bimodules.
\newblock arXiv:2412.18753, 2024.

\bibitem[Han24b]{Han24a}
Norihiro Hanihara.
\newblock Higher hereditary algebras and {C}alabi--{Y}au algebras arising from
  some toric singularities.
\newblock arXiv:2412.19040, 2024.

\bibitem[Han25]{Han25}
Norihiro Hanihara.
\newblock Non-commutative resolutions for {S}egre products and
  {C}ohen-{M}acaulay rings of hereditary representation type.
\newblock {\em Transactions of the American Mathematical Society},
  378(04):2429--2475, 2025.

\bibitem[HI25]{HI}
Norihiro Hanihara and Osamu Iyama.
\newblock Silting correspondences and {C}alabi--{Y}au dg algebras.
\newblock arXiv:2508.12836, 2025.

\bibitem[HIMO23]{HIMO}
Martin Herschend, Osamu Iyama, Hiroyuki Minamoto, and Steffen Oppermann.
\newblock {\em Representation theory of {G}eigle-{L}enzing complete
  intersections}, volume 285.
\newblock American Mathematical Society, 2023.

\bibitem[HIO14]{HIO}
Martin Herschend, Osamu Iyama, and Steffen Oppermann.
\newblock n--{R}epresentation infinite algebras.
\newblock {\em Advances in mathematics}, 252(2):292--342, 2014.

\bibitem[HP06]{HP06}
Lutz Hille and Markus Perling.
\newblock A counterexample to {K}ing's conjecture.
\newblock {\em Compositio Mathematica}, 142(6):1507--1521, 2006.

\bibitem[HP14]{HP14}
Lutz Hille and Markus Perling.
\newblock Tilting bundles on rational surfaces and quasi-hereditary algebras.
\newblock {\em Annales de l'Institut Fourier}, 64:625--644, 2014.

\bibitem[HP26]{HaPr26}
Jack Hall and Kyle Priver.
\newblock A generalized {Bondal}--{Orlov} full faithfulness criterion for
  {Deligne}--{Mumford} stacks.
\newblock {\em Selecta Mathematica. New Series}, 32:Paper No. 51, 2026.

\bibitem[IO11]{IO11}
Osamu Iyama and Steffen Oppermann.
\newblock $n$--representation--finite algebras and $n$--{APR} tilting.
\newblock {\em Transactions of the American Mathematical Society},
  363(12):6575--6614, 2011.

\bibitem[IR08]{IR}
Osamu Iyama and Idun Reiten.
\newblock Fomin--{Z}elevinsky mutation and tilting modules over {C}alabi--{Y}au
  algebras.
\newblock {\em American Journal of Mathematics}, 130(4):1087--1149, 2008.

\bibitem[IU09]{IU09}
Akira Ishii and Kazushi Ueda.
\newblock Dimer models and exceptional collections.
\newblock arXiv:0911.4529, 2009.

\bibitem[IY20]{IYang20}
Osamu Iyama and Dong Yang.
\newblock Quotients of triangulated categories and equivalences of {B}uchweitz,
  {O}rlov, and {A}miot--{G}uo--{K}eller.
\newblock {\em American Journal of Mathematics}, 142(5):1641--1659, 2020.

\bibitem[Iya07a]{Iya07b}
Osamu Iyama.
\newblock Auslander correspondence.
\newblock {\em Advances in mathematics}, 210(1):51--82, 2007.

\bibitem[Iya07b]{Iya07a}
Osamu Iyama.
\newblock Higher-dimensional {A}uslander-{R}eiten theory on maximal orthogonal
  subcategories.
\newblock {\em Advances in mathematics}, 210(1):22--50, 2007.

\bibitem[Iya11]{Iya11}
Osamu Iyama.
\newblock Cluster tilting for higher {A}uslander algebras.
\newblock {\em Advances in mathematics}, 226(1):1--61, 2011.

\bibitem[Kap88]{Kap}
Mikhail~M. Kapranov.
\newblock On the derived categories of coherent sheaves on some homogeneous
  spaces.
\newblock {\em Inventiones mathematicae}, 92(3):479--508, 1988.

\bibitem[Kaw06]{Kaw06}
Yujiro Kawamata.
\newblock Derived categories of toric varieties.
\newblock {\em Michigan Mathematical Journal}, 54(3):517--536, 2006.

\bibitem[Kel11]{Kel11}
Bernhard Keller.
\newblock Deformed {C}alabi--{Y}au completions.
\newblock {\em Journal f\"{u}r die reine und angewandte Mathematik},
  2011(654):125--180, 2011.

\bibitem[Kin97]{Kin97}
Alastair King.
\newblock Tilting bundles on some rational surfaces.
\newblock preprint at \url{https://people.bath.ac.uk/masadk/papers/}, 1997.

\bibitem[MM11]{MM}
Hiroyuki Minamoto and Izuru Mori.
\newblock The structure of {AS}-{G}orenstein algebras.
\newblock {\em Advances in Mathematics}, 226(5):4061--4095, 2011.

\bibitem[Nak22]{Nak22}
Yusuke Nakajima.
\newblock On 2-representation infinite algebras arising from dimer models.
\newblock {\em The Quarterly Journal of Mathematics}, 73(4):1517--1553, 2022.

\bibitem[Sat13]{Sat13}
Matthew Satriano.
\newblock Canonical {Artin} stacks over log smooth schemes.
\newblock {\em Mathematische Zeitschrift}, 274(3--4):779--804, 2013.

\bibitem[Tom25a]{Tom25d}
Ryu Tomonaga.
\newblock Non-commutative crepant resolutions of toric singularities with
  divisor class group of rank one.
\newblock arXiv:2510.26252, 2025.

\bibitem[Tom25b]{Tom25c}
Ryu Tomonaga.
\newblock On silting mutations preserving global dimension.
\newblock arXiv:2510.26206v2, 2025.

\bibitem[Tom25c]{Tom25a}
Ryu Tomonaga.
\newblock Weak del {P}ezzo surfaces are characterized by the existence of
  $2$-tilting bundles.
\newblock arXiv:2510.26199, 2025.

\bibitem[VdB04]{VdB04a}
Michel Van~den Bergh.
\newblock Non-commutative crepant resolutions.
\newblock {\em The Legacy of Niels Henrik Abel: The Abel Bicentennial, Oslo,
  2002. Berlin, Heidelberg: Springer Berlin Heidelberg}, pages 749--770, 2004.

\bibitem[VZB22]{VZB22}
John Voight and David Zureick-Brown.
\newblock The canonical ring of a stacky curve.
\newblock {\em Memoirs of the American Mathematical Society}, 277(1362):v+144,
  2022.

\end{thebibliography}

\end{document}